\documentclass{article}

\usepackage[margin = 1in]{geometry}

\usepackage{titlesec} % for redefining section headings

\usepackage{tikz}
\usepackage{tikz-cd}
\usetikzlibrary{cd, arrows.meta, calc}
\usepackage{url, hyperref}
\usepackage{float}

\tikzset{->-/.style={decoration={
  markings,
  mark=at position .45 with {\arrow{>}}},postaction={decorate}}}
\usepackage{amsmath}
\usepackage{stmaryrd}
\usepackage{amssymb}
\usepackage{enumitem}
\usepackage{graphicx}
\usepackage{mathdots}
\usepackage{color}
\usepackage{diagbox}
\usepackage{array, makecell}
\usepackage{rotating}
\usepackage{amsthm}
\usepackage{dsfont}
\usepackage{adjustbox}
\usetikzlibrary{arrows}
\usepackage{mathtools}

\def\N{{\mathbb{N}}}

\def\Ev{{\mathsf{Ev}}}
\def\cM{{\overline{\mathcal{M}}}}

\def\mC{\mathsf{C}}

\def\cV{\mathcal{V}}
\def\cW{\mathcal{W}}

\def\cO{\mathcal{O}}
\def\oM{\overline{\mathcal{M}}}
\def\cM{{\mathcal{M}}}

\def\Z{\mathbb{Z}}
\def\C{\mathbb{C}}
\def\R{\mathbb{R}}

\def\b1{{\bf 1}}

\def\Aut{{\rm Aut}}

\def\WG{\mathcal{W}_{\Gamma}}

\def\P{\mathbb{P}}

\newcount\colveccount
\newcommand*\colvec[1]{
        \global\colveccount#1
        \begin{pmatrix}
        \colvecnext
}
\def\colvecnext#1{
        #1
        \global\advance\colveccount-1
        \ifnum\colveccount>0
                \\
                \expandafter\colvecnext
        \else
                \end{pmatrix}
        \fi
}

\usepackage{amsthm}
\usepackage{esvect}
\newtheorem{definition}{Definition}[section]
\newtheorem{theorem}[definition]{Theorem}
\newtheorem{example}[definition]{Example}
\newtheorem{proposition}[definition]{Proposition}
\newtheorem{corollary}[definition]{Corollary}
\newtheorem{lemma}[definition]{Lemma}
\newtheorem{remark}[definition]{Remark}

\newtheorem{notation}[definition]{Notation}

\newtheorem{construction}[definition]{Construction}

\newtheorem{observation}[definition]{Observation}

\newtheorem{innercustomthm}{Theorem}
\newenvironment{customthm}[1]
{\renewcommand\theinnercustomthm{#1}\innercustomthm}
{\endinnercustomthm}

\usepackage{amsmath,calligra,mathrsfs}
\DeclareMathOperator{\cHom}{\mathscr{H}\text{\kern -3pt {\calligra\large om}}\,}

\def\log{\mathrm{log}}

\usepackage{blkarray}
\usepackage{booktabs}
\usepackage{verbatim}
\usepackage{subcaption}

\newcommand{\tropC}{\mathsf{C}}
\newcommand{\defic}[1]{\operatorname{def}(#1)}
\newcommand{\ov}[1]{\operatorname{ov}(#1)}
\newcommand{\val}[1]{\operatorname{val}(#1)}
\newcommand{\lm}[1]{\operatorname{loop}(#1)}
\newcommand{\trop}{\operatorname{trop}}

\newcommand{\gr}{\mathrm{gr}}
\newcommand{\sat}{\mathrm{sat}}
\newcommand{\rad}{\mathrm{rad}}
\newcommand{\tors}{\mathrm{tors}}
\newcommand{\fine}{\mathsf{fine}}
\newcommand{\W}{\mathsf{W}}

\newcommand{\Id}{\operatorname{Id}}

\newcommand{\nbhdtropC}{\mathring{\tropC}_0}

\newcommand{\Span}{\operatorname{Span}_{\mathbb{R}}}
\newcommand{\nn}{n^{\operatorname{Aut}}}

\definecolor{lightgreen}{HTML}{E8F9E2}
\definecolor{lightblue}{HTML}{ADD8E6}
\definecolor{lightcyan}{HTML}{E0FFFF}
\definecolor{lightred}{HTML}{ffcccb}
\definecolor{lightpink}{HTML}{ffe6e6}
\definecolor{coralpink}{HTML}{F88379}
\definecolor{burntsiena}{HTML}{E97451}
\definecolor{forestgreen}{HTML}{228B22}
\definecolor{skyblue}{HTML}{448EE4}
\definecolor{darkgreen}{HTML}{006400}
\definecolor{darkblue}{HTML}{00008B}
\definecolor{darkorchid}{HTML}{9932CC}

\title{Genus one correspondence between tropical and algebraic curves}

\author{Alessio Cela and Sae Koyama}

\date{ }

\usepackage{graphicx}

\begin{document}

\maketitle

\begin{abstract}

We show that the genuinely enumerative count of algebraic elliptic curves in any toric variety agrees with the count of the corresponding well-spaced tropical curves, weighted by explicit combinatorial multiplicities. This provides a complete genus-$1$ generalization of the celebrated Nishinou--Siebert correspondence theorem in genus $0$. The proof is algebro-geometric and relies on logarithmic deformation theory together with an explicit enumeration of logarithmic maps with fixed tropicalization.

\end{abstract}
\tableofcontents

\section{Introduction}

Given a curve in a toric variety, tropical geometry associates to it a graph embedded in a real vector space. The reverse problem of determining which tropical curves arise from algebraic curves is difficult. In genus $0$, the balancing condition is the only requirement. In positive genus, the situation is considerably more subtle and remains almost completely open. The genus $1$ case was first studied in 2005 by Speyer~\cite{Speyer-PhD, Speyer}, who established a sufficient condition (and some necessary conditions) for a tropical curve to admit an algebraic lift. At that time, he highlighted significant challenges in applying his results to enumerative geometry, suggesting that, due to its qualitative rather than quantitative nature, his framework was better suited for existence results and lower bounds rather than for deriving exact combinatorial counts. This paper directly addresses this challenge and enumerates algebraic genus $1$ curves in toric varieties of \emph{arbitrary} dimension through a purely tropical approach. Our main result is the following.

Let $X$ be a smooth projective toric variety over $\C$. Let $N_X$ be the cocharacter lattice of $X$, and set $(N_X)_\R := N_X \otimes_{\Z} \R$. Let $\beta \in H_2(X,\Z)$ be an effective curve class, and let $Z_1,\ldots,Z_n$ be general subvarieties of $X$ (in the appropriate parameter spaces). Assume that there are only finitely many maps from smooth genus one curves (possibly with a fixed general $j$-invariant) to $X$ of class $\beta$ passing through $Z_1,\ldots,Z_n$ . Denote their number by
$
\mathsf{N}^{X,\mathrm{alg}}_{1,\beta,Z_1,\ldots,Z_n}
$ (resp. by $
\mathsf{N}^{X,\mathrm{alg}}_{1;\beta;Z_1,\ldots,Z_n;j}$ if the complex structure of the curve is fixed).

Similarly, let $\mathsf{N}^{X,\mathrm{trop}}_{1;\beta;Z_1,\ldots,Z_n}$ (resp. $\mathsf{N}^{X,\mathrm{trop}}_{1;\beta;Z_1,\ldots,Z_n;j}$) denote the count (with multiplicities) of genus one well-spaced tropical curves (possibly with prescribed general cycle length $j$) in $(N_X)_\R$, of the corresponding tropical degree, passing through the tropicalizations of the $Z_i$. The multiplicity of each tropical curve is made explicit in the main body of the paper (see Theorem \ref{thm: calculation of weights}) and can be computed effectively.

\begin{customthm}{A}\label{thm: main}
The algebraic count
$\mathsf{N}^{X,\mathrm{alg}}_{1,\beta,Z_1,\ldots,Z_n}$
(resp. $\mathsf{N}^{X,\mathrm{alg}}_{1;\beta;Z_1,\ldots,Z_n;j}$)
and the tropical count
$\mathsf{N}^{X,\mathrm{trop}}_{1;\beta;Z_1,\ldots,Z_n}$
(resp. $\mathsf{N}^{X,\mathrm{trop}}_{1;\beta;Z_1,\ldots,Z_n;j}$)
coincide; that is,
\[
\mathsf{N}^{X,\mathrm{alg}}_{1,\beta,Z_1,\ldots,Z_n}
=
\mathsf{N}^{X,\mathrm{trop}}_{1;\beta;Z_1,\ldots,Z_n}
\quad \text{and} \quad
\mathsf{N}^{X,\mathrm{alg}}_{1;\beta,Z_1,\ldots,Z_n;j}
=
\mathsf{N}^{X,\mathrm{trop}}_{1;\beta,Z_1,\ldots,Z_n;j}.
\]
\end{customthm}

When $X$ is a toric surface, dimension-counting arguments can be used to show that the pathological tropical curves do not contribute to the enumeration. This observation was fundamental for Mikhalkin's breakthrough \cite{Mik05}, which established that the count of algebraic curves in toric surfaces matches the count of tropical curves when weighted by appropriate combinatorial multiplicities. In arbitrary dimensions, however, the realizability problem becomes significantly more complex. While a definitive and elegant correspondence theorem was achieved in genus 0 by Nishinou–Siebert \cite{NS06}, the higher-genus setting has remained open for over two decades due to the failure of these low-dimensional dimension-counting arguments.

In 2019, Ranganathan--Santos-Parker--Wise~\cite{RSWI, RSWII} (see also \cite{LNR} for a different approach) provided a necessary and sufficient condition for a genus $1$ tropical curve to admit an algebraic lift to a map from a smooth elliptic curve: there must exist a logarithmic map from a possibly reducible algebraic curve of the same combinatorial type as the given tropical curve, and the tropical curve must be \emph{well-spaced}. This resolves the combinatorial aspect of the problem, but not the geometric one. Indeed, the first condition is non-trivial to verify in practice. We establish in Theorem \ref{thm: realizability genus 1} a criterion that allows one to check this condition. We then use this to obtain a complete description of the genus $1$ theory and explicitly enumerate the number of algebraic lifts of each genus $1$ tropical curve, thereby proving Theorem~\ref{thm: main}.

It should be highlighted that the invariants of interest in this paper are genuinely enumerative, as opposed to virtual ones. In the latter framework, Abramovich–Chen–Gross–Siebert \cite{ACGSI} showed that logarithmic Gromov–Witten invariants can be expressed as weighted counts of certain tropical curves. A key distinction from our work is that, in this virtual setting, the weights assigned to tropical curves are given in terms of Gromov–Witten invariants, rather than being described purely by combinatorial data.

Our methodology diverges from that of previous works. Mikhalkin’s original proof of the correspondence is not algebro-geometric, and the algebraic treatments in \cite{NS06} and \cite{ACGSI} rely on degeneration techniques. This paper takes a different path by employing the theory of Minkowski weights \cite{FuSt}. This framework is particularly well-suited to our purposes, as it inherently incorporates the principle that intersections of subvarieties in a toric variety can be understood tropically by translating their tropicalizations by a generic displacement vector.

\begin{table}[htpb]
    \centering
    \renewcommand{\arraystretch}{1.3}
    \begin{tabular}{@{}llcl@{}}
        \toprule
        \textbf{Authors} & \textbf{Dimension} & \textbf{Genus} & \textbf{Count Type} \\
        \midrule
        Mikhalkin \cite{Mik05} & $\dim X = 2$ & All & Enumerative \\
        Nishinou--Siebert \cite{NS06} & Arbitrary & $g=0$ & Enumerative \\
        Abramovich-Chen-Gross-Siebert \cite{ACGSI} & Arbitrary & All & Virtual \\
        \midrule
        \textbf{Our Work (Thm.~\ref{thm: main} or more precisely Thms.~\ref{thm: tropical correspondence} and \ref{thm: calculation of weights})} & \textbf{Arbitrary} & $\mathbf{g=1}$ & \textbf{Enumerative} \\
        \bottomrule
    \end{tabular}
    \vspace{0.2cm}
    \caption{Relation with previous tropical correspondences.}
    \label{tab:historical_context}
\end{table}

\subsection{Well-spaced genus $1$ maps}

Let $X$ be non-singular, irreducible, projective toric variety over $\C$ of dimension $r$. Let $\Sigma(X)$ be its fan. Ranganathan--Parker--Wise~\cite{RSWI, RSWII} introduced the moduli space $\mathcal{W}_{\Gamma}(X)$ of genus $1$ well-spaced stable maps to $X$ with prescribed contact orders along the toric boundary $\partial X$ 
and $n$ additional free markings. 

The curve class $A \in H_2(X)$ is determined by the contact orders, and the domain curve has $n+m$ markings: $m$ markings $z_j$ with prescribed contact to $\partial X$ and $n$ free markings $p_i$ with contact order $0$. The contact order with $\partial X$ at the marking $z_j$ is encoded by a vector $\delta_j \in N_X$ and the discrete data of the contact orders and free markings is encoded by $\Gamma$. 

\begin{example}
    When $X=\mathbb{P}^r$ is projective space, the most general tangency condition for degree $d$ curves corresponds to the case in which $\delta$ consists of $d$ copies of the primitive generator of each ray of $\Sigma(\mathbb{P}^r)$. 
\end{example}

The moduli stack $\mathcal{W}_{\Gamma}(X)$ is proper, logarithmically smooth and pure of dimension 
\[
\dim \mathcal{W}_{\Gamma}(X) = n +  m +(r-\operatorname{dim}(\operatorname{Span}_\R(\delta_1,\ldots,\delta_m))).
\]

Our primary result is a correspondence theorem connecting these algebraic curve counts with counts of tropical curves. As the next proposition shows, the general element of every component of $\cW_\Gamma(X)$ has a trivial stabilizer. Consequently, the corresponding count is an integer.

\begin{proposition}\label{prop: integrality}
Assume $n \geq 1$. Then the generic point of every irreducible component of $\mathcal{W}_\Gamma(X)$ has a trivial automorphism group.
\end{proposition}

The proof is given in \S\ref{sec: integrality}.

\subsection{Tropical well-spaced genus $1$ maps}
The combinatorial analog of the algebraic space $\mathcal{W}_{\Gamma}(X)$ is the moduli space 
$W_{\Gamma}(\Sigma(X))$ of tropical well-spaced genus $1$ maps with discrete data $\Gamma$. 
Let $N_X$ denote the cocharacter lattice of the toric variety $X$, and set
$
(N_{X})_{\mathbb{R}} := N_X \otimes_{\mathbb{Z}} \mathbb{R}.
$
The space $W_{\Gamma}(\Sigma(X))$ is a pure dimensional weighted generalized cone complex of dimension 
$$
\dim W_{\Gamma}(\Sigma(X)) = n +  m +(r-\operatorname{dim}(\operatorname{Span}_\R(\delta_1,\ldots,\delta_m)))
$$ 
constructed via a sequence of subdivisions and restrictions from the moduli space of genus $1$ tropical maps 
\[
f: (\mathsf{C},p_1,\ldots,p_n) \longrightarrow (N_X)_{\mathbb{R}}
\]
to the cocharacter lattice $N_X$ of $X$ with ends prescribed by $\Gamma$. See \S\ref{sec: well-spaced tropical maps}. 

Let $\WG(X)^{\circ} \subseteq \WG(X)$ denotes the locus of maps where the logarithmic structure is trivial. As explained in \cite[Theorem~7.14]{ACMUW} and \cite[Theorem~4.6.2]{RSWII}, there is a continuous tropicalization map
\begin{equation}\label{eqn: tropicalization map}
\mathrm{trop} \colon \WG(X)^{\circ, \mathrm{an}}\longrightarrow {\Sigma}({\WG(X)} )
\to W_{\Gamma}(\Sigma(X)) ,
\end{equation}
Here $(\cdot)^{\mathrm{an}}$ denotes the analytification functor, and $\Sigma(-)$ denotes the tropicalization.
The second map is finite and restricts to an isomorphism on each cone of $\Sigma(\WG(X))$. 
We endow the generalized cone complex $W_{\Gamma}(\Sigma(X))$ with a weight function $w$ that assigns to each maximal cone $\sigma$ the quantity
$$
w(\sigma)= \sum_i \frac{1}{|\mathrm{Aut}(f_{\sigma_i})|},
$$
where the sum ranges over all maximal cones $\sigma_i$ in $\Sigma(\cW_{\Gamma}(X))$ that map to $\sigma$, and $f_{\sigma_i}$ denotes the logarithmic map corresponding to $\sigma_i$. Here, $\mathrm{Aut}(f_{\sigma_i})$ denotes the automorphism group of the logarithmic map $f_{\sigma_i}$ viewed as an object of $\cW_{\Gamma}(X)$. See \S\ref{sec: automorphisms} for a discussion about automorphism.

For example, $w(\sigma)=0$ if and only if $\sigma$ does not lie in the image of the map~\eqref{eqn: tropicalization map}. Thus, computing $w(\sigma)$ amounts, in particular, to determining whether a given tropical map arises from an algebraic map. This problem is solved in \S\ref{sec: realizability}. Example \ref{example: weight zero example} provides an example in which $w(\sigma)=0$.

\subsection{The correspondence theorem}

The main result of this paper is a correspondence theorem between the count of tropical well-spaced genus $1$ maps and the corresponding count of algebraic ones for the morphism
\begin{equation}\label{eqn: map}
\mathsf{ev} \times \mathsf{st} \colon \WG(X) \longrightarrow X^n \times  \overline{\mathcal{M}}_{1,1}=: \Ev.
\end{equation}
Denote by
\[
\Sigma(\mathsf{ev} \times \mathsf{st}) \colon 
\Sigma(\WG(X)) \longrightarrow \Sigma(\Ev)
\]
the induced map of cone complexes. This factors through the natural map
\begin{equation}\label{eqn:tropical_ev_map}
W_{\Gamma}(\Sigma(X)) \longrightarrow \Sigma(X)^n \times \mathbb{R}_{\ge 0}= \Sigma(\Ev),
\end{equation}
given by the tropical evaluation map at the free markings $p_i$ together with the map recalling the $j$-invariant of the domain graph. See \S\ref{sec: key commutative diagram}. By abuse of notation, we also denote the map \eqref{eqn:tropical_ev_map} by $\Sigma(\mathsf{ev} \times \mathsf{st})$. 

Let $\alpha \in \mathsf{CH}^*(X^n)$ and $\beta \in \mathsf{CH}^*(\overline{\mathcal{M}}_{1,1})$ be homogeneous cohomology classes such that 
\[
\deg(\alpha) + \deg(\beta)
= \dim \mathcal{W}_{\Gamma}(X)
\]
Note that $\beta$ is either a multiple of $1$ or a multiple of the point class. It is therefore not restrictive to assume that $\beta$ is one of the two classes $1$ or the point class $\mathsf{pt}_{\oM_{1,1}}$.

Denote by
\begin{equation*}\label{eqn: calpha}
c_{\alpha} \colon \Sigma(X^n)[rn-\deg(\alpha)]\longrightarrow \mathbb{Z}
\end{equation*}
the Minkowski weight associated to $\alpha$. Here, $\Sigma(X^n)[k]$ denotes the set of $k$-dimensional cones in the fan $\Sigma(X^n)$. This is an integer-valued function defined on the cones of $\Sigma(X^n)$ of codimension $\deg(\alpha)$ (equivalently, of dimension $rn-\deg(\alpha)$). We refer to the classical treatment \cite{FuSt} for the theory of Minkowski weights and intersection theory on toric varieties.

There is no complete analog of the theory of Minkowski weights for cone complexes in general. However, the geometry of $\overline{\mathcal{M}}_{1,1}$ is simple enough that we can give an ad-hoc definition of the Minkowski weight associated to $\alpha \otimes \beta$ as follows.

A cone in $\Sigma(\Ev)$ of codimension equal to $\dim \WG(X)$ is either of the form
$
\gamma \times \R_{\ge 0},
$
for some cone $\gamma \in \Sigma(X^n)$ of dimension one less, or of the form
$
\gamma \times \{0\},
$
for some cone $\gamma \in \Sigma(X^n)$ of the same dimension.

We define 
\[
c_{\alpha \otimes \beta} \colon \Sigma(\Ev)[rn+1 - \dim \WG(X)] \longrightarrow \Z
\]
as follows. For a cone $\tau$ of the above type, set
\begin{equation}\label{eqn: def of Minkowski weight c}
c_{\alpha \otimes \beta}(\tau) =
\begin{cases}
c_{\alpha}(\gamma) & \text{if } \tau = \gamma \times \R_{\ge 0}
\text{ and } \beta = 1, \\[6pt]
c_{\alpha}(\gamma) & \text{if } \tau = \gamma \times \{0\}
\text{ and } \beta = \mathsf{pt}_{\overline{\mathcal{M}}_{1,1}}, \\[6pt]
0 & \text{otherwise}.
\end{cases}
\end{equation}

\begin{customthm}{B}[Correspondence theorem, Part I]\label{thm: tropical correspondence}
Let $W \in (\mathbb{R}^r)^n \times \mathbb{R}_{\geq 0}$ be a generic displacement vector. Then
\begin{equation}\label{eqn: correspondence}
(\mathsf{ev} \times \mathsf{st})^*(\alpha \otimes \beta) \cap [\WG(X)]
=
\sum_{\sigma} 
w(\sigma)
\sum_{\tau}
c_{\alpha \otimes \beta}(\tau)\,
\bigl[
N_{\Ev}: \Sigma( \mathsf{ev} \times \mathsf{st})(N_\sigma)  +N_\tau
\bigr],
\end{equation}
where:
\begin{itemize}
\item The outer sum runs over all maximal cones $\sigma$ of $W_{\Gamma}(\Sigma(X))$.
\item The quantity $w(\sigma)$ denotes the weight of the cone $\sigma$. 
\item The inner sum runs over cones $\tau \subset \Sigma(\Ev)$ of codimension $\dim \WG(X)$ such that
$$(\tau + W) \cap \Sigma( \mathsf{ev} \times \mathsf{st})(\sigma) \neq \varnothing.$$
\item The quantity $\bigl[  N_{\Ev}: \Sigma( \mathsf{ev} \times \mathsf{st})(N_\sigma)  +N_\tau  \bigr]$ is the lattice index of the subgroup of $N_{\Ev}= N_X^n \times \Z$ generated by $\Sigma( \mathsf{ev} \times \mathsf{st})(N_\sigma)$ and $N_\tau$. Here, $N_\tau$ (resp.\ $N_\sigma$) denotes the lattice of the Euclidean space associated with $\tau$ (resp.\ $\sigma$). See Notation \ref{notation: Nsigma}.
\end{itemize}
\end{customthm}

For the more tropically minded reader, we make a few remarks connecting the expression \eqref{eqn: correspondence} to a more familiar language.

\begin{remark}
     Write 
\[
W = (W_1, \ldots, W_n, j) \in (\mathbb{R}^r)^n \times \mathbb{R}_{\ge 0}
\] for the displacement vector in Theorem \ref{thm: tropical correspondence}.
Suppose that 
$
\alpha = \alpha_1\otimes \cdots \otimes \alpha_n
$
for certain Chow classes $\alpha_i \in \mathsf{CH}^*(X)$ with associated Minkowski weight $m_i$. Then the condition
\[
(\tau + W) \cap \Sigma(\mathsf{ev} \times \mathsf{st})(\sigma) \neq \emptyset
\]
restricts the sum to those cones $\sigma$ of $W_{\Gamma}(\Sigma(X))$ that interpolate the required geometric conditions. That is, there exists a point
\[
f \colon (\mathsf{C}, p_1, \ldots, p_n) \longrightarrow \Sigma(X)
\]
of $\sigma$ such that
\[
f(p_i) \in W_i + \mathrm{supp}(m_i), \qquad 1 \le i \le n,
\]
and, if $\beta = \mathsf{P}_{\overline{\mathcal{M}}_{1,1}}$, the tropical $j$-invariant of $\mathsf{C}$—that is, the length of its unique cycle—is equal to $j$. Here $\operatorname{supp}(m_i)$ denotes the support of the Minkowski weight $m_i$.
\end{remark}

\begin{remark}\label{rmk: lattice index is determinant}
    The lattice index
\[
\bigl[
N_{\mathsf{Ev}} :  \Sigma( \mathsf{ev} \times \mathsf{st})(N_\sigma) + N_\tau
\bigr]
\]
is equal to the absolute value of the determinant of the linear transformation
\[
N_\sigma \xrightarrow{\;\Sigma(\mathsf{ev} \times \mathsf{st})\;} N_{\mathsf{Ev}} \longrightarrow N_{\mathsf{Ev}}/ N_\tau,
\]
expressed with respect to a choice of $\mathbb{Z}$-basis of $N_\sigma$ and a $\mathbb{Z}$-basis of
\[
N_{\mathsf{Ev}}/ N_\tau \simeq \bigl((\mathbb{Z}^r)^n \times \mathbb{Z}\bigr) / N_\tau.
\]
\end{remark}

Tropical geometers often work with tropical maps to $(N_X)_{\mathbb{R}}$ rather than to $\Sigma(X)$. The cone complex $W_\Gamma(\Sigma(X))$ is a subdivision of the cone complex $W_\Gamma((N_X)_{\mathbb{R}})$, which parametrizes well-spaced tropical maps to $(N_X)_{\mathbb{R}}$. This subdivision involves a rescaling of the lattice structures. For a cone $\sigma$ of $W_{\Gamma}(\Sigma(X))$, we denote the corresponding scaling factor by $\mathrm{lat}(\sigma)$. See Definition~\ref{def: lattice index} for details.

Another natural quantity associated with a tropical map is its loop multiplicity.

\begin{definition}\label{def:loop}
The loop multiplicity of a tropical map $[h \colon \tropC \to \Sigma(X)]$ in $\Sigma(X)$ (or $(N_X)_\mathbb{R}$) is defined as the index of the map
\[
\mathbb{Z}^{E(\tropC_0)} \longrightarrow \operatorname{span}_{\mathbb{R}}(h(\tropC_0)) \cap N_X,
\]
given by
\[
(\ell_1, \ldots, \ell_k) \longmapsto \sum \ell_i u_i.
\]
Here, $\tropC_0$ denotes the unique cycle of $\tropC$, $E(\tropC_0)$ its edge set, and, for each edge $e_i \in E(\tropC_0)$, $u_i$ denotes the direction vector corresponding to a chosen orientation $\vec{e}_i$ of $e_i$.
\end{definition}

Clearly, the definition of loop multiplicity does not depend on the particular tropical curve parametrized by a cone, nor on the choice of orientation of the edges in $E(\tropC_0)$. It therefore determines a well-defined loop multiplicity $\mathrm{loop}(\sigma)$ (resp. $\mathrm{loop}(\sigma')$) for every cone $\sigma \subseteq W_\Gamma(\Sigma(X))$ (resp. $\sigma' \subseteq W_\Gamma((N_X)_\R)$).

\begin{remark}
Let $\sigma'$ be the cone in the moduli space of tropical well-spaced maps to $(N_X)_{\mathbb{R}}$ to which $\sigma$ maps. Then one has
\[
\left[
N_{\mathsf{Ev}}/N_\tau: \Sigma( \mathsf{ev} \times \mathsf{st})(N_\sigma)
\right]
=
\mathrm{lat}(\sigma) \cdot
\left[
N_{\mathsf{Ev}}/N_\tau : \Sigma( \mathsf{ev} \times \mathsf{st})(N_{\sigma'})
\right].
\]

\end{remark}

Let $\sigma$ be a maximal cone of $W_{\Gamma}(\Sigma(X))$ and denote by $\{ \sigma_i\}$ the set of maximal cones of $\Sigma(\cW_{\Gamma}(X))$ above $\sigma$. For every index $i$ let $f_{\sigma_i}$ be the point in $\cW_{\Gamma}(X)$ corresponding to $\sigma_i$. Let $\{f_j^\mathrm{fine}\}$ be the set of non-isomorphic basic \textit{fine} logarithmic maps associated to the collection $\{f_{\sigma_i}\}$. More explicitly, a basic fine logarithmic map belongs to the set $\{f_j^\mathrm{fine}\}$ if and only if there exists an index $i$ such that the basic fine logarithmic map underlying $f_{\sigma_i}$ is isomorphic to it. Let $n(\sigma)$ be the cardinality of the set $\{ f_j^\fine\}$ (See Equation \ref{eqn: def n(sigma)} for a more precise definition).  In Theorem \ref{thm: automorphisms}, we will show that the automorphism group
$\mathrm{Aut}(f_{\sigma_i})$ of each $f_{\sigma_i}$ 
depends only on the automorphism group of the underlying fine basic logarithmic map. Moreover, by Theorem \ref{thm: number of logarithmic enhancements}, over each $f_j^\fine$ there are precisely $\mathrm{sat}(\sigma)$ maps $f_{\sigma_i}$ with associated basic fine logarithmic map $f_j^\fine$. Here $\mathrm{sat}(\sigma)$ denotes the saturation index of $\sigma$ (see Definition \ref{def: sat}). Then one has
\begin{equation}\label{eqn: reduction_to_aut_of_fine_map}
w(\sigma)= \mathrm{sat}(\sigma) \sum_{j=1}^{n(\sigma)} \frac{1}{|\mathrm{Aut}(f_j^\fine)|}
\end{equation}

where now $\mathrm{Aut}(f_j^\fine)$ denotes the automorphism group of the fine map $f_j^\fine$. There is a natural normal subgroup 
\[
 \prod_{V \text{ unconstrained}} \mu_{d_V} \subseteq \mathrm{Aut}(f_j^{\mathsf{fine}}).
\]
Here, the product is over the unconstrained vertices $V$ of the tropical curve - those corresponding to components $C_V$ that multiply cover a toric $\mathbb{P}^1$ in $X$ - and the value $d_V$ denotes the degree of the restriction $\underline{f}_j|_{C_V}$ onto its image. See Definition \ref{def: dv}. 

In \S\ref{sec: curves-in-Nx}, we prove the following remarkable relation between saturation and lattice index, loop multiplicity and automorphisms.

\begin{customthm}{C}\label{thm: relation indices, loop aut}
    For a maximal cone $\sigma$ of $W_{\Gamma}(\Sigma(X))$, we have following equality 
    \[
\mathrm{sat}(\sigma)= 
\frac{ \bigg(
\prod\limits_{V \text{ unconstrained}} d_V \bigg) \cdot \mathrm{loop}(\sigma)
}{\mathrm{lat}(\sigma)}.
\]
\end{customthm}
Set
\begin{equation*}
\nn(\sigma) :=  \sum_{j=1}^{n(\sigma)} \frac{1}{|\mathrm{Aut}(f_j^\fine)/ \prod \mu_V|} 
\end{equation*}

If $\sigma'$ is a maximal cone of $W_{\Gamma}((N_X)_\R)$ satisfying
\[
(\tau + W)\cap \Sigma(\mathsf{ev}\times\mathsf{st})(\sigma')\neq\emptyset,
\]
then this intersection consists of a unique point, which lies in a unique maximal cone $\sigma$ of the subdivision $W_{\Gamma}(\Sigma(X))$. Moreover, as explained in Corollary \ref{cor: n' sigma depends only on sigma'}, the automorphism-weighted factor $\nn(\sigma)$ depends only on the cone $\sigma'$. Consequently, it defines a well-defined quantity $\nn(\sigma')$. 

Putting everything together, Equation~\eqref{eqn: correspondence} yields the following corollary.

\begin{corollary}\label{cor: correspondence_thm_Hannah_style}
    Let $W \in (\mathbb{R}^r)^n \times \mathbb{R}_{\geq 0}$ be a generic displacement vector. Then
    \begin{equation}\label{eqn: correspondence bis}
(\mathsf{ev} \times \mathsf{st})^*(\alpha \otimes \beta) \cap [\WG(X)]
=
\sum_{\sigma'} \mathrm{loop}(\sigma')
\cdot \nn(\sigma') 
\sum_{\tau}
c_{\alpha \otimes \beta}(\tau)\,
\bigl[ 
N_{\Ev}: \Sigma( \mathsf{ev} \times \mathsf{st})(N_{\sigma'})  + N_\tau
\bigr],
\end{equation}
where the external sum is over maximal cones $\sigma'$ in $W_\Gamma((N_X)_\R)$ and the inner sum is over cones $\tau \subseteq \Sigma(\mathsf{Ev})$ of codimension $\dim \WG(X)$ and such that $(\tau + W) \cap \Sigma( \mathsf{ev} \times \mathsf{st})(\sigma') \neq \varnothing.$
\end{corollary}

Note that, as in Remark~\ref{rmk: lattice index is determinant}, the lattice index
\[
\bigl[
N_{\Ev}: \Sigma( \mathsf{ev} \times \mathsf{st})(N_{\sigma'})  + N_\tau
\bigr]
\]
is the absolute value of the usual determinant.

\subsection{Weights of maximal cones of $W_{\Gamma}(\Sigma(X))$}\label{sec: result weights}

Corollary \ref{cor: correspondence_thm_Hannah_style} reduces the enumerative problem in $X$ to a purely combinatorial one, the only remaining ingredient being the automorphism-weighted factor
$
\nn(\sigma')
$
attached to each maximal cone $\sigma'$ of $W_{\Gamma}((N_{X})_\R)$. An explicit formula for this factor is obtained in \S\ref{sec: weights}. Before stating our result we require some notation.

Fix a maximal cone $\sigma'$ in $W_{\Gamma}((N_X)_\R)$. This corresponds to a tropical type of well-spaced tropical maps to $(N_X)_\R$ (see \S\ref{sec: well-spaced tropical maps}). Let $H_0$ (resp. $H_1$) be the vector subspace of $(N_X)_{\mathbb{R}}$ given by the span of the direction vectors in the cycle (resp. neighborhood of the cycle).

\begin{notation}
     \label{notation: shape nbhd cycle}
    Let $V_1,\ldots,V_t$ be the vertices of the cycle with valency at least $4$.
    We may choose an orientation of the cycle, and assume that the above ordering respects the induced ordering of the vertices in the cycle. This is illustrated in Figure \ref{fig:shape-nbhd-cycle}. 

    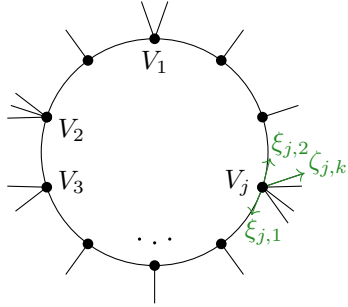
\begin{figure}[h]
        \centering
        \begin{tikzpicture}
            \def\radius{1.5}
            \def\linelength{0.5}
            \draw (0,0) circle (\radius);
            
            \foreach \i in {1,2, 4, 5, 6, 8, 9} {
                \pgfmathsetmacro{\angle}{36*\i+90}
                
                \coordinate (P) at ({\radius*cos(\angle)}, {\radius*sin(\angle)});
                \coordinate (Q) at ({(\radius+\linelength)*cos(\angle)},
                            {(\radius+\linelength)*sin(\angle)});
                \fill (P) circle (2pt);
                \draw (P) -- (Q);
            }
            \foreach \i in {0, 2, 3, 7}{
                \pgfmathsetmacro{\angle}{36*\i+90}
                     \coordinate (P) at ({\radius*cos(\angle)}, {\radius*sin(\angle)});
                     \fill (P) circle (2pt);
                    \coordinate (Q) at ({(\radius+\linelength)*cos(\angle+5)},
                                {(\radius+\linelength)*sin(\angle+5)});
                    
                    \coordinate (Q') at ({(\radius+\linelength)*cos(\angle-5)},
                                {(\radius+\linelength)*sin(\angle-5)});
                                
                    \draw (P) -- (Q);
                    \draw (P) -- (Q');
            }
            \foreach \i in {7}{
                \pgfmathsetmacro{\angle}{36*\i+90}
                    \coordinate (P) at ({\radius*cos(\angle)}, {\radius*sin(\angle)});
                    \coordinate (Q) at ({(\radius+\linelength)*cos(\angle+10)},
                                {(\radius+\linelength)*sin(\angle+10)});
                    
                    \coordinate (Q') at ({(\radius+\linelength)*cos(\angle-10)},
                                {(\radius+\linelength)*sin(\angle-10)});
                                
                    \draw (P) -- (Q);
                    \draw (P) -- (Q');

                    \coordinate (P') at ({(\radius+0.03)*cos(\angle+15)}, {(\radius+0.03)*sin(\angle+15)});
                    \coordinate (P'') at ({(\radius+0.03)*cos(\angle-15)}, {(\radius+0.03)*sin(\angle-15)});

                \draw[->, forestgreen] (P) -- (P');
                \draw[->, forestgreen] (P) -- (P'');
                \draw[->, forestgreen] (P) -- (Q);
                \fill (P) circle (2pt);
            
            }
            %labels
            \node at (0,1.2) {$V_1$};
            \node at (-1.1, 0.3) {$V_2$};
            \node at (-1.1, -0.4) {$V_3$};
            \fill (0,-1.2) circle (0.6pt);
            \fill (-0.2,-1.15) circle (0.6pt);
            \fill (0.2,-1.15) circle (0.6pt);
            \node at (1.1, -0.4) {$V_j$};

            \node[text = forestgreen] at (1.45, -1.05) {$\xi_{j,1}$}; 
            \node[text = forestgreen] at (1.8, 0.1) {$\xi_{j,2}$}; 
            \node[text = forestgreen] at (2.3,-0.15) {$\zeta_{j,k}$}; 
            
        \end{tikzpicture}
        \caption{General shape of the neighborhood of the cycle.}
        \label{fig:shape-nbhd-cycle}
    \end{figure}
    
    Let $\zeta_{j,k}$, for $k = 1, \ldots, e_j$, denote the half-edges adjacent to $V_j$ that are not contained in the cycle. Let $\xi_{j,1}$ and $\xi_{j,2}$ denote the half-edges adjacent to $V_j$ that lie in the cycle. Let
    $
    u_{j,k} = u_{\zeta_{j,k}}
    $
    be the direction vector of $\zeta_{j,k}$, and let $\bar{u}_{j,k}$ denote the image of $u_{j,k}$ under the projection
    $
    H_1 \to H_1/H_0.
    $
\end{notation}

\begin{proposition}\label{prop: nbhd-cycle-span}
In the setting above, the set of vectors 
$$\{ \bar{u}_{j,k} | 1 \leq j \leq t, 1 \leq k \leq e_j - 1\}$$
is an $\mathbb{R}$-basis for $H_1/H_0$. 
\end{proposition}
\begin{proof}
    Follows from \cite[Proposition 3.2.18(b)]{Tor14}.
\end{proof}

\begin{construction} \label{const: polytope}
     Set $r_0 = \dim H_0$ and $r_1 = \dim H_1$. Suppose there is at least one vertex in the cycle with valence at least four. Define $M_{\sigma'} \in \mathrm{Hom}(\R^{r_1-r_0}, H_1/H_0)$ to be the map given by the $e_{j,k} \mapsto \bar{u}_{j,k}$. Here $e_{j,k}$ denotes the standard basis of $\R^{r_1-r_0}$. Consider the parallelotope in $H_1/H_0$ given by the image of the unit cube under the map $M_{\sigma'}: \mathbb{R}^{r_1-r_0} \rightarrow H_1/H_0$. 
    Let $x_{j,k}$ be the coordinates on $\R^{r_1-r_0}$ and call $\triangle \subseteq \R^{r_1-r_0}$ the union of diagonals 
    $$\triangle = \bigcup_{1 \leq j \leq t, 1 \leq k,k' \leq e_j-1} \{x_{j,k} = x_{j,k'}\}.$$
    Define the region 
    $$P(\sigma') = M_{\sigma'}([0,1]^{ \times r_1-r_0}) \smallsetminus M(\triangle).$$     
\end{construction}

\begin{definition}
    For any region $P \subseteq (N_X)_{\mathbb{R}}$, we denote by $i(P)$ the number of interior lattice points in $P$. 
\end{definition}

\begin{customthm}{D}[Correspondence theorem, Part II] \label{thm: calculation of weights}
    Let $\sigma'$ be a cone of maximum dimension in $W_\Gamma((N_X)_\R)$. 
    \begin{enumerate}[label=(\alph*)]
        \item If every vertex on the cycle has valence at most three, then $\nn(\sigma') = 1/|\Aut(\sigma')|.$
        \item If there is at least one vertex on the cycle with valence at least four, then we have
    $$ \nn(\sigma') = \frac{i(P(\sigma'))}{|\Aut(\sigma')|}$$
    where $P(\sigma') \subset (N_X)_{\mathbb{R}}$ is the region associated to $\sigma'$ in Construction~\ref{const: polytope}.
    \end{enumerate}
    Here $\Aut(\sigma')$ is the automorphism group of tropical curves parametrized by $\sigma'$. 
\end{customthm}

Theorems~\ref{thm: tropical correspondence} and~\ref{thm: calculation of weights} together establish a full correspondence  between algebraic and tropical genus~$1$ curve counts in all dimensions.

\begin{remark}
For the cycle-theoretic reader, the intersection numbers
\[
(\mathsf{ev} \times \mathsf{st})^*(\alpha \otimes \beta) \cap [\WG(X)]
=
(\mathsf{ev} \times \mathsf{st})_*[\WG(X)] \cdot (\alpha \otimes \beta)
\]
for all classes $\alpha \in \mathsf{CH}^*(X^n)$ and $\beta \in \mathsf{CH}^*(\oM_{1,1})$ determine the pushforward class
$(\mathsf{ev} \times \mathsf{st})_*[\WG(X)] \in \mathsf{CH}_*(X^n \times \oM_{1,1})$.

Thus, our correspondence theorem can be interpreted as the computation of the Minkowski weight associated to the cycle $(\mathsf{ev} \times \mathsf{st})_*[\WG(X)]$.
\end{remark}

\subsection{Specialization to low dimension}

The automorphism-weighted factor $\nn(\sigma')$ takes a particularly simple form for low dimenisonal varieties $X$.

\begin{definition} \label{def:deficandov}
    For a parametrized tropical curve $[h: \tropC \rightarrow (N_X)_{\mathbb{R}}]$:
    \begin{enumerate}
        \item The deficiency of $[h: \tropC \rightarrow (N_X)_\R]$ is the codimension in $(N_X)_\R$ of the subspace spanned by the image of the cycle $h(\tropC_0)$.
        \item The overvalence of $\tropC$ is 
        $$ \ov{\tropC} = \sum_{V \in V(\tropC), \val{V} \geq 3} \val{V} - 3.$$
    \end{enumerate}
     
\end{definition}

These depend only on the tropical type, so for each cone $\sigma'$ of $W_{\Gamma}((N_X)_{\mathbb{R}})$, we can define $\defic{\sigma'}, \ov{\sigma'}$ as the deficiency and overvalence respectively of any parametrized tropical curve with this type.

\begin{proposition}\label{thm:mult1} 
    Suppose the dimension of $X$ is $1$ and let $\sigma'$ be a cone of maximum dimension in $W_{\Gamma}((N_X)_{\mathbb{R}})$. Then $\nn(\sigma')$ is given as follows:
    \begin{enumerate}
        \item \textbf{Deficiency 0.} $\nn(\sigma') = 1/|\Aut(\sigma')|$. 
        \item  \textbf{Deficiency 1.} There are three cases (examples illustrated in Figure \ref{fig: dim1, defic1}):
        
        \begin{enumerate}
    \item The neighborhood of the cycle spans. In this case we must have $\ov{\sigma'} = 1$ and the unique $4$-valent vertex $V$ is on the cycle. The vertex $V$ has two adjacent non-contracted edges. Let $\pm d \cdot u$ be their direction vectors, where $ d \in \Z_{>0}$ and $u \in N_X$ is primitive. 
    Then $\nn(\sigma') = (d-1)/|\Aut(\sigma')|$. 
    \item The neighborhood of the cycle is contracted, and $\ov{\sigma'} = 1$. Then the unique $4$-valent vertex is at the minimum distance from the cycle to where the curve leaves the point of contraction. We have $\nn(\sigma') = 1/|\Aut(\sigma')|$. 
    \item The neighborhood of the cycle does not span, and $\ov{\sigma'} = 0$. Then at the minimum distance from the cycle to where the curve leaves the point of contraction, there must be exactly two vertices. We have $\nn(\sigma') = 1/|\Aut(\sigma')|$. 
    \begin{figure}[h]
    \centering
    \begin{subfigure}[t]{0.3\textwidth}
        \centering
        \begin{tikzpicture}
            \draw (-2, 0) -- (2,0); 
            \draw[dashed] (0,0)  to[in=50,out=130,loop, style={min distance=12mm}] (0,0);
            \draw[->] (0, -0.7) -- (0,-1.2); 
            \draw (-2,-1.5) -- (2, -1.5);
            \node[forestgreen] at (0.6,0.2) {$d \cdot u$};
            \node[forestgreen] at (-0.7,0.2) {$-d\cdot u$};
            \draw[forestgreen, ->] (0,0) -- (-1,0);
            \draw[forestgreen, ->] (0,0) -- (1,0);
            \node at (0, -0.3) {$V$};
            \fill (0,0) circle[radius=2pt];
        \end{tikzpicture}
        \caption{Neighborhood spans.}
        \label{fig:dim1nbhdspan}
    \end{subfigure}
    \begin{subfigure}[t]{0.3\textwidth}
        \centering
        \begin{tikzpicture}
            \fill (0,0) circle[radius=2pt];
            \draw (-2, 0) -- (0,0); 
            \draw (0, -0.05) -- (2, -0.05); 
            \draw (0, 0.05) -- (2, 0.05);
            \draw[dashed] (0,0) -- (0,1);
            \draw[dashed] (0,1)  to[in=50,out=130,loop, style={min distance=12mm}] (0,1);
            \draw[->] (0, -0.5) -- (0,-1); 
            \draw (-2,-1.5) -- (2, -1.5);
        \end{tikzpicture}
        \caption{Neighborhood doesn't span, overvalence $1$.}
        \label{fig:dim1ov1}
    \end{subfigure}
    \begin{subfigure}[t]{0.3\textwidth}
        \centering
        \begin{tikzpicture}
            \fill (0,0) circle[radius=2pt];
            \fill (0,1.5) circle[radius=2pt];
            \draw (-2, 0) -- (2,0); 
            \draw (-2, 1.5) -- (2,1.5);
            \draw[dashed] (0,0) -- (0, 0.5); 
            \draw[dashed] (0.05, 0.5) -- (0.05, 1); 
            \draw[dashed] (-0.05, 0.5) -- (-0.05, 1); 
            \draw[dashed] (0, 1) -- (0,1.5); 
            \draw[->] (0, -0.5) -- (0,-1); 
            \draw (-2,-1.5) -- (2, -1.5);
            \fill[gray] (0,0.5) circle[radius=2pt];
            \fill[gray] (0,1) circle[radius=2pt];
            \node at (0.2,1.25) {$l$};
            \node at (0.2,0.25) {$l$};
        \end{tikzpicture}
        \caption{Neighborhood doesn't span, overvalence 0.}
        \label{fig:dim1,ov0}
    \end{subfigure}
    \caption{Dimension 1, deficiency 1. Contracted parts are shown dashed.}
    \label{fig: dim1, defic1}
\end{figure}
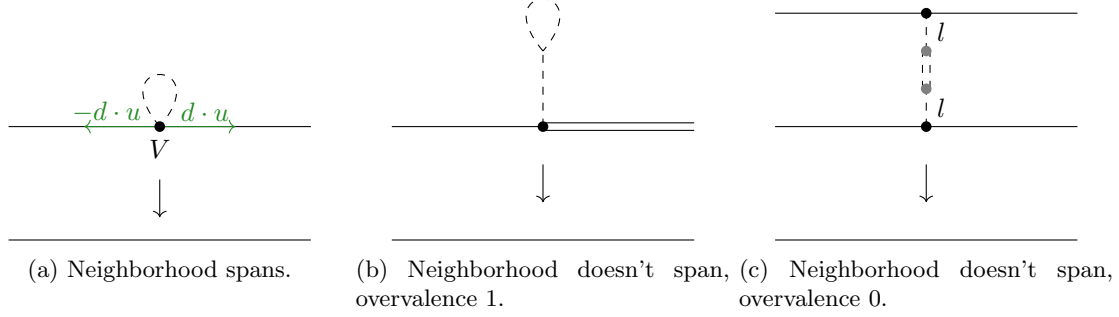
\end{enumerate}
    \end{enumerate}
\end{proposition}

  When $\dim X = 2$ and $\sigma'$ has deficiency $2$, the polygon $P(\sigma')$ is closely related to the dual polygon associated with the tropical curve at the unique vertex in the cycle.

\begin{proposition} \label{thm:mult2}
    Suppose the dimension of $X$ is $2$ and let $\sigma'$ be a cone of maximum dimension in $W_{\Gamma}((N_X)_{\mathbb{R}})$. Then $\nn(\sigma')$ is given as follows :
    \begin{enumerate}
        \item \textbf{Deficiency 0.} $\nn(\sigma') = 1$.
        \item \textbf{Deficiency 1.}  There are three cases (examples illustrated in Figure \ref{fig: dim2, defic1}):
            \begin{enumerate}
                \item The neighborhood of the cycle spans. In this case we must have $\ov{\sigma'} = 1$ and the unique $4$-valent vertex $V$ is on the cycle. Let $n_1$ be the primitive vector of an edge in the cycle adjacent to $V$. Let $u$ be the direction vector of one of the edges adjacent to $V$ leaving the cycle. Then 
                $$\nn(\sigma') = \frac{|\det(u, n_1)| - 1}{|\Aut(\sigma')|}.$$ 
                \item The neighborhood of the cycle does not span, and $\ov{\sigma'} = 1$. Then $\nn(\sigma') = 1/|\Aut(\sigma')|$. 
                \item The neighborhood of the cycle does not span, and $\ov{\sigma'} = 0$. Then $\nn(\sigma') = 1/|\Aut(\sigma')|$. 

                \begin{figure}[h]
                \centering
                \begin{subfigure}[t]{0.3\textwidth}
            		\centering
            		\begin{tikzpicture}
            			\fill (0,0) circle[radius=2pt];
            			\fill (2,0) circle[radius=2pt];
            			\fill (1.5,0) circle[radius=2pt];
            			\draw (0, 0.05) -- (1.5, 0.05);
            			\draw (0, -0.05) -- (1.5, -0.05);
            			\draw (1.5, 0) -- (2,0);
            			\draw (0, 0) -- (-1, 1); 
            			\draw (0, 0) -- (-1, -1); 
            			\draw (2, 0) -- (3, 1); 
            			\draw (2, 0) -- (3, -1);
            
                        \draw[forestgreen, ->] (0,0.15) -- (-0.5, 0.65);
                        \node[forestgreen] at (-0.5,0.8) {$u$};
            
                        \draw[forestgreen, ->] (0.2, 0.15) -- (1, 0.15); 
                        \node[forestgreen] at (0.6, 0.35) {$a n_1$};

                        \draw[forestgreen, ->] (0.2, -0.15) -- (1, -0.15);
                        \node[forestgreen] at (0.6, -0.4) {$b n_1$};
            			
                        \node at (-0.4,0) {$V$};
            		\end{tikzpicture}
            		\caption{Neighborhood spans.}
            		\label{fig:dim2def1}
            	\end{subfigure}	
                \begin{subfigure}[t]{0.3\textwidth}
                    \centering
                    \begin{tikzpicture}
                        \fill (-1,0) circle[radius=2pt];
                        \fill (2,0) circle[radius=2pt];
                        \fill (0.5,0) circle[radius=2pt];
                        \fill (1.5,0) circle[radius=2pt];
                        \draw (-1, 0) -- (0.5, 0);
                        \draw (0.5, 0.05) -- (1.5, 0.05); 
                        \draw (0.5, -0.05) -- (1.5, -0.05); 
                        \draw (1.5, 0) -- (2, 0); 
                        \draw (-1, 0) -- (-2, 1); 
                        \draw (-1, 0) -- (-2, -1); 
                        \draw (2, 0) -- (3, 1); 
                        \draw (2, 0) -- (3, -1);
                        \draw (2,0) -- (3,0.5); 
                    \end{tikzpicture}
                    \caption{Neighborhood doesn't span, overvalence $1$.}
                    \label{fig:dim2def1case2}
                  \end{subfigure}
                \begin{subfigure}[t]{0.3\textwidth}
                    \centering
                    \begin{tikzpicture}
                        \fill (0,0) circle[radius=2pt];
                        \fill (2,0) circle[radius=2pt];
                        \fill (0.5,0) circle[radius=2pt];
                        \fill (1.5,0) circle[radius=2pt];
                        \draw (0, 0) -- (0.5, 0);
                        \draw (0.5, 0.05) -- (1.5, 0.05); 
                        \draw (0.5, -0.05) -- (1.5, -0.05); 
                        \draw (1.5, 0) -- (2, 0); 
                        \draw (0, 0) -- (-1, 1); 
                        \draw (0, 0) -- (-1, -1); 
                        \draw (2, 0) -- (3, 1); 
                        \draw (2, 0) -- (3, -1);
                        \node at (0.3,0.3) {$l$};
                        \node at (1.7,0.3) {$l$};
                    \end{tikzpicture}
                    \caption{Neighborhood doesn't span, overvalence $0$.}
                    \label{fig:dim2def1case3}
                  \end{subfigure}
                    
                  \caption{Dimension 2, deficiency 1.}
                \label{fig: dim2, defic1}
            \end{figure}
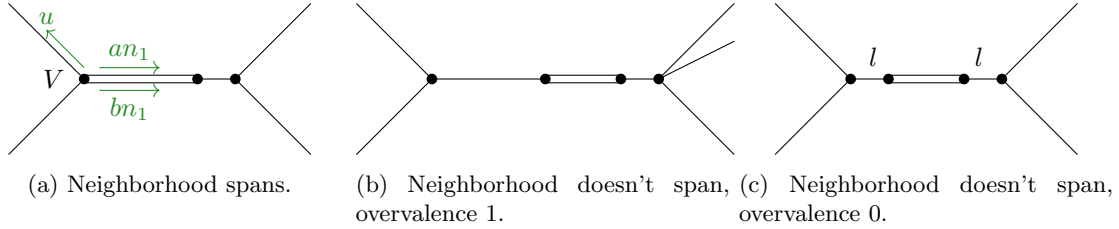

            \end{enumerate}
        \item \textbf{Deficiency 2.} We have the following cases (examples illustrated in Figure \ref{fig: dim2, defic2}):
        \begin{enumerate}
            \item The neighborhood of the cycle spans. \begin{enumerate}[label=(\roman*)]
                \item If there is a single $5$-valent vertex $V$ on the cycle, then
                $$\nn(\sigma') = \frac{2 \times i(\text{dual polygon at $V$})}{|\Aut(\sigma')|}.$$ 
                \item If there are two $4$-valent vertices on $V_1, V_2$, suppose the non-contracted half-edges adjacent to $V_1, V_2$ have directions $\pm u_1, \pm u_2$. Then $\nn(\sigma')$ is the number of interior points in the parallelogram with vertices $0, u_1, u_1+u_2, u_2$, divided by $|\Aut(\sigma')|$. 
            \end{enumerate}
            \item The neighborhood of the cycle is contained in a line but not contracted. There is a single $4$-valent vertex on the cycle with two adjacent non-contracted edges. Let $\pm d \cdot u$ be their direction vectors, where $ d \in \Z_{>0}$ and $u \in N_X$ is primitive. 
            Then,  $\nn(\sigma') = (d-1)/|\Aut(\sigma')|$.
            \item The neighborhood of the cycle is contracted. We have $\nn(\sigma') = 1/|\Aut(\sigma')|$.
        \end{enumerate}
        \begin{figure}[H]
            \centering
            \begin{subfigure}[t]{0.35\textwidth}
                    \centering
                    \begin{tikzpicture}
                        \draw (-1.2, 0) -- (0,0); 
                        \draw (0, 0) -- (1,1); 
                        \draw (0,0) -- (1, -1);
                        \draw[dashed] (0,0)  to[in=50,out=130,loop, style={min distance=12mm}] (0,0);
                        \fill[white] (1.5,0) circle[radius=2pt];
                        \node at (-0.1, -0.3) {$V$};
                        \draw[->, forestgreen] (0,0) -- (0.5, 0.5);
                        \draw[->, forestgreen] (0,0) -- (0.5, -0.5); 
                        \node[forestgreen] at (0.8, 0.5) {$u_1$};
                        \node[forestgreen] at (0.8, -0.5) {$u_2$};
                        \fill (0,0) circle[radius=2pt];
                    \end{tikzpicture}
                    \begin{tikzpicture}
                        \draw[dashed] (0, 0) to[out=30, in=150, relative] (0, 1); 
                        \draw[dashed] (0, 0) to[out=60,in=300] (0, 1); 
                        \draw[skyblue] (-1, 1) -- (1, 1); 
                        \draw[coralpink] (-0.8, 0.3) -- (0.8,-0.3);
                        \draw[->] (0, -0.3) -- (0, -0.6);
                        \draw[coralpink] (-0.8, -0.7) -- (0.8, -1.3);
                        \draw[skyblue] (-1, -1) -- (1, -1);
                        \node at (0.3, 0.8) {$V_1$};
                        \node at (-0.3, -0.2) {$V_2$};
                        \draw[->, forestgreen] (0,-1) -- (0.5, -1);
                        \draw[->, forestgreen] (0,-1) -- (0.4, -1.15); 
                        \node[forestgreen] at (0.5, -0.8) {$u_1$};
                        \node[forestgreen] at (0.2, -1.4) {$u_2$};
                        \fill (0,-1) circle[radius=2pt];
                        \fill (0,0) circle[radius=2pt];
                        \fill (0,1) circle[radius=2pt];
                    \end{tikzpicture}
                    \caption{The neighborhood of the cycle spans.}
                    \label{fig:loopfivevalent}
            \end{subfigure}
            \begin{subfigure}[t]{0.6\textwidth}
    			\centering
                \begin{tikzpicture}
    				\fill (0,0) circle[radius=2pt];
    				\fill (2,0) circle[radius=2pt];
    				\fill (1.5,0) circle[radius=2pt];
    				\draw (0, 0) -- (2, 0);
    				\draw (0, 0) -- (-1, 1); 
    				\draw (0, 0) -- (-1, -1); 
    				\draw (2, 0) -- (3, 1); 
    				\draw (2, 0) -- (3, -1);
                    \draw (2,0) -- (3,0.5); 
    				\draw[dashed] (1.5,0)  to[in=50,out=130,loop, style={min distance=12mm}] (1.5,0);
    				\fill[white] (3.5,0) circle[radius=2pt];
                    \node at (1.5, -0.3) {$V_1$};
                    \node at (2, -0.3) {$V_2$};
    			\end{tikzpicture}
    			\begin{tikzpicture}
    				\fill (0,0) circle[radius=2pt];
    				\fill (2,0) circle[radius=2pt];
    				\fill (1,0) circle[radius=2pt];
    				\draw (0, 0) -- (2, 0);
    				\draw (0, 0) -- (-1, 1); 
    				\draw (0, 0) -- (-1, -1); 
    				\draw (2, 0) -- (3, 1); 
    				\draw (2, 0) -- (3, -1);
                    \node at (0.3,0.3) {$l$};
                    \node at (1.7,0.3) {$l$};
    				\draw[dashed] (1,0)  to[in=50,out=130,loop, style={min distance=12mm}] (1,0);
    			\end{tikzpicture}
    			\caption{The neighborhood of the cycle is contained in a line but not contracted. }
    			\label{fig:dim2defic2case2}
    		  \end{subfigure}

    \end{figure}

    \begin{figure}[H]
    \ContinuedFloat
              
		\begin{subfigure}[t]{\textwidth}
    			\centering

                \begin{tikzpicture}
                    \fill (0,0) circle[radius=2pt];
                    \draw (-1, -0.5) -- (0,0); 
                    \draw (-1, 0.6) -- (0,0);
                    \draw (0, 0) -- (1,1); 
                    \draw (0,0) -- (1, -1);
                    \draw[dashed] (0,0) -- (0,0.6);
                    \draw[dashed] (0,0.6)  to[in=50,out=130,loop, style={min distance=12mm}] (0,0.6);
                    \node at (0,-0.3) {$V$};
                    \fill[white] (2,0) circle[radius=2pt];

                    \node at (0,-0.3) {$V$};
                \end{tikzpicture}
                \begin{tikzpicture}
                    \draw[skyblue] (-1.2, 0) -- (0,0); 
                    \draw[skyblue] (0, 0) -- (1,0.3); 
                    \draw[skyblue] (0,0) -- (1, -0.3);
                    
                    \draw[coralpink] (1,1.2) -- (-1, 1.8); 
                    \draw[dashed] (0,0) -- (0, 0.5); 
                    \draw[dashed] (0, 0.5) to[out=30, in=150, relative] (0, 1); 
                    \draw[dashed] (0, 0.5) to[out=60,in=300] (0, 1); 
                    \draw[dashed] (0, 1) -- (0,1.5); 
                    \fill[gray] (0,0.5) circle[radius=2pt];
                    \fill[gray] (0,1) circle[radius=2pt];
                    
                    \draw[->] (0,-0.3) -- (0, -0.7);
                    
                    \draw[gray] (-1, -1) -- (2, -1) -- (1, -2) -- (-2, -2) -- cycle; 
                    \draw[skyblue] (-1.2, -1.5) -- (0,-1.5); 
                    \draw[skyblue] (0, -1.5) -- (1,-1.2); 
                    \draw[skyblue] (0,-1.5) -- (1, -1.8);
                    \draw[coralpink] (-0.3, -1.1) -- (0.3, -1.9); 
                    \fill (0,-1.5) circle[radius=2pt];
                    \fill (0,0) circle[radius=2pt];
                    \fill (0,1.5) circle[radius=2pt];
                    
                    \fill[white] (2,0) circle[radius=2pt];
                    \fill[white] (0,2) circle[radius=2pt];

                    \node at (-0.3,0.3)  {$l$};
                    \node at (-0.3,1.3) {$l$};
    			\end{tikzpicture}
                \begin{tikzpicture}
    				
                    \draw[skyblue] (-1.2, 0) -- (0,0); 
                    \draw[skyblue] (0,0) -- (0.5, 0); 
                    \draw[skyblue] (0.5, 0) -- (1.5,0.3); 
                    \draw[skyblue] (0.5,0) -- (1.5, -0.3);
                    \draw[skyblue] (0.5, 0) -- (1.5, 0.1); 
                    
                    \draw[coralpink] (1,1.5) -- (-1, 1.5); 
                    \draw[dashed] (0,0) -- (0, 0.5); 
                    \draw[dashed] (0, 0.5) to[out=30, in=150, relative] (0, 1); 
                    \draw[dashed] (0, 0.5) to[out=60,in=300] (0, 1); 
                    \draw[dashed] (0, 1) -- (0,1.5); 
                    \fill[gray] (0,0.5) circle[radius=2pt];
                    \fill[gray] (0,1) circle[radius=2pt];
                    
                    \draw[->] (0,-0.3) -- (0, -0.7);
                    
                    \draw[gray] (-1, -1) -- (2, -1) -- (1, -2) -- (-2, -2) -- cycle; 
                    \draw[skyblue] (-1.2, -1.5) -- (0,-1.5); 
                    \draw[skyblue] (0, -1.5) -- (0.3, -1.5); \draw[skyblue] (0.3, -1.5) -- (1, -1.4); 
                    \draw[skyblue] (0.3, -1.5) -- (1,-1.2); 
                    \draw[skyblue] (0.3,-1.5) -- (1, -1.8);
                    \draw[coralpink] (-1, -1.51) -- (1, -1.51); 
                    \fill (0,-1.505) circle[radius=2pt];
                    \fill (0,0) circle[radius=2pt];
                    \fill (0,1.5) circle[radius=2pt];
                    
                    \fill[white] (2,0) circle[radius=2pt];
                    \fill[white] (0,2) circle[radius=2pt];

                    \node at (-0.3,0.3)  {$l$};
                    \node at (-0.3,1.3) {$l$};
    			\end{tikzpicture}
    			\begin{tikzpicture}
    				
                    \draw[darkorchid] (-1.2, -0.8) -- (1.2, -0.2); 
                    \draw[coralpink] (-1.2, -0.1) -- (1.2, -0.5); 
                    
                    \draw[skyblue] (1,1.2) -- (-1, 1.8); 
                    
                    \draw[dashed] (0.15,-0.5) -- (0.15, 0.5); 
                    \draw[dashed] (-0.15,-0.3) -- (-0.15, 0.5);

                    \draw[dashed] (-0.15, 0.5) -- (0, 1); 
                    \draw[dashed] (0.15, 0.5) -- (0, 1);
                    \draw[dashed] (-0.15, 0.5) -- (0.15, 0.5); 
                    \draw[dashed] (0, 1) -- (0,1.5); 
                    \fill[gray] (-0.15,0.5) circle[radius=2pt];
                    \fill[gray] (0.15,0.5) circle[radius=2pt];
                    \fill[gray] (0,1) circle[radius=2pt];
                    
                    \draw[->] (0,-0.5) -- (0, -0.9);
                    
                    \draw[gray] (-1, -1) -- (2, -1) -- (1, -2) -- (-2, -2) -- cycle; 
                    \draw[darkorchid] (-1, -1.8) -- (1,-1.2); 
                    \draw[skyblue] (-1,-1.2) -- (1, -1.8);
                    \draw[coralpink] (-0.3, -1.1) -- (0.3, -1.9); 
                    
                    \fill (0,-1.5) circle[radius=2pt];
                    \fill (0.15,-0.47) circle[radius=2pt];
                    \fill (-0.15,-0.28) circle[radius=2pt];
                    \fill (0,1.5) circle[radius=2pt];
                    
                    \fill[white] (2,0) circle[radius=2pt];
                    \fill[white] (0,2) circle[radius=2pt];

                    \node at (-0.3,0.1)  {$l$};
                    \node at (0.3,0)  {$l$};
                    \node at (-0.3,1.3) {$l$};
    			\end{tikzpicture}
                \begin{tikzpicture}
    				
                    \draw[coralpink] (-1.2, 0) -- (0.8,0); 
                    \draw[coralpink] (0.8, 0) -- (1.2,0.3); 
                    \draw[coralpink] (0.8,0) -- (1.2, -0.3);
                    \fill (0,1.5) circle[radius=2pt];
                    \draw[skyblue] (-0.8,1.5) -- (1.2, 1.5); 
                    \draw[skyblue] (-0.8,1.5) -- (-1.2, 1.8); 
                    \draw[skyblue] (-0.8, 1.5) -- (-1.2, 1.2);
                    \draw[dashed] (0,0) -- (0, 0.5); 
                    \draw[dashed] (0, 0.5) to[out=30, in=150, relative] (0, 1); 
                    \draw[dashed] (0, 0.5) to[out=60,in=300] (0, 1); 
                    \draw[dashed] (0, 1) -- (0,1.5); 
                    \fill[gray] (0,0.5) circle[radius=2pt];
                    \fill[gray] (0,1) circle[radius=2pt];
                    
                    \draw[->] (0,-0.5) -- (0, -0.9);
                    
                    \draw[gray] (-1, -1) -- (2, -1) -- (1, -2) -- (-2, -2) -- cycle; 
                    \draw[skyblue] (-0.8, -1.5) -- (1.2,-1.49); 
                    \draw[skyblue] (-0.8, -1.5) -- (-1, -1.1); 
                    \draw[skyblue] (-0.8, -1.5) -- (-1.5, -1.9);
                    \draw[coralpink] (-1.2, -1.51) -- (0.8,-1.51); 
                    \draw[coralpink] (0.8,-1.51) -- (1.6, -1.1);
                    \draw[coralpink] (0.8, -1.51) -- (1, -1.9); 
                    \fill (0,-1.5) circle[radius=2pt];
                    \fill (0,0) circle[radius=2pt];
                    \fill (0,1.5) circle[radius=2pt];
                    \fill[white] (2,0) circle[radius=2pt];
                    \fill[white] (0,2) circle[radius=2pt];

                    \node at (-0.2,0.3)  {$l_1$};
                    \node at (-0.2,1.2) {$l_1$};
                    \node at (0.4,0.2)  {$l_2$};
                    \node at (-0.4,1.7) {$l_2$};
    			\end{tikzpicture}
    			\caption{The neighborhood of the cycle is contracted.}
    			\label{fig:dim2defic2case3}
    		  \end{subfigure}
            \caption{Dimension 2, deficiency 2.}
            \label{fig: dim2, defic2}
        \end{figure}
    \end{enumerate}
\end{proposition}

For $\dim X \geq 3$, there is no analogue of the dual polygon associated to a vertex of a tropical curve. On the other hand, the computation of $\nn(\sigma')$ exhibits a recursive structure. Indeed, the automorphism-weighted factor $\nn(\sigma')$ depends only on the neighborhood of the cycle inside the vector space $H_1/H_0$. In particular, if $\dim(H_1/H_0)<\dim(X)$, the computation of $\nn(\sigma')$ reduces to the corresponding problem in lower dimension. If instead $\dim(H_1/H_0)=\dim(X)$, then the tropical curves parametrized by $\sigma'$ contract the cycle $\tropC_0$, and the neighborhood of the cycle necessarily spans all of $(N_X)_\mathbb{R}$, that is,
$
H_1=(N_X)_\mathbb{R}.
$

For example, when $\dim(X)=3$, by Proposition \ref{prop: nbhd-cycle-span}, there are only three new cases to consider:
\begin{enumerate} [label=(\alph*)]
    \item There is a $6$-valent vertex on the cycle (Figure \ref{fig: dim3 nbhd spans}(a)) 
    \item There is a $5$ valent vertex $V_1$ and a $4$ valent vertex $V_2$ on the cycle (Figure \ref{fig: dim3 nbhd spans}(b)). 
    \item There are three $4$-valent vertices on the cycle (Figure \ref{fig: dim3 nbhd spans}(c)).
\end{enumerate}
\begin{figure}[h]
    \renewcommand{\thesubfigure}{\alph{subfigure}}
    \centering
    \begin{subfigure}[t]{0.3\textwidth}
        \centering
        \begin{tikzpicture}
            \fill (0,0) circle[radius=2pt];
            \draw (-1, -0.5) -- (0,0); 
            \draw (-1, 0.6) -- (0,0);
            \draw (0, 0) -- (1,1); 
            \draw (0,0) -- (1, -1);
            \draw[dashed] (0,0)  to[in=50,out=130,loop, style={min distance=12mm}] (0,0);
        \end{tikzpicture}
        \caption{$6$-valent vertex on cycle.}
        
    \end{subfigure}	
    \begin{subfigure}[t]{0.3\textwidth}
        \centering
        \begin{tikzpicture}

            \draw[dashed] (0, 0) to[out=30, in=150, relative] (0, 1); 
            \draw[dashed] (0, 0) to[out=60,in=300] (0, 1); 
            \draw[coralpink]  (-1, 0.7) -- (1, 1.3); 
            \draw[skyblue] (-1.2, 0) -- (0,0); 
            \draw[skyblue] (0, 0) -- (1,0.3); 
            \draw[skyblue] (0,0) -- (1, -0.3);
            \draw[->] (0, -0.3) -- (0, -0.6);
        
            \draw[skyblue] (-1.2, -1.5) -- (0,-1.5); 
            \draw[skyblue] (0, -1.5) -- (1,-1.2); 
            \draw[skyblue] (0,-1.5) -- (1, -1.8);
            \draw[coralpink] (0.3, -0.5) -- (-0.3, -2.5); 
            
            \fill (0,0) circle[radius=2pt];
            \fill (0,1) circle[radius=2pt];
            \fill (0,-1.5) circle[radius=2pt];
        \end{tikzpicture}
        \caption{$5$-valent vertex and $4$-valent vertex on cycle.}
      \end{subfigure}
    \begin{subfigure}[t]{0.3\textwidth}
        \centering
        \begin{tikzpicture}
            
            \draw[darkorchid] (-1.2, -0.1) -- (1.2, 0.1); 
            \draw[coralpink] (-1.2, 0.2) -- (1.2, 0.8); 
            \draw[skyblue] (1,1.2) -- (-1, 1.8);

            \draw[dashed] (0, 1.5) to[out=30, in=150, relative] (0.05, 0); 
            \draw[dashed] (-0.1, 0.5) to[out=30, in=150, relative] (0, 1.5);
            \draw[dashed] (0.05,0) -- (-0.1, 0.5);

            \draw[->] (0, -0.3) -- (0, -0.6);
            
            \draw[darkorchid] (-1, -1.8) -- (1,-1.2); 
            \draw[skyblue] (-1,-1.2) -- (1, -1.8);
            \draw[coralpink] (0.3, -0.5) -- (-0.3, -2.5); 

            \fill (0,-1.5) circle[radius=2pt];

            \fill (0.05,0) circle[radius=2pt];
            \fill (-0.1,0.5) circle[radius=2pt];
            \fill (0,1.5) circle[radius=2pt];
        \end{tikzpicture}
        \caption{Three $4$-valent vertices on cycle.}
      \end{subfigure}
        
      \caption{Dimension 3, deficiency 3, neighborhood spans.}
    \label{fig: dim3 nbhd spans}
\end{figure}
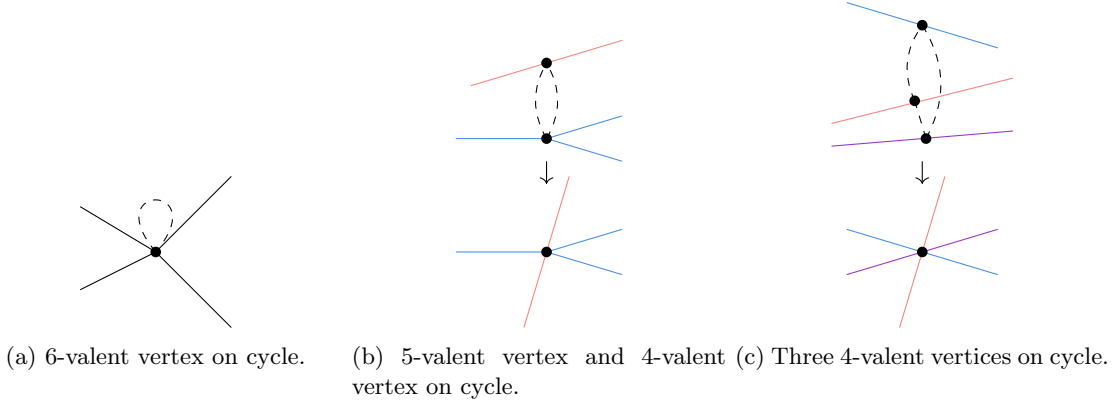

\subsection{Comparisons with previous works}

The problem of tropically counting elliptic curves with fixed $j$-invariant in certain surfaces was studied in \cite{KM09,LR18}. Kerber and Markwig introduced ad hoc multiplicities for cones of tropical genus $1$ curves in $\mathbb{P}^2$ in order to obtain a well-defined count of degree $d$ plane elliptic curves with fixed $j$-invariant passing through $3d-1$ point conditions. They showed that, with this choice of multiplicities, the resulting tropical count agrees with the algebraic count of Pandharipande \cite{Pandh97}.

As explained in \cite{LR18}, which corrected some of these multiplicities, these weights do not arise from algebraic curves. Both \cite{KM09} and \cite{LR18} appeared before the construction of the Ranganathan--Santos-Parker--Wise moduli space $\mathcal{W}_{\Gamma}(X)$. In the present work, the weights of cones in $\Sigma(\mathcal{W}_{\Gamma}(X))$ arise directly from algebraic curves.

We report in Tables~\ref{tab: comparing weights deficiency 1} and \ref{tab: comparing weights deficiency 2} the cases in which our weights differ from those of \cite{KM09} and \cite{LR18}. Since both references assign weights to tropical curves in $(N_X)_\mathbb{R}$, from Equation \eqref{eqn: correspondence bis}, the appropriate quantities to compare are $\mathrm{loop}(\sigma')\, \nn(\sigma')$ for maximal cones $\sigma'$ of $W_\Gamma((N_X)_\mathbb{R})$. In both tables, the vector $n_1$ is always assumed to be primitive.

\begin{table}[H]
    \centering
    \renewcommand{\arraystretch}{1.5}
    \begin{tabular}{@{}lcc@{}}
        \toprule
        &
        \begin{tikzpicture}[scale=0.8]
            \fill[white] (0,1.2) circle[radius=2pt];
            \fill (0,0) circle[radius=2pt];
            \fill (2,0) circle[radius=2pt];
            \fill (1.5,0) circle[radius=2pt];
            \draw (0, 0.05) -- (1.5, 0.05);
            \draw (0, -0.05) -- (1.5, -0.05);
            \draw (1.5, 0) -- (2,0);
            \draw (0, 0) -- (-1, 1); 
            \draw (0, 0) -- (-1, -1); 
            \draw (2, 0) -- (3, 1); 
            \draw (2, 0) -- (3, -1);

            \draw[forestgreen, ->] (0,0.15) -- (-0.5, 0.65);
            \node[forestgreen] at (-0.5,0.8) {$u$};

            \draw[forestgreen, ->] (0.2, 0.15) -- (1, 0.15); 
            \node[forestgreen] at (0.6, 0.35) {$a n_1$};

            \draw[forestgreen, ->] (0.2, -0.15) -- (1, -0.15);
            \node[forestgreen] at (0.6, -0.4) {$b n_1$};
            
            \node at (-0.4,0) {$V_1$};
        \end{tikzpicture} 
        &
        \begin{tikzpicture}[scale=0.8]
            \fill (0,0) circle[radius=2pt];
            \fill (2,0) circle[radius=2pt];
            \fill (0.5,0) circle[radius=2pt];
            \fill (1.5,0) circle[radius=2pt];
            \draw (0, 0) -- (0.5, 0);
            \draw (0.5, 0.05) -- (1.5, 0.05); 
            \draw (0.5, -0.05) -- (1.5, -0.05); 
            \draw (1.5, 0) -- (2, 0); 
            \draw (0, 0) -- (-1, 1); 
            \draw (0, 0) -- (-1, -1); 
            \draw (2, 0) -- (3, 1); 
            \draw (2, 0) -- (3, -1);
            \node at (0.3,0.3) {$l$};
            \node at (1.7,0.3) {$l$};

            \draw[forestgreen, ->] (0.6, 0.15) -- (1.4, 0.15); 
            \node[forestgreen] at (1, 0.35) {$a n_1$};

            \draw[forestgreen, ->] (0.6, -0.15) -- (1.4, -0.15);
            \node[forestgreen] at (1, -0.4) {$b n_1$};
        \end{tikzpicture} 
        \\
        \midrule
        $\lm{\sigma'} \nn(\sigma')$ & $\gcd(a,b) \times (|\det(u, n_1)| - 1)/\Aut(\sigma').$ & $\gcd(a,b)/\Aut(\sigma').$ \\
        $\lm{\sigma'} \nn(\sigma')$ in \cite{KM09} & $\gcd(a,b) \times |\det(u, n_1)|/\Aut(\sigma').$ & 0 \\
        $\lm{\sigma'} \nn(\sigma')$ in \cite{LR18} & Not included & $\gcd(a,b)/\Aut(\sigma').$ \\
        \bottomrule
    \end{tabular}
    \vspace{0.2cm}
    \caption{Comparing weights in dimension $2$ and deficiency $1$.}
    \label{tab: comparing weights deficiency 1}
\end{table}

\begin{table}[H]
    \centering
    \renewcommand{\arraystretch}{1.5}
    \begin{tabular}{@{}lcc@{}}
        \toprule
        &
        \begin{tikzpicture}[scale=0.8]
            \fill[white] (0,1.2) circle[radius=2pt];
            \fill (0,0) circle[radius=2pt];
            \fill (2,0) circle[radius=2pt];
            \draw (0, 0) -- (2,0);
            \draw (0, 0) -- (-1, 1); 
            \draw (0, 0) -- (-1, -1); 
            \draw (2, 0) -- (3, 1); 
            \draw (2, 0) -- (3, -1);
            \draw[dashed] (0,0)  to[in=50,out=130,loop, style={min distance=12mm}] (0,0);

            \node at (-0.4,0) {$V_1$};
        \end{tikzpicture} 
        &
        \begin{tikzpicture}[scale=0.8]
            \fill (0,0) circle[radius=2pt];
            \fill (2,0) circle[radius=2pt];
            \draw (0, 0) -- (2,0);
            \draw (1.5, 0) -- (2, 0); 
            \draw (0, 0) -- (-1, 1); 
            \draw (0, 0) -- (-1, -1); 
            \draw (2, 0) -- (3, 1); 
            \draw (2, 0) -- (3, -1);
            \draw[dashed] (1,0)  to[in=50,out=130,loop, style={min distance=12mm}] (1,0);
            \node at (0.3,0.3) {$l$};
            \node at (1.7,0.3) {$l$};

            \draw[forestgreen, ->] (1,0) -- (0.3, 0);
            \draw[forestgreen, ->] (1,0) -- (1.7, 0);

            \node[forestgreen] at (1.4, -0.3) {$d n_1$};

            \fill (1,0) circle[radius=2pt];
        \end{tikzpicture} 
        \\
        \midrule
        $\lm{\sigma'} \nn(\sigma')$  & $i(\text{dual polygon at $V_1$})$ & $(d-1)/2$ \\
        $\lm{\sigma'} \nn(\sigma')$ in \cite{KM09}  & $\text{Area}(\text{dual polygon at } V_1) - \frac{1}{2}$ & 0 \\
        $\lm{\sigma'} \nn(\sigma')$ in \cite{LR18}   & $i(\text{dual polygon at $V_1$})$ & $(d-1)/2$ \\
        \bottomrule
    \end{tabular}
    \vspace{0.2cm}
    \caption{Comparing weights in dimension $2$ and deficiency $2$.}
    \label{tab: comparing weights deficiency 2}
\end{table}

\subsection{Strategy and methods}

This paper makes use of the full strength of the technical developments of the last decades, including Artin fans, desingularizations of moduli spaces via subdivisions, the correspondence between logarithmic and tropical geometry, tropicalizations of moduli spaces, and logarithmic Gromov--Witten theory.

The proof of Theorem \eqref{thm: tropical correspondence} starts with modifying the map \eqref{eqn: map} to reduce to the situation illustrated in the next proposition, which is of independent interest.

\begin{notation}
    For a logarithmic smooth stack $X$ and a cone $\sigma$ in $\Sigma(X)$, let $V(\sigma)$ 
    denote the substack corresponding to the logarithmic stratum associated with $\sigma$.
\end{notation}

\begin{proposition}\label{prop: technical tool}
    Let $f : X \to Y$ be a logarithmic morphism between two fs log-smooth Deligne-Mumford stacks. 
    Assume that $X$ and $Y$ are pure dimensional, that $Y$ is smooth, and that $f$ is integral \footnote{Under these assumptions the cones of $\Sigma(Y)$ are all isomorphic to $\mathbb{R}_{\ge 0}^n$ for some $n \ge 0$, with integral structure $\mathbb{N}^n \subseteq \mathbb{R}_{\ge 0}^n$. By \cite[Proposition~2.4]{Mirror-GS}, $f$ is integral if and only if the restriction of the map
    $
    \Sigma(X) \to \Sigma(Y)
    $
    to any cone of $\Sigma(X)$ surjects onto a cone of $\Sigma(Y)$.}. Assume also that there is a commutative diagram
    \begin{equation}\label{eqn: artin-fan-diagram}
\begin{tikzcd}
X 
\arrow[r, "f"] 
\arrow[d, "\varphi"']               
& Y
\arrow[d, "\psi"]     \\
\mathcal{A}_{X} 
\arrow[r, "\mathcal{A}(f)"] 
& \mathcal{A}_{Y}
\end{tikzcd}
\end{equation}
    Let $\tau$ be a cone in $\Sigma(Y)$. Then
\[
f^*[V(\tau)] = \sum_{\sigma} [\, N_\tau: \Sigma(f)(N_\sigma) \,]\,[V(\sigma)],
\]
where the sum runs over the cones $\sigma \in \Sigma(X)$ such that 
$\Sigma(f)(\sigma) = \tau$.
\end{proposition}

The commutative diagram~\eqref{eqn: artin-fan-diagram} guarantees that the map $X \to \mathcal{A}_Y$ is flat. Such a diagram exists when $X$ is a logarithmic scheme with Zariski log structure and $Y$ is a scheme~\cite[Proposition~2.8]{ACGSI}, but not in general; see~\cite[Section~4.6.2]{ACGSI}. Consequently, verifying that the morphism from the moduli space of logarithmic maps to the evaluation space $\mathsf{Ev}:=X^n \times \oM_{1,1}$ satisfies the hypotheses of Proposition~\ref{prop: technical tool} is not formal. This verification is carried out in Theorem~\ref{thm: commutative diagram for mapping spaces}.

The proof of Theorem~\ref{thm: calculation of weights} is delicate. Given a maximal cone $\sigma \in W_{\Gamma}(\Sigma(X))$, it is not a priori clear whether the corresponding stratum of $\cW_\Gamma(X)$ is non-empty. Indeed, it may be empty (see Example~\ref{example: weight zero example}). When the stratum is non-empty, determining how many points it contains is a subtle problem. A result of Jonathan Wise \cite{Wise} implies that the number of fine maps of type $\sigma$ is equal to the number of scheme maps $f:(\underline{C},\mathsf{p}) \to \underline{X}$ shaped by $\sigma$ (see Observation~\ref{disc:properties-over-sigma} for the precise meaning of this condition) that admit an fs logarithmic enhancement. Theorem~\ref{thm: realizability genus 1} reduces the existence of such logarithmic enhancements, in arbitrary genus, to the existence of suitable coordinates at the nodes of $\underline{C}$ together with certain meromorphic functions satisfying prescribed conditions. In genus $1$, and in combination with \cite[Theorem 4.4.8]{RSWII}, this criterion becomes particularly powerful and completely resolves Speyer's realizability problem (see \S\ref{sec: realizability}).

Theorem~\ref{thm: realizability genus 1} is effective: it yields a system of equations on the special points of $(\underline{C},\mathsf{p})$ which holds if and only if the map $\underline{f}:(\underline{C},\mathsf{p}) \to \underline{X}$ admits an fs logarithmic enhancement. A careful study of the set of solution of such system identifies configurations of special points on $\underline{C}$ giving rise to isomorphic basic fine maps. In Theorem \ref{thm: calculation of n(sigma)} we explicitly compute the number $n(\sigma)$ of distinct fine basic well-spaced maps to $X$ over a maximal cone $\sigma$ in terms of the polytope associated to $\sigma$. An examination of the $n(\sigma)$ non-isomorphic solutions then determines their automorphism groups (Theorem~\ref{thm: automorphisms}). The resulting automorphism-weighted multiplicity $\nn(\sigma)$ is then related in a precise way to the interior lattice points of the polytope $P(\sigma)$ in Construction~\ref{const: polytope} by Theorem~\ref{thm: calculation of weights}. A further difficulty in this study is to determine, for each of the $n(\sigma)$ basic fine maps, the number of basic fs maps having it as their associated fine map, and to relate this number to the loop multiplicity of the cone $\sigma$. This is explained in Theorems \ref{thm: number of logarithmic enhancements} and \ref{thm: relation indices, loop aut}.

\subsection{Future directions and applications}

This work opens the way to various future directions. We list some of them below.
\begin{enumerate}
    \item[$\bullet$] In upcoming work \cite{CKUpcoming}, we use the correspondence theorem~\ref{thm: tropical correspondence}), together with the weight computation of Theorem~\ref{thm: calculation of weights}, to tropically enumerate plane elliptic curves of fixed degree $d$ passing through $3d-1$ general points and having fixed $j$-invariant. This is the problem considered by Kerber and Markwig \cite{KM09}. Although the multiplicities assigned to the cones differ from those in their work, the explicit final formula we obtain agrees with those of \cite{KM09} and \cite{Pandh97}.
    \item[$\bullet$] In upcoming work \cite{KUpcoming}, the second author uses the correspondence Theorem \ref{thm: tropical correspondence} to compare well-spaced invariants with logarithmic Gromov-Witten invariants \cite{GS2013}. When $X$ is a toric threefold, one obtains a Getzler--Pandharipande type equation (see \cite[Theorem 6.1]{Ge97}) relating the two. 
    \item[$\bullet$] The problem of determining the enumerative (as opposed to virtual) count $\tau_d^{\mathbb{P}^r}(\beta_1,\ldots,\beta_n)$ of elliptic curves in $\mathbb{P}^r$ of degree $d$ satisfying incidence conditions $\beta_1,\ldots,\beta_n$ was studied by Ionel \cite{Ionel}. There, the enumerative count is recovered from the corresponding Gromov--Witten invariant by carefully removing the excess contributions arising in the virtual count. In \cite[Proposition 0.2]{Ionel}, Ionel gives a formula expressing the number $\tau_d^{\mathbb{P}^3}(\mathsf{p}^a,\mathsf{l}^b)$ of elliptic curves in $\mathbb{P}^3$ passing through $a$ points and meeting $b$ lines in terms of genus-$0$ Gromov--Witten invariants of $\mathbb{P}^3$. However, her formula appears to be incorrect.

    The simplest example occurs when $d=3$, $a=5$, and $b=1$. In this case,
    $
        \tau_3^{\mathbb{P}^3}(\mathsf{p}^5,\mathsf{l})
    $
    clearly vanishes, since every degree-$3$ curve in $\mathbb{P}^3$ is contained in a plane. On the other hand, Ionel's formula gives
    \[
        \tau_3^{\mathbb{P}^3}(\mathsf{p}^5,\mathsf{l})
        =
        \frac{2}{3}\sigma_3^{\mathbb{P}^3}(\mathsf{p}^5,\mathsf{l}^2)
        -\frac{1}{3}
        \Bigl(
        2\,\sigma_1^{\mathbb{P}^3}(\mathsf{p},\mathsf{l}^2)
        \cdot
        \sigma_2^{\mathbb{P}^3}(\mathsf{p}^4,\mathsf{l}^0)
        +
        \sigma_2^{\mathbb{P}^3}(\mathsf{p}^3,\mathsf{l}^2)
        \cdot
        \sigma_1^{\mathbb{P}^3}(\mathsf{p}^2,\mathsf{l}^0)
        \Bigr),
        \]
    where $\sigma_d^{\mathbb{P}^3}(\mathsf{p}^a,\mathsf{l}^b)$ denotes the number of genus-$0$ stable maps of degree $d$ through $a$ points and meeting $b$ lines. These numbers were computed by Getzler \cite[Table~2]{Ge97}, and one finds
    \[
        \sigma_3^{\mathbb{P}^3}(\mathsf{p}^5,\mathsf{l}^2)=5,\qquad
        \sigma_1^{\mathbb{P}^3}(\mathsf{p},\mathsf{l}^2)=1,\qquad
        \sigma_2^{\mathbb{P}^3}(\mathsf{p}^4, \mathsf{l}^0)=0,\qquad
        \sigma_2^{\mathbb{P}^3}(\mathsf{p}^3,\mathsf{l}^2)=1,\qquad
        \sigma_1^{\mathbb{P}^3}(\mathsf{p}^2, \mathsf{l}^0)=1.
    \]
    Substituting these values into Ionel's formula yields
    $
        \tau_3^{\mathbb{P}^3}(\mathsf{p}^5,\mathsf{l})=3,
    $
    contradicting the fact that the enumerative count is zero.

    With the Correspondence Theorem~\ref{thm: tropical correspondence} in hand, we are currently developing a tropical approach to the computation of $\tau_d^{\mathbb{P}^3}(\mathsf{p}^a,\mathsf{l}^b)$ that will provide a corrected enumerative formula.

    \item[$\bullet$] It would be interesting to extend the correspondence theorem to incidence conditions arising from classes on the moduli stack $\overline{\mathcal{M}}_{1,n}$ beyond the point class on $\overline{\mathcal{M}}_{1,1}$. For example, one may hope for a complete tropical understanding of the genus-one reduced descendant theory. In genus zero, this problem has been resolved by Mandel--Ruddat \cite{RT}.% Trnot yet fully understood, despite the fact that $\overline{M}_{0,n}$ is closely related to a toric variety.

    %The methods developed in the present paper are well suited to this setting and should lead to explicit tropical correspondence theorems incorporating incidence conditions pulled back from moduli spaces of curves in both genus zero and genus one. 
    
    \item[$\bullet$] Extending the results of this paper to higher genus would first require the construction of a higher-genus analogue of $\mathcal{W}_{\Gamma}(X)$ for general toric varieties. See \cite{BC-g2} for the case of maps to $\P^r$.
\end{enumerate}

\subsection{Conventions}

We work in the category of schemes of finite type over $\C$. 

We denote log schemes by $X = (\underline{X}, \mathcal{M}_X)$, and we generally avoid underlining ordinary schemes unless we wish to emphasize the absence of a log structure. Similarly, for a morphism $f$ of log schemes, we sometimes write $\underline{f}$ for the underlying morphism of schemes.

The ghost sheaf (also called the characteristic sheaf) associated to a log structure $\cM$ is denoted
$
\oM := \cM / \mathcal{O}_X^{\times}.
$
We write $\cM$ multiplicatively and $\oM$ additively.

Unless otherwise stated, logarithmic schemes will always be fine and saturated over the trivial log point.
We will also sometimes consider fine log schemes over the trivial log point, and we will make this explicit when needed.

For a monoid $Q$, we write
$
Q^{\vee} := \operatorname{Hom}(Q, \mathbb{N})
$
for its dual in the category of monoids, $Q^{\mathrm{gp}}$ for the associated abelian group, and
$
Q^{*} := \operatorname{Hom}(Q, \mathbb{Z}).
$

\subsection*{Acknowledgments}

We are indebted to Dhruv Ranganathan for introducing us to the moduli stack $\mathcal{W}_{\Gamma}(X)$ of well-spaces genus $1$ maps and for countless hours of discussions on the geometry of this space and the content of this paper.

We are also grateful to Hannah Markwig, who explained to the first author many years ago the content of her paper \cite{KM09}. She also explained that her multiplicities were explicitly designed to yield a well-defined count, although it was not clear that they correspond to counts of algebraic curves. %We are further grateful to her for encouraging us to pursue this direction during her visit to Cambridge University in April 2025. 

We would like to express our gratitude to Xuanchun Lu for many helpful conversations on logarithmic geometry and various aspects of this work. 
Finally, we have benefited greatly from discussions with Dan Abramovich, Francesca Carocci, Renzo Cavalieri, Erin Dawson, Mark Gross, Samouil Molcho and Jonathan Wise. 

The first presentation of the results proven here was given at the Oberseminar Kombinatorische Algebraische Geometrie in Tubigen with very helpful discussions with Hannah Markwig and Erin Dawson.

A.~C.\ is supported by SNF grant P500PT-222363. S.~K.'s doctorial programme is funded by St John's College, Cambridge.

\section{Cones and Artin fans}

\subsection{Weighted generalized cone complexes}\label{sec: cones}

Let $\mathbf{Cones}$ denotes the category of cones. Objects in  $\mathbf{Cones}$ are pairs $\sigma = (\sigma_{\mathbb{R}}, N)$, where $N \simeq \mathbb{Z}^n$ is a lattice and 
\[
    \sigma_{\mathbb{R}} \subset N_{\mathbb{R}} := N \otimes_{\mathbb{Z}} \mathbb{R}
\]
is a top-dimensional strictly convex rational polyhedral cone.  In particular, every cone $\sigma$ comes with a lattice structure. A morphism of cones $\varphi: \sigma_1 \to \sigma_2$ is the data of a homomorphism of the underlying lattices taking $\sigma_1$ to $\sigma_2$. Such a morphsim is a face morphism if it identifies $\sigma_1$ with a face of $\sigma_2$.

\begin{notation}\label{notation: Nsigma}
    If we need to specify that $N$ is associated to $\sigma$ we write $N_\sigma$ instead.
\end{notation}

\begin{definition}\label{def: cones}
    A generalized cone complex is a topological space with a presentation as the colimit of an arbitrary finite diagram in the category $\mathbf{Cones}$ with all morphisms being face morphisms. 
\end{definition}

\begin{definition}
    We say that a generalized cone complex is pure-dimensional if all its maximal cones have the same dimension.
\end{definition}

Generalized cone complexes form a category, which we denote by $\mathsf{GenCones}$, whose morphisms are continuous maps that restrict on each cone to maps of cones.

\begin{definition}
    Let $\Delta$ be a pure-dimensional generalized cone complex. A weight function on $\Delta$ is the data of an integer-valued function from the maximal cones of $\Delta$.
\end{definition}

\begin{definition}
    Let $\Delta'$ be a generalized cone complex. A subdivision of $\Delta'$ is a generalized cone complex $\Delta$ equipped with a morphism
    $$ 
    \Delta \longrightarrow \Delta'
    $$
    such that:
    \begin{enumerate}
        \item[$\bullet$] The induced map on the underlying rational polyhedral spaces $|\Delta| \to |\Delta'|$ is a bijection. 
        \item[$\bullet$] For every cone $\sigma$ in $\Delta$ mapping to a cone $\sigma'$ in $\Delta'$, the corresponding map of integral lattices $N_{\sigma} \to N_{\sigma'}$ is an injective homomorphism with finite cokernel. 
    \end{enumerate}
\end{definition}

\subsection{Artin fans and logarithmic blow-ups}

Let $(Z,\mathcal{M}_Z)$ be a Zariski fs log scheme of finite type, and let 
$
\oM_{Z,\eta} = \mathcal{M}_{Z,\eta}/\mathcal{O}_{Z,\eta}^*
$
be the characteristic monoid at a generic point $\eta$ of a stratum of $Z$. 
Define
$
(\oM_{Z,\eta})^* := \mathrm{Hom}(\oM_{X,\eta}, \Z)$ and 
$ 
(\oM_{Z,\eta})_\R^\vee := \mathrm{Hom}(\oM_{Z,\eta}, \R_{\ge 0}).
$

The cone associated to $\eta$ is 
\[
\sigma_\eta := \bigl((\oM_{Z,\eta})_\R^\vee, (\oM_{Z,\eta})^*\bigr),
\]
and the generalized cone complex $\Sigma(Z)$ of $Z$ is obtained by gluing these cones along face maps corresponding to inclusions of strata. This construction generalizes to algebraic log stacks. See \cite[Section 2.1.4]{ACGSI}.

An Artin stack logarithmically étale over $\mathrm{Spec}(\C)$ is called Artin fan. From $\Sigma(X)$, one can construct the Artin fan $\mathcal{A}_X$ of $X$ as follows. 
For a cone $\sigma$, let
$
P = \sigma^\vee= \mathrm{Hom}(\sigma \cap N_\sigma, \N)
$
be the associated monoid, that is,
\[
\sigma^\vee = \{ m \in N_\sigma^* \mid \langle m, n \rangle \geq 0 \text{ for all } n \in \sigma \}.
\]
Set
\[
\mathcal{A}_\sigma := \bigl[ \mathrm{Spec}\, \mathbb{C}[P] \big/ \mathrm{Spec}\, \mathbb{C}[P^{\mathrm{gp}}] \bigr].
\]

Then the Artin fan of $Z$ is obtained as the colimit
\[
\mathcal{A}_Z := \varinjlim_{\sigma \in \Sigma(Z)} \mathcal{A}_\sigma
\]
in the category of sheaves over the Olsson's stack $\mathsf{Log}$ of fine logarithmic structures \cite{Olsson03}.

A subdivision of $\widetilde{\Sigma(Z)}$ of $\Sigma(Z)$ naturally induces a morphism of Artin fans $\widetilde{\mathcal{A}}_Z \longrightarrow \mathcal{A}_Z$. 
\begin{definition}
    A logarithmic blowup $\widetilde{Z} \to Z$ is  the fs base change of a morphism of Artin fans arising from a subdivision  $\widetilde{\Sigma(Z)} \to \Sigma(Z)$. 
\end{definition}

\section{Tropical well-spaced genus $1$ maps} \label{sec: tropical definitions}

    In this section, we review the definition of well-spaced tropical curves from~\cite{RSWII}. Given that we will use several variants of moduli spaces of tropical maps, this section primarily serves to establish the notation for subsequent sections.

    \subsection{Families of tropical maps to $\Sigma(X)$}

    We follow \cite{ACGSI} for the definition of tropical curves and maps. 

    \begin{definition}\cite[Definition 2.16]{GeomAC-ACG}
        A graph $G$ consists of the following data: 
        \begin{enumerate}
            \item (the set of vertices) a finite non-empty set $ V(G)$ ; 
            \item (the set of half-edges) a finite set $ L(G)$ ; 
            \item (the edges) an involution $\iota: L(G) \rightarrow L(G)$; and 
            \item (adjacency) a partition of $L(G)$ indexed by $V(G)$ - that is  $L(G)_V \subset L(G)$ such that $L(G) = \cup_V L(G)_V$ and $L_V \cap L_W = \emptyset$ for vertices $V \neq W$ . 
        \end{enumerate}
    \end{definition}

    A pair of distinct elements of $L$ interchanged by $\iota$ is called a \emph{edge} of the graph. Denote the set of edges by $E(G)$. An element of $L$ mapped to itself under $\iota$ is called a \emph{infinite end}. Denote the set of infinite ends by $L^e(G)$. 
    For each vertex $V$, the number of adjacent half-edges $|L_V|$ is the \emph{valency} of the vertex $V$.

    \begin{definition}
        A decorated graph $(G, g, p, z)$ consists of a (finite) connected graph $G$, a genus function $g: V(G) \rightarrow \mathbb{Z}_{\geq 0}$, marking functions $p: \{1, \ldots, n\} \rightarrow L^e(G)$ and $z: \{1, \ldots, m\} \rightarrow L^e(G)$ such that $p \sqcup z: \{1, \ldots, n\} \sqcup \{1,\ldots, m\} \rightarrow L^e(G)$ is a bijection. 
    \end{definition}

    Let $N = n + m$ denote the \emph{number of markings} of $(G, g, p, z)$. We distinguish between the $p$- and $z$-markings because, in what follows, the $p$-markings will play the role of free markings of a tropical curve, while the $z$-markings will correspond to markings with prescribed contact orders.
    
    \begin{definition}
        \cite[Definition 2.19]{ACGSI}
		A family of tropical curves $\tropC/\sigma = (G, g, p, z, \ell)$ over a cone $\sigma$ consists of a decorated graph $(G, g, p, z)$, and a length function $\ell:E(G) \rightarrow \operatorname{Hom}(\sigma \cap N_{\sigma}, \mathbb{N}) \smallsetminus \{0\}$ from the edge set of $G$. 
	\end{definition}

    \begin{construction}\cite[Construction 2.20]{ACGSI} \label{const: tropCspace}
        Given a family of tropical curves $\tropC/\sigma = (G, g, p, z, \ell)$ over a cone $\sigma$, we can build a morphism of cone complexes $\pi: \hat{C}/ \sigma \rightarrow \sigma$ as follows. Associate:
        \begin{enumerate}
            \item To each vertex $V \in V(\tropC)$ one copy $\sigma_V$ of $\sigma$. 
            \item To each edge $e \in E(\tropC)$ the cone
            $\sigma_e = \{(s, \lambda) \in \sigma \times \mathbb{R}_{\geq 0} \mid \lambda \leq \ell(e)(s)\}.$
            \item To each infinite end $z$ the cone $\sigma_{z} = \sigma \times \mathbb{R}_{\geq 0}$.
        \end{enumerate}
        The cone $\sigma_e$ has two facets, each isomorphic to $\sigma$ via the projection to the first factor. Suppose $e$ is an edge between $V$ and $V'$. The inclusions $s \mapsto (s,0)$ and $s \mapsto (s, \ell(s))$ define face morphisms $\sigma_{V}, \sigma_{V'} \rightarrow \sigma_e $. For each infinite end $z$ adjacent to $V$, we have a face morphsim $\sigma_V \rightarrow \sigma_{z}$ given by the inclusion $\sigma \times \{0 \} \rightarrow \sigma_{z}$. The cone complex $\hat{\tropC}$ is defined by this directed system. The morphism $\pi$ is defined on $\sigma_{e}$ and $\sigma_{z}$ by projection to the first factor. 
    \end{construction}

    \begin{definition}\label{def: family of trop maps to Sigma}\cite[Definition 2.21]{ACGSI}
        A family of tropical maps to $\Sigma(X)$ over a cone $\sigma$ is a family of tropical curves $\tropC/\sigma$ over $\sigma$ with associated cone complex $\hat{\tropC}/\sigma$ together with a morphism of cone complexes $h: \hat{\tropC}/\sigma \rightarrow \Sigma(X)$.

    Given such a family, we can extract the following discrete data: 
    \begin{enumerate}
        \item (Image cones.) For vertex, edge, or infinite end $x$ of $\tropC/\sigma$, let $\sigma_x$ be the cone of $\hat{\tropC}/\sigma$ associated to $x$. Let $\vartheta_x$ be the minimal cone of $\Sigma(X)$ which contains $h(\sigma_x)$. 
        \item (Contact orders.) Suppose $\zeta$ is a half-edge of $G$ adjacent to a vertex $V$. If $\zeta$ is an infinite end, say $\zeta=z$, define $u_{\zeta}= u_{z} := h((0,1))$ where $ (0,1) \in N_{\sigma_\zeta} = N_\sigma \times \mathbb{R}$. If $\zeta$ is contained in an edge $e$, orient $e$ away from $V$ and define $u_{\zeta} := h((0,1))$ where $ (0,1) \in N_{\sigma_e} = N_\sigma \times \mathbb{R}$. We will also use the notation $u_{\vec{e}}$ for $u_{\zeta}$.
    \end{enumerate}

    \end{definition}

    \begin{remark}
        For each $V \in V(G)$, there is a natural section $\sigma \rightarrow \hat{\tropC}/\sigma$, $s \mapsto V(s) = s \in \sigma_V$. Suppose $\vec{e}: V \to V'$ is an oriented edge. The element $u_{\vec{e}} \in N_X$ satisfies 
        $$h(V(s)) - h(V'(s)) = \ell(e)(s) \cdot u_{\vec{e}}$$
    \end{remark}

    \begin{definition}\cite[Definition 2.23]{ACGSI}
        The combinatorial type of a family of tropical maps $h: \tropC/\sigma \rightarrow \Sigma(X)$ consists of the information obtained by dropping $\ell$. That is $(G, g, p, z)$, the underlying decorated graph, as well as the the image cones of the vertices and contact orders $(\vartheta_V, u_{\zeta})$.
    \end{definition}

    \subsection{Moduli spaces of tropical maps to $\Sigma(X)$ and $(N_X)_{\mathbb{R}}$}\label{sec: trop maps}

    In this section, we recall the construction of various moduli spaces of tropical maps and establish the notation used throughout this work. 

\begin{definition}
    \cite[Definition 2.2.1]{RSWII}
    A tropical curve $\tropC = (G, g, p, z, \ell)$ consists of a decorated graph $(G, g, p, z)$ and a \emph{length function} $\ell: E(G) \rightarrow \mathbb{R}_{>0}$ defined on the edge set of $G$. 
\end{definition}

    \begin{notation}
        For $\tropC = (G, g, p, z, \ell)$, let $E(\tropC) := E(G), L(\tropC) := L(G)$, and $V(\tropC) := V(G)$. 
    \end{notation}

The length function may be viewed as a function on half-edges $\ell: L(G) \setminus L^e(G) \rightarrow \operatorname{Hom}(\sigma \cap N_{\sigma}, \mathbb{N})$ such that $\ell(\zeta) = \ell(\iota(\zeta))$ for all $\zeta \in L(G)$. 

An \emph{isomorphism} of tropical curves $(G, g, p, z, \ell) \rightarrow (G', g', p', z', \ell')$ is a pair of bijections 
$$ \varphi_V: V(G) \rightarrow V(G') \quad \text{and} \quad \varphi_L: L(G) \rightarrow L(G') $$
such that:
\begin{enumerate}
    \item the maps induce an isomorphism of graphs $\varphi: G \rightarrow G'$; and 
    \item the decorations are preserved: $g = g' \circ \varphi_V$, $p' = \varphi_L \circ p$, $z' = \varphi_L \circ z$, and $\ell = \ell' \circ \varphi_L$. 
\end{enumerate}

\begin{construction}\label{const: tropCmetricspace}

To a tropical curve $\tropC$, we associate a metric space $\hat{\tropC}$ via the following construction. First, choose an arbitrary orientation for the edges of the underlying graph $G$ of $\tropC$ (the resulting metric space is independent of this choice).

\begin{itemize}
    \item For each vertex $V \in V(G)$, we introduce a corresponding point $V$ in $\hat{\tropC}$.
    
    \item For each oriented edge $\vec{e}: V \to V'$ of length $\ell(e)$, let $I_e = [0, \ell(e)]$, and identify $0 \in I_e$ with $V$ and $\ell(e) \in I_e$ with $V'$.
    
    \item For each infinite end $z$ adjacent to a vertex $V$, let $I_z = [0, \infty)$, and identify $0 \in I_z$ with $V$.
\end{itemize}

The topological space $\hat{\tropC}$ is the quotient of the disjoint union of these points and intervals under the given identifications. It is naturally equipped with the path metric induced by the Euclidean metric on each interval.
\end{construction}

    \begin{definition}
        The genus of a tropical curve $\tropC = (G, g, p, z, \ell)$ is $g(\tropC) = h^1(\hat{\tropC}) + \sum_{V \in V(G)}g(V)$.
    \end{definition}

    Let $\delta=(\delta_1,\ldots,\delta_m)$ be an ordered collection of vectors in $N_X$ such that $\sum_{j=1}^m \delta_j=0$. 
    
	\begin{definition}\label{def: curves in NX}
	      A parametrized tropical curve $[h: \tropC \rightarrow (N_X)_{\mathbb{R}}]$ of genus $g$ and degree $\delta$ is a piecewise linear map $h$ from $\hat{\tropC}$ (as in Construction \ref{const: tropCmetricspace}) to $(N_X)_{\mathbb{R}}$ satisfying the following conditions:
\begin{enumerate}
    \item (Integer slopes) Let $\zeta$ be a half-edge adjacent to a vertex $V$. If $\zeta$ is contained in a finite edge $e$, we orient the edge as $\vec{e} \colon V \to V'$. Otherwise, $\zeta$ is an infinite end $z$. The function $h$ restricts to an affine linear map on the corresponding interval--either $[0, \ell(e)]$ or $[0, \infty)$--given by $h(s) = h(0) + s u_{\zeta}$. We require the slope $u_\zeta$ to lie in $N_X$. When $\zeta$ is the initial part of the oriented edge $\vec{e}$, we often write $u_\zeta = u_{\vec{e}}$, and when $\zeta = z$ is an infinite end, we write $u_\zeta = u_z$.
    
    \item (Balancing condition) For every vertex $V$ of $\tropC$ with adjacent half-edges $z_1, \ldots, z_k$, we have $\sum_{i=1}^k u_{z_i} = 0$.
    
    \item (Degree condition) For $i = 1, \ldots, n$, we have $u_{p(i)} = 0$, and for $j = 1, \ldots, m$, we have $u_{z(j)} = \delta_j$.
\end{enumerate}

	\end{definition}

    \begin{definition} \label{def: curves in Sigma}
     A parametrized tropical curve in $\Sigma(X)$ of genus $g$ and degree $\delta$, denoted $[h: \tropC \rightarrow \Sigma(X)]$, is a parametrized tropical curve $[h: \tropC \rightarrow (N_X)_{\mathbb{R}}]$ of genus $g$ and degree $\delta$ satisfying the additional condition that every edge is mapped entirely into a single cone of $\Sigma(X)$. \footnote{Sometimes in the literature, it is assumed that the vectors in $\delta$ are parallel to some ray in $\Sigma(X)$. This is not restrictive, as given any $\delta$ one can always find a smooth blowup of $X$ with this property. However, this assumption is not necessary, and we will avoid making it.}
\end{definition}

    When we wish to emphasize the marked points, we write $[h: (\tropC, \mathsf{p}) \rightarrow (N_X)_{\mathbb{R}}]$ or $[h: (\tropC, \mathsf{p}) \rightarrow \Sigma]$ for a parametrized tropical curve in $(N_X)_{\mathbb{R}}$ or $\Sigma(X)$, respectively.

    \begin{remark}
        Consider a family of tropical curves $\tropC = (G, g, p, z, \ell)$ over $\sigma$. For any point $s \in \operatorname{Int}(\sigma) \cap N_{\sigma}$, we obtain a tropical curve $\mathsf{C}$ by associating to each edge $e$ the length $\ell(e)(s) \in \mathbb{N}$. The preimage of $s$ under the map $\pi: \hat{\tropC}/\sigma \rightarrow \sigma$ of Construction \ref{const: tropCspace} is precisely the metric space $\hat{\mathsf{C}}$ associated to $\mathsf{C}$ in Construction \ref{const: tropCmetricspace}. Moreover, given a family of tropical maps $h: \hat{\tropC}/\sigma \to \Sigma(X)$, the restriction of $h$ to the fiber over $s$ is a parametrized tropical curve in $\Sigma(X)$ as in Definition \ref{def: curves in Sigma}. Finally, under these identifications, the two notions of direction vectors $u_{\vec{e}}$ and $u_z$ in Definitions \ref{def: family of trop maps to Sigma} and \ref{def: curves in NX} coincide.
    \end{remark}

    \begin{definition}
        A parametrized tropical curve in $(N_X)_{\mathbb{R}}$ is stable if every genus $0$ vertex is at least $3$-valent. A parametrized tropical curve in $\Sigma(X)$ is stable if for every $2$-valent genus $0$ vertex $V$, the set $h(\operatorname{Star}(V))$ is not contained in the relative interior of a single cone of $\Sigma(X)$.
    \end{definition}

    Note in particular that a stable parametrized tropical curve in $\Sigma(X)$ is not necessarily stable when viewed as a tropical curve in $(N_X)_{\mathbb{R}}$. However, by forgetting the $2$-valent vertices (and replacing their two adjacent edges with a single edge whose length is the sum of the original lengths), we obtain a stable parametrized tropical curve in $(N_X)_{\mathbb{R}}$. We will refer to this as the \emph{associated stable parametrized tropical curve} in $(N_X)_{\mathbb{R}}$.

     \begin{remark} 
        The notions of deficiency and overvalency given in Definition \ref{def:deficandov} for a parametrized tropical curve in $(N_X)_{\mathbb{R}}$ are in fact well-defined for curves  $h \colon (\mathsf{C},\mathsf{p}) \to\Sigma(X)$ in $\Sigma(X)$, and the deficiency and overvalency of $h$ are the same as those of its associated stable parametrized tropical curve in $(N_X)_{\mathbb{R}}$. Note that, by definiton, the overvalency of $h$ is a sum over vertices of valency at least $3$.
    \end{remark}

     \begin{definition} \label{def: dv}
         Given the tropical type of a parametrized tropical curve in $\Sigma(X)$, we say that a vertex is unconstrained if it is mapped to the interior of a maximal-dimensional cone of $\Sigma(X)$, and constrained otherwise.

    Let $V$ be a vertex constrained to a codimension-one cone $\tau \in \Sigma(X)$ spanning a hyperplane $H \subseteq (N_X)_\mathbb{R}$. Let $n_1, \ldots, n_{r-1}$ be a $\mathbb{Z}$-basis of $H \cap N_X$, and let $n_r \in N_X$ be an element completing this to a $\mathbb{Z}$-basis of $N_X$. Let $m_1, \ldots, m_r \in M_X$ be the corresponding dual basis. We define $d_V$ to be the absolute value 
\[
d_V = |m_r(u_{\vec{e}})|,
\]
where $\vec{e}$ is any choice of orientation of one of the two edges $e$ adjacent to $V$.\footnote{The quantity $d_V$ in Definition~\ref{def: dv} is defined only for vertices constrained in cones of $\Sigma(X)$ of codimension~$1$. The reason why this is enough for our purposes is that if $\sigma$ is a maximal cone in the cone complex $W_\Gamma(\Sigma(X))$ of tropical well-spaced maps to $\Sigma(X)$ (to be introduced in \S\ref{sec: well-spaced tropical maps}), then all constrained vertices of the tropical curves parametrized by $\sigma$ generically lie in a single codimension~$1$ cone of $\Sigma(X)$.} Note that this definition is independent of the choice of edge, the orientation $\vec{e}$, and the choice of the basis elements.\footnote{For a basic well-spaced logarithmic map $(C/S,\mathsf{p}, f, \rho)$ in $\cW_\Gamma(X)$ (see Notation \ref{not: well-spaced maps}) with basic monoid corresponding to tropical curves parametrized by a cone $\sigma$ of $W_\Gamma(\Sigma(X))$, let $C_V$ be the component corresponding to a vertex $V$ constrained in a codimension $1$ cone $\tau$ of $\Sigma(X)$. Then the restriction $\underline{f}|_{C_V}$ is a multiple cover of degree $d_V$ of the toric $\mathbb{P}^1$ in $X$ corresponding to $\tau$. See Observation \ref{disc:properties-over-sigma}.}
    \end{definition}

    \begin{definition}
        An isomorphism of parametrized tropical curves $\Phi: [h:\tropC \rightarrow (N_X)_{\mathbb{R}}] \rightarrow [h':\tropC'\rightarrow (N_X)_{\mathbb{R}}]$ is an isomorphism of tropical curves $\tropC \rightarrow \tropC'$ such that for the induced map $\varphi: \hat{\tropC} \rightarrow \hat{\tropC'}$ makes the diagram
    \begin{center}
        % https://tikzcd.yichuanshen.de/#N4Igdg9gJgpgziAXAbVABwnAlgFyxMJZABgBpiBdUkANwEMAbAVxiRAB12cAnCNAYRABfUuky58hFAEZyVWoxZtOPPvwDkw0SAzY8BIrOnz6zVog7sAtnRwALAEYPgAJSEA9bsPkwoAc3giUAAzXiskMhAcCCRZBTNldnpuNDssLRCw2OpopAAmalMlCzsMkFCIcMRI3MQC+OKQO01qBjoHGAYABXF9KRBuLD87HG8hIA
        \begin{tikzcd}
        \hat{\tropC} \arrow[r, "\varphi"] \arrow[rd, "h'"'] & \hat{\tropC'} \arrow[d, "h"] \\
                                                      & (N_X)_\R         
        \end{tikzcd}
    \end{center}
    commutative. Parametrized tropical curves in $\Sigma(X)$ are isomorphic if they are isomorphic when viewed as parametrized tropical curves in $(N_X)_{\mathbb{R}}$.
    \end{definition}

    \begin{definition} \cite[Definition 4.1.1]{RSWII}  \label{def: comb-type}
        The combinatorial type of a parametrized tropical curve $[h:\tropC \rightarrow (N_X)_{\mathbb{R}}]$ consists of the data of the underlying decorated graph $(G, g, p, z)$ together with the direction vectors $u_\zeta$ for every half-edge $\zeta$. 
        
    The combinatorial type of a parametrized tropical curve in $\Sigma(X)$ consists of this same data, along with the assignment of the cone $\vartheta_V$ in $\Sigma(X)$ for each vertex $V$ of $\tropC$, such that $h(V)$ lies in the relative interior of $\vartheta_V$.
    \end{definition}

    Let $\Gamma  = (g,n,m,\delta)$ be the data of a genus, number of free and constrained markings and degree. There is a tropical moduli space $M_\Gamma(\Sigma(X))$ (resp. $M_\Gamma((N_X)_\R)$) of stable parametrized tropical curves in $\Sigma(X)$ (resp. $(N_X)_\R)$) of genus $g$ with $n$ free markings and degree $\delta$, which has the structure of a generalized cone complex. Each cone corresponds to a tropical type. The lattice points are those tropical curves with integer edge lengths $\ell(e)$. Finally, each cone $\sigma$ of $M_\Gamma(\Sigma(X))$ naturally yields a family of tropical maps to $\Sigma(X)$ in the sense of Definition \ref{def: family of trop maps to Sigma}.

    Recall the definitions of overvalence and deficiency given in Definition \ref{def:deficandov}. 
    The expected dimension of a cone in $M_\Gamma(\Sigma(X))$ (or $M_\Gamma((N_X)_\R)$ is $$n+ m+(r-3)(1-g)- \ov{\sigma} .$$ 
    If the image under the tropical map $h$ of the cycle $\tropC_0 \subseteq \mathsf{C}$ lies in a proper subspace of $(N_X)_{\mathbb{R}}$, then the corresponding cone will have dimension $\defic{\sigma}$ larger than expected dimension.

    There is a morphism of cone complexes 
    \begin{equation} \label{eqn: forgetfan}
    M_\Gamma(\Sigma(X)) \rightarrow M_\Gamma((N_X)_{\mathbb{R}})
    \end{equation}
    defined by sending a stable tropical curve in $\Sigma(X)$ to its associated stable tropical curve in $(N_X)_{\mathbb{R}}$. This is a subdivision, but there may be a non-trivial lattice index, as shown in Example \ref{exa:dilation}.

    \begin{example} \label{exa:dilation}
        Consider the tropical curve $(h,\tropC)$ in $\Sigma(\mathbb{P}^2)$ shown in black in Figure \ref{fig:difflat}. We obtain a tropical curve $(h',\tropC')$ in $(N_{\P^2})_{\mathbb{R}}$  by removing the bivalent vertex. Suppose that $\vec{e}: V_1 \to V_2$ has direction vector $u_{\vec{e}}=(2,0)$. For $(h,\tropC)$ to have integer edge lengths, the first coordinate of $V_1$ must be even. There is no such constraint for $(h',\tropC')$. 

        \begin{figure}[h]
    			\centering
    			\begin{tikzpicture}
                    %fan of P2
    				\draw[burntsiena] (0, 0) -- (2, 0);
    				\draw[burntsiena] (0,0) -- (0,2); 
                    \draw[burntsiena] (0,0) -- (-1.4,-1.4);

                    %tropical curve
                    \fill (-1,1) circle[radius=2pt];
                    \fill (0,1) circle[radius=2pt];
                    \draw (-1,1) -- (0,1); 
                    \draw[->] (0,1) -- (2,1);
                    \draw[->] (-1,1) -- (-1,2); 
                    \draw[->] (-1,1) -- (-3, -1);

                    \node at (-1.3,1.3) {$V_1$};
                    \node at (0.3,1.3) {$V_2$};
                    \node at (-0.5, 0.7) {$(2,0)$};

                    %markings
                    \draw[dashed] (-1,1.5) -- (-0.5, 2);
                    \draw[dashed] (-2, 0) -- (-2, -0.8);
    				
    			\end{tikzpicture}
    			\caption{A tropical curve mapping to the fan of $\mathbb{P}^2$.}
    			\label{fig:difflat}
    		  \end{figure}
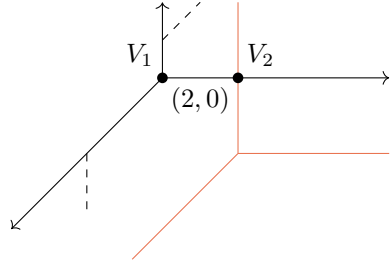
    \end{example}

There are tropical evaluation and stabilization maps fitting into the following commutative diagram:

\begin{equation}\label{eqn: full tropical ev map}
    \begin{tikzcd}
        M_\Gamma(\Sigma(X)) 
        \arrow[r] 
        \arrow[d]                
        & \Sigma(X)^n \times M_{g,n}^{\trop}
        \arrow[d]      \\
        M_\Gamma((N_X)_{\mathbb{R}}) 
        \arrow[r] 
        & (N_X)_{\mathbb{R}}^n \times M_{g,n}^{\trop}
    \end{tikzcd}
\end{equation}

The evaluation maps to $\Sigma(X)$ and $(N_X)_{\mathbb{R}}$ are defined by sending a tropical map $h$ to the images of the markings $h(p_i)$ for $i = 1, \ldots, n$. The stabilization maps to $M_{g,n}^{\trop}$ record the underlying stabilized graph together with its edge lengths. Here $M_{g,n}^{\trop}$ denotes the moduli space of tropical curves of genus $g$ with $n$ markings. 

Since we only require the existence of this morphism in genus~$1$—and, more specifically, only after further projection to $M_{1,1}^{\trop}$—we do not discuss the general case here.

    \subsection{Radial alignment and well-spacedness}\label{sec: well-spaced tropical maps}

    In this section, we specialize to genus $1$ and recall the definition of the moduli space of tropical well-spaced curves.

Let $\mathsf{C}$ be a tropical curve of genus $1$. The core $\mathsf{C}_0$ of $\mathsf{C}$ is its minimal genus $1$ subgraph. We denote by $T(\mathsf{C}) = E(\mathsf{C}) \setminus E(\mathsf{C}_0)$ the set of edges of $\mathsf{C}$ that are not contained in $\mathsf{C}_0$. Moreover, we define a partial order on the vertices of $\mathsf{C}$ as follows: for $V, V' \in V(\mathsf{C})$, we set $V < V'$ if and only if the distance from $V$ to $\mathsf{C}_0$ is strictly less than the distance from $V'$ to $\mathsf{C}_0$.

\begin{definition}\label{def:radalign}
A radial alignment on $\mathsf{C}$ is a surjective morphism
\[
\rho \colon V(\mathsf{C}) \to \{0,\ldots,k\}
\]
for some $k \geq 0$, such that $\rho^{-1}(0) = V(\mathsf{C}_0)$ and whenever $V < W$, we have $\rho(V) < \rho(W)$. In particular, if $\rho(V)=\rho(V')$, then $V$ and $V'$ have the same distance from $\tropC_0$.

The integer $k$ is called the length of $\rho$ and is denoted by $\ell(\rho)$.
\end{definition}

Adding the data of a radial alignment to the combinatorial type of the tropical curve, one obtains a subdivision $M^{\rad}_\Gamma(\Sigma(X))$ (resp. $M^{\rad}_\Gamma( (N_X)_\R)$) of $M_\Gamma(\Sigma(X))$ (resp. $M_\Gamma( (N_X)_\R)$). More precisely, the cones in $M^{\rad}_\Gamma(\Sigma(X))$ (resp. $M^{\rad}_\Gamma( (N_X)_\R)$)  parametrize tropical curves in $\Sigma(X)$ (resp. $(N_X)_\R$) with a fixed combinatorial type, where this combinatorial type now includes the radial alignment $\rho$ in addition to the combinatorial data of curves in $\Sigma(X)$ in Definition \ref{def: comb-type}.

    \begin{definition} \label{def: nbhd-cycle}
    Given a parametrized tropical curve $[h: \tropC \rightarrow (N_X)_{\mathbb{R}}]$, the neighborhood of the cycle is defined as the parametrized tropical curve $[h_0: \nbhdtropC \rightarrow (N_X)_{\mathbb{R}}]$ where:
    \begin{enumerate}
        \item the tropical curve $\nbhdtropC$ is given by taking all the vertices $V \in V(\tropC_0)$ and any adjacent half-edges to $V$ (replacing edges with infinite legs as necessary); and 
        \item the piecewise linear function $h_0$ on $\nbhdtropC$ is the restriction of $h$ on the cycle $\tropC_0$, and has the same slopes $u_{\zeta}$ on all half-edges $\zeta$ of $\nbhdtropC$. 
    \end{enumerate}
    \end{definition}

    \begin{definition}\cite[Definition 4.4.4.]{RSWII} 
        Let $[h:\tropC \rightarrow (N_X)_{\mathbb{R}}]$ be a parametrized tropical curve of genus $1$. Let $\tropC_0$ be the circuit of $\tropC$. Given a half-edge $\zeta$ in $\tropC$ adjacent to vertex $V$, let $d(\zeta,\tropC_0)$ be the distance between the $\tropC_0$ to $V$. Then $[h:\tropC \rightarrow (N_X)_{\mathbb{R}}]$ is well-spaced if for every hyperplane $H \subset (N_X)_{\mathbb{R}}$, one of the following holds: 
        \begin{enumerate}[label=(\arabic*)]
            \item the neighborhood of the cycle $h_0(\nbhdtropC)$ is not contained in $H$; or
            \item the neighborhood of the cycle $h_0(\nbhdtropC)$ is contained in $H$, and if the set of half-edges $\zeta$ adjacent to $V$ with $h(V) \in H$ and $h(\zeta) \nsubseteq H$ is non-empty, then among this set the minimum of $d(\zeta, \tropC_0)$ occurs at least three times. 
        \end{enumerate}

        A parametrized tropical curve $[h:\tropC \rightarrow \Sigma(X)]$ of genus $1$ is well-spaced if it is well-spaced as a parametrized tropical curve in $(N_X)_{\mathbb{R}}$.
    \end{definition}

    The cone complex $W_\Gamma(\Sigma(X))$ (resp. $W_\Gamma((N_X)_{\mathbb{R}})$) of \emph{well-spaced} parametrized tropical curves is, by definition, the subcone complex of $M^{\rad}_\Gamma(\Sigma(X))$ (resp. $M^{\rad}_\Gamma((N_X)_{\mathbb{R}})$) parametrizing well-spaced tropical curves. By \cite[Theorem 3.2.10]{Tor14}, it is a pure-dimensional cone complex of dimension
    $$n+ m+ (r - \dim \Span(\delta_1, \ldots, \delta_m)).$$

    The combinatorial type of a well-spaced tropical curve in $\Sigma(X)$ (resp. in $(N_X)_\R$) with radial alignment $\rho$ is the same as the combinatorial type of the corresponding curve viewed in $M^{\rad}_{\Gamma}(\Sigma(X))$ (resp.~$M^{\rad}_{\Gamma}((N_X)_\R)$).
    
    \begin{definition} \label{def: filtration}
    Consider a combinatorial type $\sigma$ of radially aligned tropical curves in $(N_X)_{\mathbb{R}}$ with radial alignement $\rho: V(\tropC) \rightarrow \{ 0, \ldots, k\}$. For $a = 0, \ldots, k$, we define
    $$H'_a = \Span \{ u_{\zeta} \mid \zeta \text{ adjacent to } V, \rho(V) \leq a \}.$$
    Let $H_1 \subsetneq H_2 \subsetneq \ldots \subsetneq H_{e}$ be vector subspaces such that
    $$\{H_1, \ldots, H_e\} = \{H'_a \mid 0 \leq a \leq k\},$$
    that is, we remove repetitions from the $H_a'$. Note that $H_1$ is precisely the neighborhood of the cycle as defined in \S\ref{sec: result weights}.
    Let $a_i$ be the minimum $a$ for which $H'_a = H_i$. For instance, $a_1=0$.
    We also define 
    $$H_0 = \Span \{u_{\zeta} \mid \zeta \in L(\tropC_0)\}$$ 
    to be the span of the cycle.

    Define integers $r_0 = \dim H_0$ and $r_1 = \dim H_1$, the dimensions of the span of the cycle and the neighborhood of the cycle respectively. 
    For each $i = 1, \ldots, e$, let $V_{ij}$ for $j =1 , \ldots, h_i$ be the vertices in $\rho^{-1}(a_i)$. Let $\zeta_{i,j,k}$ for $k = 1, \ldots, f_{i,j}$ be the set of half-edges adjacent to $V_{i,j}$ and that are part of an edge connecting $V$ to a vertex where $\rho$ is bigger than $\rho(V)$. Let $u_{i,j,k} = u_{\zeta_{i,j,k}}$ be the direction vector of $\zeta_{i,j,k}$. This is illustrated in Figure \ref{fig:radial alignment}. 
    
    \begin{figure}[h]
        \centering
        \begin{tikzpicture}
            \draw[coralpink, dashed] (0, 3) -- (0, -3); 
            \draw[coralpink, dashed] (1, 3) -- (1, -3); 
            \draw[coralpink, dashed] (3, 3) -- (3, -3); 
            \draw[skyblue, dotted] ( 2, 3) -- (2, -3); 

            \draw (0,0) -- (1,1); 
            \draw (0,0) -- (1, -1); 
            \draw (1, 1) -- (3, 2); 
            \draw (3,2) -- (4, 2.5); 
            \draw (3,2) -- (4, 2) ; 
            \draw (3,2) -- (4, 1.5); 
            \draw (1,1) -- (4, 0.5); 
            \draw (1,-1) -- (4, -0.5); 
            \draw (1,-1) -- (3, -2); 
            \draw (3,-2) -- (4, -1.5); 
            \draw (3,-2) -- (4, -2.5);

            \draw[forestgreen, ->] (3,2) -- (3.5, 2.25); 
        
            \fill (0,0) circle[radius=4pt];
            \fill (1, 1) circle[radius=2pt];
            \fill (1, -1) circle[radius=2pt];
            \fill (3, 2) circle[radius=2pt];
            \fill (3, -2) circle[radius=2pt];
            \fill (2, -0.833) circle[radius=2pt];

            \node at (-0.4, 0) {$\tropC_0$};
            \node at (2.6, 2.2) {$V_{3,j}$};
            \node[text=forestgreen] at (3.5, 2.6) {$u_{3,j,k}$};
            
            \node at (0, -3.3) {$0$};
            \node at (1, -3.3) {$1$};
            \node at (2, -3.3) {$2$};
            \node at (3, -3.3) {$3$};
            
            \node at (0, -3.8) {$a_1$};
            \node at (1, -3.8) {$a_2$};
            \node at (3, -3.8) {$a_3$};
            
        \end{tikzpicture}
        \caption{Example of radial alignment. The cycle is drawn as a large bullet, while the vertices outside the cycle are drawn as small bullets. The group of vertices in $\rho^{-1}(a)$ are shown on the vertical lines, which are pink dashed lines when $a$ is equal to some $a_i$ and blue dotted lines otherwise. }
        \label{fig:radial alignment}
    \end{figure}
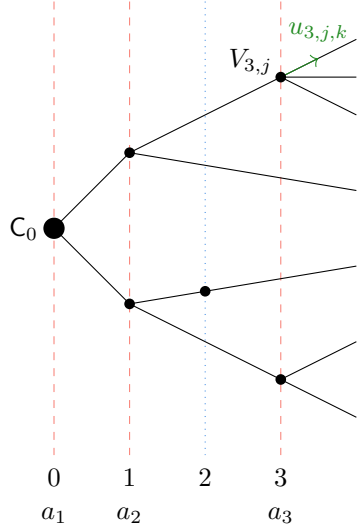	
    
    \end{definition}
    
    Let $\sigma$ be a maximal cone of $W_{\Gamma}((N_X)_{\mathbb{R}})$. With notation as in Definition~\ref{def: filtration}, for each $i \in 2, \ldots, e$, for any $1 \leq j, j' \leq h_i$, the radial alignment requires that $d(\tropC_0, V_{i,j}) = d(\tropC_0, V_{i,j'})$. As the next lemma shows these are in fact all the relations required by the radial alignment $\rho$.

    \begin{lemma}\label{lemma: Ha'=Ha+1'}
        Let $a \geq 1$. If $H_a'=H_{a+1}'$, then $|\rho^{-1}(a)|=1$. 
    \end{lemma}

    \begin{proof}
        Let $\mathcal{W}_a$ be the set of half-edges adjacent to vertices $V \in \rho^{-1}(a)$ and that are part of an edge connecting $V$ to a vertex where $\rho$ is bigger than $\rho(V)$. 
        From \cite[Proposition 3.2.18(c)]{Tor14}, for $a = 2, \ldots, k$ we have that $$\dim(H_a'/H_{a-1}') = |\mathcal{W}_a| - | \rho^{-1}(a)| - 1.$$
        Since we are working in $(N_X)_{\mathbb{R}}$, every vertex outside of the cycle must have valence at least $3$, so 
        $$ |\mathcal{W}_a| \geq 2 \cdot | \rho^{-1}(a)|.$$
        It follows that if $\dim(H_a'/H_{a-1}') = 0$, then $| \rho^{-1}(a)| = 1$. 
    \end{proof}

    Now, fix $i=2,\ldots,e$. For each pair $j,j'=1,\ldots, h_i$, there exists a unique path from $\tropC_0$ to $V_{i,j}$ (resp. $V_{i,j'}$) consisting of edges $e_{i,j,t}$ (resp. $e_{i,j',t'}$). Within the cone $\sigma$, we impose the relations
\begin{equation*} \label{eqn: well-spaced relations}
    \sum_{t} \ell(e_{i,j,t}) = \sum_{t'} \ell(e_{i,j',t'}).
\end{equation*}

These constraints are clearly redundant and we can be reduced this system of equations to the equivalent set
\begin{equation*} 
    \sum_{t} \ell(e_{i,h_i,t}) = \sum_{t'} \ell(e_{i,j,t'})
\end{equation*}
for $j = 1, \ldots, h_i-1$.
    
    \begin{definition}\label{def: sets V}
        Using the notation as in Definition \ref{def: filtration}, for $i = 2, \ldots, e$, we define the sets
    \begin{equation}\label{def: cV}
    \mathcal{V}_i = \{V_{i,1}, \ldots, V_{i,h_i}\} \quad 
    \end{equation}
    and let $\cV= \bigcup_{i=1}^{e-1} \mathcal{V}_i$.
    \end{definition}
    
    Note that the morphism of cone complexes in Equation \eqref{eqn: forgetfan} restricts to a morphism
\begin{equation} \label{eqn:ws-forgetfan}
    W_\Gamma(\Sigma(X)) \rightarrow W_\Gamma((N_X)_{\mathbb{R}}),
\end{equation}
which is also a subdivision involving lattice indices.

    \begin{definition} \label{def: lattice index}
    Let $\sigma$ be a maximal-dimensional cone of $W_{\Gamma}(\Sigma(X))$ mapping to $\sigma'$ in $W_{\Gamma}((N_X)_{\mathbb{R}})$. The \emph{lattice scaling} of $\sigma$, denoted by $\mathrm{lat}(\sigma)$, is the lattice index of the map $N_\sigma \rightarrow N_{\sigma'}$.
\end{definition}

Finally, the diagram in \eqref{eqn: full tropical ev map} induces a similar commutative diagram involving well-spaced tropical curves rather than general tropical curves.

\section{Well-Spaced genus $1$ maps over fine and fs logarithmic schemes}

In this section, we review the construction of $\cW_\Gamma(X)$ from \cite{RSWII}. We also use this section to establish most of the notation for logarithmic maps used in subsequent sections.

\subsection{Abramovich-Chen and Gross-Siebert spaces of logarithmic maps}\label{sec: ACGS space}

Abramovich-Chen \cite{AC14} and Gross-Siebert \cite{GS2013} independently introduced the moduli stack of stable logaritmic maps to $X$. 

\begin{definition} [{\cite[definition 2.12]{ACGSI}}]\label{def: stable log map}
A stable logarithmic map $(C/S,\mathsf{p},f)$ is a diagram
\[
\begin{tikzcd} 
& C \arrow[r, "f"] \arrow[d, "\pi"'] & X \\ &S \arrow[u, bend left=60, "p_i"] \arrow[u, bend right=60, "z_j"'] 
\end{tikzcd} 
\] where: 
\begin{enumerate}
    \item $\pi$ is a proper, logarithmically smooth and integral morphism of log schemes with a tuple of sections $\mathsf{p}=(p_1,\ldots,p_n,z_1,\ldots,z_m)$ of $\underline{\pi}$ such that every geometric fiber of $\underline{\pi}$ is a reduced and connected curve, and if $U \subseteq \underline{C}$ is the non-critical locus of $\underline{\pi}$ then $\overline{\cM}_C|_U \simeq \pi^*\overline{\cM}_S \oplus \bigoplus_{i=1}^n (p_i)_* \mathbb{N} \oplus \bigoplus_{j=1}^m (z_j)_*\N$
    \item For every geometric point $s \in \underline{S}$, the restriction of $\underline{f}$ to $\underline{C}_s$ together with $\mathsf{p}$ is an ordinary stable map. 
\end{enumerate} 
\end{definition}

In the subsequent discussion, the $p$-markings will have trivial contact order with $\partial X$ and will be constrained to lie in prescribed subvarieties of $X$, while the $z$-markings will have prescribed, non-trivial contact order with $\partial X$.

\begin{remark}\label{rmk: local structure at node}
    
Suppose $\underline{S}=\operatorname{Spec}(A)$ with $(A,\mathfrak{m})$ a strictly Henselian local ring and let $0 \in \underline{S}$ be the closed point and $Q= \oM_{S,0}$. Let $\sigma: Q \to A$ be a chart. By \cite[Page 222]{Kato}, ètale locally around a node $q$ in the central fiber $C_0$, the curve $C$ is isomorphic to the logarithmic scheme
\[
V:=\operatorname{Spec}\big(A[z,w]/(zw - t)\big) \to \operatorname{Spec}(A),
\]
with $t \in \mathfrak{m}$, and with the log structure induced from the homomorphism
\[
Q \oplus_{\mathbb{N}} \mathbb{N}^2 \longrightarrow \mathcal{O}_V,
\qquad
(q,(a,b)) \longmapsto \sigma(q)\, z^a w^b.
\]
Here $\mathbb{N} \to \mathbb{N}^2$ is the diagonal embedding, and $\mathbb{N} \to Q$ is a homomorphism $1 \mapsto \rho_q$ with $\rho_q \neq 0$.

\end{remark}

\subsubsection{Basic maps and the ACGS moduli stack}

In order to define the functor of stable logarithmic maps on the category of schemes, one needs to introduce the notion of basicness (also called  minimality in \cite{GS2013}).

Let Let $(C/S,\mathsf{p},f)$ be a logarithmic map over $S = \mathrm{Spec}(Q' \to k)$, where $Q'$ is a sharp fs monoid and $k$ is an algebraically closed field. Then we have morphisms

\[
\begin{tikzcd} 
& \oM_C  & \underline{f}^{-1}\oM_X \arrow[l, "f^\flat "] \\ 
&\underline{\pi}^{-1} Q' \arrow[u, "\pi^\flat"']
\end{tikzcd} 
\]

as follows. The morphism $\pi^\flat$ is an isomorphism outside of nodes and markings of $C$. At a marked point, the sheaf $\oM_C$ has stalk $Q' \oplus \mathbb{N}$ and $\pi^\flat$ is the inclusion in the first factor. At a node, the stalk is
$
\oM_{C,q} = Q' \oplus_{\mathbb{N}} \mathbb{N}^2,
$
where the pushout is determined by the map
\begin{equation*}
\mathbb{N} \to Q' \quad  1 \mapsto \rho_q \in Q' \setminus \{0\}.
\end{equation*}
and the diagonal map $\mathbb{N} \to \mathbb{N}^2$. In particular, we can identify
\begin{equation}\label{eqn: structure at nodes}
    Q' \oplus_{\mathbb{N}} \mathbb{N}^2 \simeq \{(m_1,m_2) \in Q \times Q \ | \ m_1-m_2 \in \Z\rho_q \subseteq (Q')^{gr} \} \subseteq Q \times Q
\end{equation}

Again, $\pi^\flat$ is the inclusion in the $Q'$ factor.

Let $x \in \underline{C}$ be a geometric point. Then $f^\flat$ induces a map on stalks
\[
f^\flat_x: \oM_{X,\underline{f}(x)} \to \oM_{\underline{C},x}.
\]
%which depends only on the type of $x$:
The behavior of $f_x^\flat$ depends on \(x\) as follows: 

\begin{itemize}
    \item If $x = \eta$ is a generic point of a component, then
    \[
    f^\flat_\eta : \oM_{X,\underline{f}(\eta)} \to Q'
    \]
    is a local isomorphism of monoids. 

    \item If $x = p$ or $x=z$ is a marked point, then
    $
    f^\flat_x: \oM_{X,\underline{f}(x)} \to Q' \oplus \mathbb{N}
    $
    induces a map
    \begin{equation}\label{eqn:contact_order}
        u_x: \oM_{X,\underline{f}(p)} \xrightarrow{f^\flat_p} Q' \oplus \mathbb{N} \xrightarrow{\mathrm{pr}_2} \mathbb{N},
    \end{equation}
    called the \emph{contact order} at $x$.

    \item If $x = q$ is a node lying at the intersection of two components with generic points $\eta_1$ and $\eta_2$, let
    \[
    \chi_i: \oM_{X,\underline{f}(q)} \to \oM_{X,\underline{f}(\eta_i)}
    \]
    be the generization maps. Then, as explained in \cite{GS2013}, we have a commutative diagram

    \[
    \begin{tikzcd}
    & \oM_{X,\underline{f}(\eta_1)} \ar[rr, "f^\flat_{\eta_1}"]  & & Q'  \\
    \oM_{X,\underline{f}(q)} \ar[ur, "\chi_1"] \ar[dr, "\chi_2"] \ar[rr, "f^\flat_q"'] & & Q' \oplus_{\mathbb{N}} \N^2 \ar[ur] \ar[dr] \ar[r] & Q' \times Q' \ar[d, "pr_2"] \ar[u, "pr_1"] \\
    & \oM_{X,\underline{f}(\eta_2)} \ar[rr, "f^\flat_{\eta_2}"']  & & Q'
    \end{tikzcd}
    \]
    By \eqref{eqn: structure at nodes}, there exists a homomorphism
    \begin{equation}
        u_q: \oM_{X,\underline{f}(q)} \to \Z
    \end{equation}
    called \emph{contact order} at $q$, such that for all $m \in \oM_{X,\underline{f}(q)}$
    \begin{equation}\label{eqn: relations}
    f^\flat_{\eta_1}(\chi_1(m))-f^\flat_{\eta_2}(\chi_2(m))= u_q(m) \rho_q
    \end{equation}
    with $\rho_q \neq 0$ as in Remark \ref{rmk: local structure at node}. Note that the sign of $u_q$ depends on a choice of ordering of $\eta_1$ and $\eta_2$. We will assume that we have chosen an orientation on $\mathsf{C}$, and all the definitions will be independent of this choice.
\end{itemize}

 Let $R \subseteq (\prod_{\eta} \oM_{X,f(\eta)} \times \prod_e \N)^{gr}$ be the subgroup generated by the elements
$$
a_q(m):=( (\ldots, \chi_1(m),\ldots,-\chi_2(m),\ldots),(\ldots,u_q(m),\ldots) 
$$
where $q$ runs over the nodes and $m \in \oM_{X,\underline{f}(q)}$. Let $Q$ be the fs free monoid defined by
\begin{equation}\label{eqn: basic monoids}
   Q=\Bigg[ \frac{\prod_{\eta} \oM_{X,\underline{f}(\eta)} \times \prod_q \N}{R^{\sat}} \Bigg]^{\sat} 
\end{equation}

where $\eta$ runs over  the generic points of the components of $C$ and $q$ over the nodes of $C$. Then, by the above discussion and Equation \eqref{eqn: relations}, we obtain a map 
\begin{equation}\label{eqn: basicness map}
    Q \to Q'
\end{equation}

\begin{definition}
    Let $(C/S,\mathsf{p},f)$ be a stable logarithmic map. We say that $f$ is basic if for all geometric points $s\in \underline{S}$, the map \eqref{eqn: basicness map} defined by the restriction of $(C_s/s,\mathsf{p}(s),f_s)$  is an isomorphism.
\end{definition}

\subsubsection{Discrete data}

\begin{definition}[Rephrasing of {\cite[Definition 2.15]{ACGSI}}]
We denote by $\Gamma$ the discrete data consisting of:
\begin{itemize}
    \item the genus $g$, the curve class $A \in H_2(X)$, the number of marked points $N=n+m$;
    \item integral elements
    $
    u_{p_1}, \ldots, u_{p_n}, u_{z_1},\ldots,u_{z_{m}} \in (N_X)_{\mathbb{R}}
    $. We will always be in the situation where $u_{p_i}=0$ for $i=1,\ldots,n$.
\end{itemize}

We say that a stable logarithmic map $(C/S, \mathbf{p}, f)$ is of type $\Gamma$ if the following conditions are satisfied:
\begin{itemize}
    \item The underlying ordinary stable map is of type $(g, A, N)$. 
    \item For $j=1,\ldots, m$ define the closed subset $Z_j \subset X$ to be the union of strata with generic points $\eta$ such that $u_{z_j}$ lies in the image of $\sigma_\eta= \mathrm{Hom}(\oM_{X,\eta},\N) \to (N_X)_{\mathbb{R}}$. Then for any $j$ we have $\mathrm{im}(f \circ z_j) \subset Z_j$, and for any geometric point $\bar{s} \to S$ such that $z_j(\bar{s})$ lies in the stratum of $X$ with generic point $\eta$, there exists $u \in \sigma_\eta = \mathrm{Hom}(\oM_{X,\bar{\eta}}, \mathbb{N})$ mapping to $u_{z_j} \in (N_X)_{\mathbb{R}}$ making the following diagram commute:
\[
\begin{tikzcd}[row sep=large, column sep=large]
\oM_{X,f(z_j(\bar{s}))} \arrow[r, "f^\flat"] \arrow[d] & \oM_{C,z_j(\bar{s})} = \oM_{S,\bar{s}} \oplus \mathbb{N} \arrow[d, "\mathrm{pr}_2"] \\
\oM_{X,\bar{\eta}} \arrow[r, "u"'] & \mathbb{N}
\end{tikzcd}
\]
\end{itemize}
\end{definition}

The curve class $A \in H_2(X)$ is uniquely determined by its intersection multiplicities with the boundary divisors of the toric variety $X$. We will always assume the following. For each ray $\rho \in \Sigma[1]$, let $n_\rho \in N_X$ be its primitive generator. If a contact order $u_j \in N_X$ belongs to a cone $\tau \in \Sigma(X)$, it decomposes uniquely as a non-negative integral combination 
\[
u_j = \sum_{\rho' \in \tau[1]} c_{\rho'}(u_j) n_{\rho'}.
\]
We define $(u_j)_\rho = c_\rho(u_j)$ if $\rho \in \tau[1]$, and $(u_j)_\rho = 0$ otherwise. Then, we will always assume that
\[
D_\rho \cdot A = \sum_{j} (u_j)_\rho.
\]

There is a proper Deligne--Mumford moduli stack $\mathsf{ACGS}_{\Gamma}(X)$ parametrizing families of basic stable logarithmic maps with discrete data $\Gamma$ \cite[Theorem~2.4]{GS2013}.

\subsubsection{Variants of the basic monoid} \label{sec:define monoid}

One may consider various variants of $Q$. We define

\begin{equation}\label{eqn: variants of Q}
Q^{\mathsf{fine}} =
\left[
\frac{\prod_{\eta} \oM_{X,\underline{f}(\eta)} \times \prod_q \mathbb{N}}{R}
\right] \quad \text{and} \quad 
Q^{\mathrm{un}} =
\left[
\frac{\prod_{\eta} \oM_{X,\underline{f}(\eta)} \times \prod_q \mathbb{N}}{R^{\mathrm{sat}}}
\right].
\end{equation}

The monoid $Q^{\mathsf{fine}}$ is fine, but in general neither saturated nor torsion free; and $Q^{\mathrm{un}}$ is its torsion-free part. Finally,
$
Q = (Q^{\mathrm{un}})^{\mathrm{sat}}
$
is the saturation of $Q^{\mathrm{un}}$.

The monoid $Q^{\mathrm{un}}$ already appeared in the literature before \cite{R17}, while $Q^{\mathsf{fine}}$ did not and will play an essential role in \S\ref{sec: weights}, where we will deal with fine but not necessarily saturated maps.

\subsubsection{Tropical interpretation of the basic monoid and combinatorial type}

For the tropically inclined reader, it may be useful to think of $Q$ using the identification

\begin{equation}\label{eqn: dual-of-Q}
Q^\vee = \mathrm{Hom}(Q,\N)=
\Bigg\{
\bigl( (V_\eta)_\eta, (e_q)_q \bigr) \in \bigoplus_\eta \overline{\mathcal{M}}_{X,\underline{f}(\eta)}^\vee \;\oplus\; \bigoplus_q \N
\;\Big|\;
\forall q: V_{\eta_1} \circ \chi_1 - V_{\eta_2} \circ \chi_2 = e_q u_q
\Bigg\}.
\end{equation}

See \cite[Remark 1.18]{GS2013}.

We conclude this section with the definition of combinatorial type of logarithmic maps.

\begin{definition}[{\cite[Definition 2.13]{ACGSI}}]
    Let $(C/S,\mathsf{p},f)$ be a stable logarithmic map over $S=\mathrm{Spec}(Q \to k)$. The combinatorial type of $(C/S,\mathsf{p},f)$ consists of the following data:
    \begin{itemize}
        \item the dual graph $G$ of $C$;
        \item the genus function $g: V(G) \to \Z_{\geq 0}$ associating to a vertex $V$ the genus of the corresponding component $C_V$;
        \item the map $ V(G) \cup E(G) \cup L^e(G) \to \Sigma(X)$ mapping $x \in C$ to the cone $(\oM_{X,\underline{f}(x)}^\vee)_{\R} \in \Sigma(X)$ (see \cite[Definition 2.13 and Proposition 2.25]{ACGSI});
        \item the contact data $\mathsf{u}=\{u_{p}, u_{z}, u_q\}$ at marked points $p, z$ and nodes $q$.
    \end{itemize}
\end{definition}

\subsubsection{Maps to the Artin fan}

There is also a moduli stack  $\mathfrak{M}_{\Gamma'}(\mathcal{A}_X)$ parametrizing basic logarithmic maps to the Artin fan $\mathcal{A}_X$. Since curve classes do not make sense on $\mathcal{A}_X$, the discrete data is just
$$
\Gamma'=(g,u_{p_1},\ldots,u_{p_n}, u_{z_1}, \ldots, u_{z_m})
$$
Since the logarithmic tangent space $T_{\mathcal{A}_X}^{log}= 0$ is the zero sheaf, we also have $H^1(C, f^*T_{\mathcal{A}_X}^{log})=0$ for all maps $(C/S, \mathsf{p},f)$ in $\mathfrak{M}_{\Gamma'}(\mathcal{A}_X)$. In particular, $\mathfrak{M}_{\Gamma'}(\mathcal{A}_X)$ is a smooth pure dimensional Artin stack, and clearly the natural map
$$
\mathsf{ACGS}_{\Gamma}(X) \to \mathfrak{M}_{\Gamma'}(\mathcal{A}_X)
$$
is strict.

\subsection{Radially aligned genus $1$ curves}\label{sec: aligned curves}

We start with recalling the notion of radially aligned curves. We refer to \cite{RSWI, BNR21, KS} for the details.

Recall the notion of radial alignment for a tropical curve $(\mathsf{C},\mathsf{p})$ in Definition \ref{def:radalign}. Adding the data of a radial alignment to the combinatorial type of the tropical curve, one obtains a logarithmic blow-up of the moduli stack of prestable $n$-pointed genus $1$ curves
\[
\mathfrak{M}_{1,n}^{\mathrm{rad}} \longrightarrow \mathfrak{M}_{1,n},
\]
which is the moduli stack of prestable radially aligned $n$-pointed genus $1$ curves.

\begin{remark}
    
For every logarithmic map $\mathrm{Spec}(Q \to k) \to \mathfrak{M}_{1,n}^{\mathrm{rad}}$ with $Q$ a fs monoid, denote by $\mathsf{C}$ the tropicalization of the induced curve. For a vertex $V$ of $\mathsf{C}$,  define
\[
\lambda(V) = \sum \rho_{q_i} \in Q,
\]
where $e_1,\ldots,e_k$ is the path from $\mathsf{C}_0$ to $V$, corresponding to nodes $q_i$, and $\rho_{q_i}$ is as in Remark \ref{rmk: local structure at node}. If $V \in V(\mathsf{C}_0)$, then we set $\lambda(V)$ to be the identity element of $Q$.

Then, it follows from the construction above that for every $V,V' \in V(\mathsf{C})$, the quantities $\lambda(V)$ and $\lambda(V')$ are comparable, that is, either
\[
\lambda(V) - \lambda(V') \in Q \subseteq Q^{\mathrm{gp}}
\quad \text{or} \quad
\lambda(V') - \lambda(V) \in Q \subseteq Q^{\mathrm{gp}}.
\]
\end{remark}

A modular interpretation of $\mathfrak{M}_{1,n}^{\mathrm{rad}}$ is provided in \cite[Lemma~3.2]{KS}. Fix a tropical curve and radial alignment $(\mathsf{C},\rho)$ and denote by
\[
\mathfrak{M}(\mathsf{C},\rho) \subseteq \mathfrak{M}_{1,n}^{\mathrm{rad}}
\]
the locally closed stratum parametrizing curves whose dual graph is $\mathsf{C}$ and whose radial alignment is $\rho$. This stratum maps to the stratum
$
\mathfrak{M}(\mathsf{C}) \subseteq \mathfrak{M}_{1,n}
$
of prestable curves with dual graph $\mathsf{C}$.

\begin{notation}
   Let $(C,p_1,\ldots,p_n) \in \mathfrak{M}(\mathsf{C})$ and let $V \in V(\mathsf{C}) \smallsetminus V(\mC_0)$ be a vertex in the dual graph, corresponding to a component $C_V$ of $C$. We denote by $\nu_V$ the node associated to the unique edge $e$ of $\mathsf{C}$ connecting $V$ to a vertex $V'$ with $V' < V$.
\end{notation}

\begin{lemma}[{\cite[Lemma 3.2]{KS}}]
The map 
$
\mathfrak{M}(\mathsf{C},\rho) \to \mathfrak{M}(\mathsf{C})
$
is a torsor with structure group
\[
\mathbb{G}_m^{|T(\mathsf{C})|-\ell(\rho)}
=
\prod_{i=1}^{\ell(\rho)} \mathbb{G}_m^{|\rho^{-1}(i)|-1}.
\]
Furthermore, each factor $\mathbb{G}_m^{|\rho^{-1}(i)|-1}$ is naturally identified with collections of isomorphisms
\begin{equation}\label{eqn: identification of tangent spaces}
\{ T_{\nu_V}C_V \simeq T_{\nu_{V'}}C_{V'} \}_{V,V' \in \rho^{-1}(i)}
\end{equation}
that are compatible under composition.
\end{lemma}

Given a point $\mathrm{Spec}(Q \to k) \to \mathfrak{M}(\mathsf{C},\rho)$ with underlying curve $(C,p_1,\ldots,p_n)$, one can produce a logarithmic blow-up $\widetilde{C} \to C$ and a contraction $\widetilde{C} \to \overline{C}$ to a curve with a Gorenstein elliptic singularity as follows.
Fix a distance $d \in Q$ such that there exists a vertex $V$ in the tropicalization $\mathsf{C}$ of $C$ with $\lambda(V)=d$. First, subdivide $\mathsf{C}$ by adding a vertex at every point with distance $d $ from $\mathsf{C}_0$. This yields the logarithmic modification $\widetilde{C} \to C$. Next, contract all the components $C_{V'}$ with $\lambda(V') < d $ to a point and glue the components $C_V$ with $V$ at distance $d$ from $\mathsf{C}_0$ according to the isomorphisms \eqref{eqn: identification of tangent spaces}.

The new components in $\widetilde{C}$ do not come with a canonical choice of such gluings; however, any two different choices will yield curves that are isomorphic as logarithmic pointed curves. 
Note that, because we subdivide first, and declaring the logarithmic structure to be trivial around the genus $1$ singularity, the curve $\overline{C}$ naturally acquires a logarithmic structure making $\widetilde{C} \to \overline{C}$ a logarithmic map.

\subsection{Well-spaced genus $1$ maps}

The moduli stack of well-spaced maps $\WG(X)$ is obtained as a strict closed substack of the fs fiber product
\begin{equation}\label{eqn: inclusion WG(X) in ACGS}
\WG(X) \subseteq \widetilde{\mathsf{ACGS}_{\Gamma}(X)}:=\mathsf{ACGS}_{\Gamma}(X) \times_{\mathfrak{M}_{1,1}}^{\mathsf{fs}} \mathfrak{M}_{1,1}^{\rad}
\end{equation}
By definition, the right-hand side is a subdivision of $\mathsf{ACGS}_{\Gamma}(X)$. In order to write the defining closed condition, we require some definition.

Consider now $S= \mathrm{Spec}(Q \to k) \to \widetilde{\mathsf{ACGS}_{\Gamma}(X)}$ a radially aligned curve together with a logarithmic map $f: C/S \to X$. 

\begin{notation}\label{not: well-spaced maps}
   We will simply write $(C/S,\mathsf{p},f,\rho)$ to denote the data corresponding to a $\underline{S}$-point of $\widetilde{\mathsf{ACGS}_{\Gamma}(X)}$. Strictly speaking, because the fiber product defining 
$
\widetilde{\mathsf{ACGS}_{\Gamma}(X)}
$ 
is taken in the fs category, a $\underline{S}$-point
$
\underline{S} \to \widetilde{\mathsf{ACGS}_{\Gamma}(X)}
$
contains slightly more information than the logarithmic map
$
(C/S,\mathsf{p},f)\in \mathsf{ACGS}_{\Gamma}(X)(\underline{S})
$
together with the radial alignment
$
\rho\in \mathfrak{M}_{1,1}^{\mathrm{rad}}(\underline{S}).
$
However, this additional data will not play any role in what follows.
\end{notation}

For any codimension $1$ saturated subgroup $H \subseteq N_X$ consider the morphism
$$
\alpha_H: (N_X/H)^* \to N_X^* \to \Gamma(X,\cM_{X}^{\gr}) \to \Gamma(C,\cM_{C}^{\gr})
$$

and let $\overline{\alpha}_H$ be the induced morphism $(N_X/H)^* \to \Gamma(C,\oM_{C}^{\gr})$. Let $\mathsf{C}$ be the tropicalization of $C$, and define
$ \delta_{\alpha_H} \in Q$
to be the largest value $\lambda(V)$ among vertices $V \in V(\mathsf{C})$ such that $\alpha_H(1)$ is constant (when viewed as a piecewise linear function with values in $Q$) on the subgraph consisting of vertices $V'$ with $\lambda(V') < \lambda(V)$.

Let $\widetilde{C} \to C$ and $\widetilde{C} \to \overline{C}$ the destabilization and contraction constructed in \S\ref{sec: aligned curves}. 

\begin{definition} \label{def: factorization}
We say that $\alpha_H$ satisfies the factorization property if 
$$
(N_X/H)^* \xrightarrow{\alpha_H} \Gamma(C,\cM_{C}^{\gr}) \to \Gamma(C,\cM_{\widetilde{C}}^{\gr})
$$
factors through the map
$
\Gamma(\overline{C}, \cM_{\overline{C}}^{\gr}) \longrightarrow \Gamma(\widetilde{C}, \cM_{\widetilde{C}}^{\gr}).
$
\end{definition}

\begin{definition}[{\cite[Definition 3.4.1]{RSWII}}]
    We say that the radially aligned map $(C/S,\mathsf{p},f,\rho)$ 
    is well-spaced if for all $s \in S$ and codimension $1$ saturated subgroups $H \subseteq N_X$, the induced map $\alpha_H$ satisfies the factorization property.
\end{definition}

This defines a proper and logarithmically smooth Deligne--Mumford stack $\mathcal{W}_{\Gamma}(X)$ parametrizing families of basic well-spaced maps with fixed discrete data $\Gamma$. It sits as a strict closed substack inside $\widetilde{\mathsf{ACGS}_{\Gamma}(X)}$, as in Equation~\eqref{eqn: inclusion WG(X) in ACGS}. See \cite[Theorems~3.4.10 and~3.5.1]{RSWII}. The notion of basic monoids is determined by the fact that the inclusion in Equation~\eqref{eqn: inclusion WG(X) in ACGS} is strict.

\subsection{Integrality of the invariants}\label{sec: integrality}

In this section we prove Proposition \ref{prop: integrality}, and in particular that the well-spaced invariants in Equation \eqref{eqn: correspondence} are integers.

%\begin{proposition}
%Assume $n \geq 1$. Then the generic point of every irreducible component of $\mathcal{W}_\Gamma(X)$ has trivial automorphism group. In particular, the well-spaced invariants in Equation \eqref{eqn: correspondence} are integers.
%\end{proposition}

\begin{proof}[Proof of Proposition \ref{prop: integrality}]
Let $Y$ be an irreducible component of $\mathcal{W}_\Gamma(X)$, and let $(C,\mathsf{p},f,\rho)$ be a general point of $Y$. In particular, $C$ is a smooth curve. Write  $p = p_1 \in C$.

If the morphism $f$ has degree $1$ onto its image, then any automorphism of $(C,\mathsf{p},f,\rho)$ is trivial. Otherwise, $f$ factors as
\[
C \xrightarrow{f'} C' \xrightarrow{f''} X,
\]
where $f''$ is birational onto its image.

After moving $p$ slightly in $C$ if necessary, we may assume that $f'$ is unramified at $p$. Let $\varphi : C \to C$ be an automorphism of $(C,\mathsf{p},f,\rho)$. Since $f''$ is birational onto its image, the automorphism $\varphi$ fixes $p$ and commutes with $f'$.

Consider the fixed locus
\[
\{ x \in C \mid \varphi(x) = x \}.
\]
This is a closed subset of $C$ containing $p$. We claim that it contains a nonempty open subset, and hence that $\varphi = \mathrm{Id}_C$.

Because $f'$ is unramified at $p$, there exists an open neighbourhood $U \subseteq C'$ of $f'(p)$ such that
\[
(f')^{-1}(U) = \bigsqcup_i U_i,
\]
a disjoint union of open subsets $U_i \subseteq C$, with each restriction $f' : U_i \to U$ an isomorphism.

Let $U_1$ be the component containing $p$. Then, since $\varphi$ commutes with $f''$ and fixed $p$, it has to preserve $U_1$ and the restriction $\varphi:U_1 \to U_1$ must be the identity map.
The proposition follows.
\end{proof}

Dropping the assumption $n \geq 1$, Proposition \ref{prop: integrality} fails, as the following example shows.

\begin{example}
Let $X = \mathbb{P}^1$, $n = 0$, and $m = 2$, with degree vector $\delta = (2,-2)$. Then every degree-$2$ cover of $\mathbb{P}^1$ that is fully ramified over $0$ and $\infty$ has automorphism group containing $\mu_2$.
\end{example}

\subsection{Well-spaced maps over fine logarithmic schemes}

The construction of the moduli stack $\mathsf{ACGS}_{\Gamma}(X)$ can be carried out over the category of fine (not necessarily saturated) log schemes. 

\begin{notation}
In this section, we will use the adjective ``log'' to refer to the $2$-category $\mathsf{LS}$ of fs logarithmic stacks. We will instead explicitly say ``fine log'' when working in the $2$-category $\mathsf{LS}^{\mathsf{fine}}$ of fine log stacks that are not necessarily saturated. 
\end{notation}

As in \S\ref{sec: ACGS space}, if $S = \mathrm{Spec}(Q' \to k)$, where $Q'$ is a fine monoid and $k$ is an algebraically closed field, and if $f \colon C/S \to X$ is a stable logarithmic map, one obtains a morphism
\[
Q' \to Q^{\mathsf{fine}}.
\]
where $Q^\fine$ is the monoid in Equation \eqref{eqn: variants of Q}. We say that a family of logarithmic maps $f \colon C/S \to X$, where $S$ is a fine log scheme, is \emph{basic} if for every geometric point $s \in S$ the induced map
\[
\oM_{S,s} \to Q^{\mathsf{fine}}
\]
is an isomorphism.

This yields an algebraic Deligne--Mumford stack 
$\mathsf{ACGS}_{\Gamma}(X)^{\mathsf{fine}}$ parametrizing families of basic fine stable logarithmic maps to $X$ of type $\Gamma$, together with an Artin stack 
$\mathfrak{M}_{\Gamma'}(\mathcal{A}_X)^{\mathsf{fine}}$ parametrizing families of basic fine stable logarithmic maps to $\mathcal{A}_X$ of type $\Gamma'$. Furthermore, there is a strict morphism
\[
\mathsf{ACGS}_{\Gamma}(X)^{\mathsf{fine}} \longrightarrow \mathfrak{M}_{\Gamma'}(\mathcal{A}_X)^{\mathsf{fine}}.
\]

We also have a proper and logarithmically smooth (in $\mathsf{LS}^{\mathsf{fine}}$) stack parametrizing families of basic fine well-spaced maps, denoted $\mathcal{W}_\Gamma(X)^{\mathsf{fine}}$, together with a strict inclusion
\[
\mathcal{W}_\Gamma(X)^{\mathsf{fine}} \subseteq 
\widetilde{\mathsf{ACGS}_{\Gamma}(X)^{\mathsf{fine}}} 
:= \mathsf{ACGS}_{\Gamma}(X)^{\mathsf{fine}} 
\times_{\mathfrak{M}_{1,n}}^{\mathsf{fine}} \mathfrak{M}_{1,n}^{\mathrm{rad}},
\]
similar to \eqref{eqn: inclusion WG(X) in ACGS}, where the fiber product is now taken in $\mathsf{LS}^{\mathsf{fine}}$. 

\begin{notation}
   We will write $(C^{\fine}/S^{\fine},\mathsf{p},f^{\fine},\rho)$ to denote the data corresponding to a point of
$
\widetilde{\mathsf{ACGS}_{\Gamma}(X)^{\mathsf{fine}}}.
$
Note that, unlike in Notation~\ref{not: well-spaced maps}, because the fiber product defining 
$
\widetilde{\mathsf{ACGS}_{\Gamma}(X)^{\mathsf{fine}}}
$
is now taken in the category $\mathsf{LS}^{\fine}$, such a point does not contain any additional information beyond the associated fine logarithmic map
$
(C,\mathsf{p},f)\in \mathsf{ACGS}^{\fine}_{\Gamma}(X)(\underline{S})
$
together with the radial alignment
$
\rho\in \mathfrak{M}_{1,1}^{\mathrm{rad}}(\underline{S}).
$
\end{notation}

As functors on logarithmic schemes, the functor $\WG(X)$ is the restriction of $\WG(X)^{\fine}$ to the category of fs logarithmic schemes. Both $\WG(X)$ and $\WG(X)^{\fine}$ admit basic logarithmic structures. Hence there is a natural morphism
\begin{equation}\label{eqn: map W}
\begin{aligned}
\WG(X) &\longrightarrow \WG(X)^{\fine}, \\
(C/S,\mathsf{p},f,\rho) &\longmapsto (C^{\fine}/S^{\fine},\mathsf{p},f^{\fine},\rho).
\end{aligned}
\end{equation}
defined as follows. Any object $(C/S,\mathsf{p},f,\rho)$ of $\WG(X)$ over a scheme $\underline{S}$ carries a basic fs logarithmic structure, which is obtained as the pullback, along a unique morphism $S \to S^{\fine}$, of a unique basic fine logarithmic map $(C^{\fine}/S^{\fine},\mathsf{p},f^{\fine},\rho)$. The morphism \eqref{eqn: map W} maps the family $(C/S,\mathsf{p},f,\rho)$ to $(C^{\fine}/S^{\fine},\mathsf{p},f^{\fine},\rho)$. 

\subsection{Basic monoids of stacks of well-spaced maps}

In this section we compute the basic monoid of the spaces $\WG(X)$ and $\WG(X)^{\mathsf{fine}}$. Since these are strict closed substacks of $\widetilde{\mathsf{ACGS}}_{\Gamma}(X)$ and $\widetilde{\mathsf{ACGS}}_{\Gamma}(X)^\fine$ respectively, it suffices to study the latter.

We begin with $\widetilde{\mathsf{ACGS}}{\Gamma}(X)^\fine$. Since $\mathfrak{M}_{1,n}^{\rad}$ is fs and is a logarithmic blow-up of $\mathfrak{M}_{1,n}$, the basic monoid of a point $(C,\mathsf{p},\rho)$ can be described as follows. Let $\tau$ be cone of $\Sigma(\mathfrak{M}_{1,n}^{\rad})$ corresponding to the stratum containing $(C,\mathsf{p}, \rho)$. Note that $\tau$ naturally lives in $\R_{\geq 0}^{\#\mathrm{nodes}}$. The basic monoid at $(C,\mathsf{p}, \rho)$  is then  
\[
\tau^\vee = \left\{ m \in M_{\tau} \;\middle|\; m(n) \ge 0 \text{ for all } n \in \tau \right\}
\]
where $M_\tau=N_\tau^*$.

The stack $\widetilde{\mathsf{ACGS}_{\Gamma}(X)}^\fine$ is defined as a cartesian product in $\mathsf{LS}^{\mathsf{fine}}$, so, by construction, the basic monoid at a point $(C^{\fine},\mathsf{p},f^{\fine},\rho)$ corresponding to a point $(C^{\fine},\mathsf{p},f^{\fine})$ in $\mathsf{ACGS}_{\Gamma}(X)^{\fine}$ with basic monoid $Q^{\mathsf{fine}}$ and a point $(C,\rho)$ in $\mathfrak{M}_{1,n}^{\rad}$ with basic monoid $\tau^{\vee}$ is the fine refinement of the monoid
\begin{equation}\label{eqn: basic monoid Wfine}
Q^{\mathsf{fine}} \oplus_{(\prod_q \mathbb{N})}\tau^{\vee} = \left[
\frac{\prod_{\eta} \oM_{X,\underline{f}(\eta)} \times \prod_q \mathbb{N}}{R}
\right] \oplus_{(\prod_q \mathbb{N})}\tau^\vee =\left[
\frac{\prod_{\eta} \oM_{X,\underline{f}(\eta)} \times \tau^\vee}{\overline{R}}
\right]
\end{equation}
where $\overline{R}$ is the image of $R$ under the morphism 
\begin{equation} \label{eqn: relations}
\prod_{\eta} \oM_{X,\underline{f}(\eta)}^{\gr} \times \prod_q \Z \to \prod_{\eta} \oM_{X,\underline{f}(\eta)} ^{\gr}\times(\tau^\vee)^{\gr}
\end{equation}

The monoid in Equation \eqref{eqn: basic monoid Wfine} is visibly fine, and thus that is the basic monoid at $(C^{\fine},\mathsf{p},f^{\fine},\rho)$ in $\widetilde{\mathsf{ACGS}_{\Gamma}(X)}^\fine$. When $(C^{\fine},\mathsf{p},f^{\fine},\rho)\in \WG(X)^{\mathsf{fine}}$, we will write $Q^{\mathsf{fine}}_{\W}$ to refer to it.

Finally, $\WG(X)$ is by definition the saturation of $\WG(X)^{\mathsf{fine}}$, thus the basic monoid at a point $(C,\mathsf{p},f,\rho)$ mapping to $(C^{\fine},\mathsf{p},f^{\fine},\rho)$ is 
\begin{equation}\label{eqn: basic monoid W(X)}
Q_{\mathsf{W}} = \Bigg[ \frac{Q^{\mathsf{fine}}_{\mathsf{W}} }{(Q^{\mathsf{fine}}_{\mathsf{W}}\big)_{\mathrm{tors}} } \Bigg]^{\sat} = \Bigg[ \frac{Q^{\mathsf{fine}}_{\mathsf{W}} }{(Q^{\mathsf{fine}}_{\mathsf{W}}\big)_{\mathrm{tors}} } \Bigg]^{\sat}= \left[
\frac{\prod_{\eta} \oM_{X,f(\eta)} \times \tau^\vee}{(\overline{R})^{\sat}}
\right]^{\sat}
\end{equation}
where $(\overline{R})^{\sat}$ denotes the saturation of $\overline{R}$ as a subgroup of $\prod_{\eta} \oM_{X,\underline{f}(\eta)}^{\gr} \times (\tau^\vee)^{\gr}$.

\begin{lemma}

Given a strict geometric point $ W= \mathrm{Spec}(Q_{\W}^{\mathsf{fine}} \to \C)$ in $\WG(X)^{\mathsf{fine}}$ and a point $(C,\mathsf{p},f, \rho)$ in $\WG(X)$ over it, there exists an étale and strict neighborhood $U \to \WG(X)^{\fine}$ and a commutative diagram
\begin{equation}\label{diagram: local description map W}
\begin{tikzcd}[row sep=1.2cm, column sep=1.5cm]
& \WG(X)_{U} \arrow[r] \arrow[d, "{\begin{array}{c}\mathrm{strict,} \\  \mathrm{smooth}\end{array}}"] & U  \arrow[d, "{\begin{array}{c}\mathrm{strict,} \\  \mathrm{smooth}\end{array}}"]   \\
\mathrm{Spec}(\C[Q_{\W}]) \arrow[r] & \mathrm{Spec}(\C[(Q_{\W}^{\mathsf{fine}})^{\sat}]) \arrow[r] & \mathrm{Spec}(\C[Q_{\W}^{\mathsf{fine}}])  \arrow[ul, phantom, "\square"] \\
(C,\mathsf{p},f, \rho)=\mathrm{Spec}(Q_{\W} \to \C) \arrow[r, hook] \arrow[u, hook] & W^{\mathrm{sat}} \arrow[r] \arrow[u, hook] & W \arrow[u, hook]
\arrow[ul, phantom, "\square"]
\end{tikzcd}
\end{equation}
where the squares on the right are cartesian of stacks and $\WG(X)_{U}=\WG(X)_{\times \WG(X)^{\fine}}U \to U$ is the base change of the map in Equation \eqref{eqn: map W}.

Moreover, if $W$ is a maximally degenerate point, the morphisms labelled smooth in diagram~\eqref{diagram: local description map W} can be taken to be étale.

\end{lemma}

\begin{proof}
    By~\cite[Corollary~3.3.4]{Ogus}, there exists an étale and strict neighbourhood $U$ of $W$ together with a strict and smooth morphism
\[
U \to \mathrm{Spec}(\C[Q_{\W}^{\mathsf{fine}}]).
\]
Since the morphism $U \to \WG(X)^{\fine}$ is strict, the base change
\[
\WG(X)_U := \WG(X) \times_{\WG(X)^{\fine}} U
\]
is independent of whether it is taken in the category of schemes, fine logarithmic schemes, or fs logarithmic schemes. In particular, $\WG(X)_U$ is an fs logarithmic scheme.

Moreover, $\WG(X)_U$ inherits a morphism to $\mathrm{Spec}(\C[Q_{\W}^{\mathsf{fine}}])$, and hence to its saturation $\mathrm{Spec}(\C[(Q_{\W}^{\mathsf{fine}})^{\sat}])$. This yields the diagram at the top, which is automatically cartesian as both horizontal arrows are the saturation morphisms.

It also follows from this that
\[
\WG(X)_U \to \mathrm{Spec}(\C[(Q_{\W}^{\mathsf{fine}})^{\sat}])
\]
is strict and smooth. 

Finally, when $W$ is a maximally degenerate point, since both $\WG(X)$ and $\WG(X)^{\fine}$ are logarithmically smooth (in their respective categories, $\mathsf{LS}$ and $\mathsf{LS}^{\fine}$), $W$ corresponds to a cone of maximal dimension, and therefore all smooth morphisms in diagram~\eqref{diagram: local description map W} are automatically étale.
\end{proof}

\subsection{The tropicalization map}

Let $\underline{S}=\mathrm{Spec}(K)$, where $K$ is a non-Archimedean extension of $\mathbb{C}$. Let $(C/S, \mathsf{p},f,\rho)$ be a minimal well-spaced logarithmic map to $X$, where $C$ is smooth over $K$. Since $\cW_\Gamma(X)$ is proper, after a base change, the map $\underline{S} \to \cW_\Gamma(X)$ extends over the valuation ring $R$ of $K$. Let $k$ be the residue field of $R$. Then we obtain a logarithmic map
\begin{equation}\label{eqn: map before tropicalization}
\mathrm{Spec}(G \to k) \to \cW_\Gamma(X)
\end{equation}
where $G = K^\times/R^\times \simeq \mathrm{val}(K^\times) \subseteq \mathbb{R}$ denotes the valuation group. By the universal property of minimality, the map in Equation~\eqref{eqn: map before tropicalization} factors through  $\mathrm{Spec}(Q_\W \to k)$ where $Q_\W$ denotes the minimal monoid of $\cW_\Gamma(X)$ at $\mathrm{Spec}(k) \to \cW_\Gamma(X)$, and thus we obtain a map $Q_\W \to G \subseteq \mathbb{R}$, which defines an element of $\Sigma(\cW_\Gamma(X))$. This defines a map
$$
\cW_\Gamma(X)^{\circ, \mathrm{an}} \to \Sigma(\cW_{\Gamma}(X)).
$$

The basic monoid $Q_{\W}$ has been computed in Equation~\eqref{eqn: basic monoid W(X)}. 
The dual cone
$
\mathrm{Hom}(Q_{\W}, \mathbb{R}_{\ge 0}) = (Q_{\W})^\vee_{\mathbb{R}}
$
naturally injects into the moduli space $M_{\Gamma}^{\mathrm{red}}(\Sigma(X))$ of aligned tropical curves in $\Sigma(X)$ introduced in \S\ref{sec: well-spaced tropical maps}. 
By \cite[Theorem~4.4.8]{RSWII} (also reported below in Theorem \ref{thm: rswii}), this is in fact a cone in $W_\Gamma(\Sigma(X))$. 
This yields a map
$
\Sigma(\mathcal{W}_{\Gamma}(X)) \to W_\Gamma(\Sigma(X))
$
of cone complexes, and we will refer to the composition
\[
\mathrm{trop} \colon \mathcal{W}_\Gamma(X)^{\circ, \mathrm{an}} 
\to \Sigma(\mathcal{W}_{\Gamma}(X)) 
\to W_\Gamma(\Sigma(X))
\]
as the tropicalization map.

\begin{definition}\label{def: weights}
    We endow the generalized cone complex $W_{\Gamma}(\Sigma)$ with a weight function $w$ that assigns to each maximal cone $\sigma$ the quantity
\[
w(\sigma)= \sum_i \frac{1}{|\mathrm{Aut}(f_{\sigma_i})|},
\]
where the sum ranges over all maximal cones $\sigma_i$ in $\Sigma(\cW_{\Gamma}(X))$ that map to $\sigma$,  $f_{\sigma_i}$ denotes the logarithmic map corresponding to $\sigma_i$, and $\mathrm{Aut}(f_{\sigma_i})$ denotes the automorphism group of the logarithmic map $f_{\sigma_i}$.
\end{definition}

\section{Lifting tropical curves to algebraic curves in every genus}\label{sec: realizability}

Let $\sigma$ be a maximal \footnote{ The maximality of $\sigma$ is not strictly necessary here; we assume it only to simplify the statement of Observation~\ref{obs: further properties from sigma}. While the case where $\sigma$ is maximal is sufficient for our purposes, the entire section can be generalized to remove this assumption. } cone in $M_{\Gamma}(\Sigma(X))$, and let $[h \colon \mathsf{C} \to \Sigma(X)]$ be a tropical curve in $\sigma$. In this section the genus $g$ is arbitrary. Recall that the tropical curve $[h \colon \mathsf{C} \to \Sigma(X)]$ is said to be \emph{realizable} if there exists a smooth curve $C$ over a non-Archimedean field $K$ equipped with a map $C \to X$ whose tropicalization is $[h \colon \mathsf{C} \to \Sigma(X)]$. %In particular, there is a well-spaced map to $X$ with basic monoid $\sigma^\vee= \mathsf{Hom}(\sigma \cap N_\sigma,\N)$.

In this section, we provide a criterion for the existence of a logarithmic map whose combinatorial type (as a logarithmic map, as opposed to an radially aligned logarithmic map) matches that of $[h \colon \mathsf{C} \to \Sigma(X)]$. Together with the following result, this yields a complete resolution to the tropical realizability problem in genus $1$: 

\begin{theorem}\cite[Theorem 4.4.8]{RSWII} \label{thm: rswii}
Let $[\mathsf{C} \to \Sigma(X)]$ be a tropical stable map of genus $1$, and assume there is a minimal logarithmic map $f:(C,\mathsf{p}) \to X$ whose combinatorial type is that of $[\mathsf{C} \to \Sigma(X)]$. Then $[\mathsf{C} \to \Sigma(X)]$ is realizable if and only if it is well-spaced.
\end{theorem}

The first attempt to classify realizable genus $1$ curves is due to Speyer~\cite{Speyer}. As explained in~\cite[Theorem~C]{Ran17} and~\cite[Figure~5]{RSWII}, Speyer’s condition is sufficient but not necessary for realizability.

Before stating our result regarding the existence of $f:(C,\mathsf{p}) \to X$ matching the tropical type $[h \colon \mathsf{C} \to \Sigma(X)]$ we require some discussion. Suppose that $f:C/S \to X$ is a minimal logarithmic map in $\mathsf{ACGS}_{\Gamma}(X)$ whose combinatorial type is that of $[\mathsf{C} \to \Sigma(X)]$. Here $S=\mathrm{Spec}( Q=\sigma^\vee \to \C)$.

\begin{observation}\label{disc:properties-over-sigma}
The following facts are clear:
\begin{enumerate}[label=(\arabic*)]
    \item The dual graph of $C$ is the graph underlying $\mathsf{C}$. 
    
    \item Suppose that $h(V)$ lies in a cone $\tau$. Since $\sigma$ is maximal, the cone $\tau$ is either maximal in $\Sigma(X)$ or has codimension $1$. For each component $C_V$, the restriction $\underline{f}|_{C_V}$ is, up to an isomorphism of $C_V$, either the constant map to the torus-fixed point $V(\tau)$ when $\tau$ is maximal, or a degree $d_V$ cover of the torus-invariant curve $V(\tau)$, totally ramified over its torus-fixed points, when $\tau$ has codimension $1$ in $\Sigma(X)$.
\end{enumerate}
\end{observation}

The composition 
\begin{equation}\label{eqn: map M to Pic}
M_X= N_X^* \to \Gamma(X, \mathcal{M}^{\gr}_X) \to \Gamma(C, \cM^{\gr}_{C}) \to \Gamma(C, \oM^{\gr}_{C}) \to H^1(C, \cO_C^{\times})
\end{equation}
must vanish being the last two maps in the cohomology exact sequence associated to 
$$
0 \to \cO_C^{\times} \to \cM^{\gr}_C \to \oM^{\gr}_C \to 0
$$
For every log scheme $(Z,\oM_Z)$, the map $\Gamma(Z, \oM^{\gr}_{Z}) \to H^1(Z, \cO_Z^{\times})$ can explicitly be described as follows. Given a global section $m \in \Gamma(Z, \oM^{\gr}_{Z}) $, the inverse image of $m$ under $\cM^{\gr}_Z \to \oM^{\gr}_Z$ is a $\cO_Z^{\times}$-torsor $\cO_Z(-m)^{\times}$ and the image of $m$ is its associated line bundle $\cO_C(-m)$. We have a natural section $\sigma_m \in \Gamma(Z,\cO_Z(m))$ obtained from the composition $\cO_Z(-m)^{\times} \subseteq \cM_Z \to \cO_Z$.

In the case of a nodal curve, the space of sections $\Gamma(C,\oM_C)$ is naturally identified with the space of piecewise linear functions on $\mathsf{C}$, that is we have an isomorphism
\begin{equation}\label{eqn: sections barM}
\Gamma(C, \oM_C^{\gr})
\;\xrightarrow{\sim}\;
\left\{
\begin{array}{c}
\{m_V \in Q^{\gr}\}_V \\
\{m_i \in \mathbb{Z}\}_{p_i} \ \{n_j \in \mathbb{Z}\}_{z_j}
\end{array}
\ \middle|\ 
\begin{aligned}
&\forall \vec{e}: V \to V',\ \exists!\, m_{\vec{e}} \in \mathbb{Z} \\
&\text{such that } m_{V'} = m_V + m_{\vec{e}} \rho_e
\end{aligned}
\right\}.
\end{equation}

Given $m \in \Gamma(C, \oM_C)$, let $m_V$, $m_i$, and $m_{\vec{e}}$ be as in Equation \eqref{eqn: sections barM}. The restriction $\mathcal{O}_C(m)|_{C_V}$ is isomorphic to the line bundle
\begin{equation}\label{eqn: line bdl LmV}
L_{m,V} := \pi^*\mathcal{O}_S(m_V)\big|_{C_V} \otimes \mathcal{O}_{C_V}(D_{m,V}),
\end{equation}
where $D_{m,V}$ is the Cartier divisor
\begin{equation}\label{eqn: divisor DmV}
D_{m,V} := \sum_{p_i \in C_V} m_i \cdot p_i \;+\; \sum_{z_j \in C_V} n_j \cdot z_j \;+\; \sum_{q \in C_V} m_{\vec{e}_q} \cdot q.
\end{equation}
Here $e_q$ denotes the edge corresponding to the node $q$, oriented away from $V$. Moreover, the restriction $\sigma_{m,V} := \sigma_m|_{C_V}$ is a section of $L_{m,V}$.

For us, $S$ is a point; thus $\mathcal{O}_S(m_V)$ is isomorphic to the trivial line bundle. Moreover, when $m$ comes from $M_X$ in Equation \eqref{eqn: map M to Pic}, the section $\sigma_m$ is a trivializing section. Consequently, $\sigma_{m,V}$ is also a trivializing section of $L_{m,V}$, and when regarded as a meromorphic section on $C_V$ (only defined up to scalar), its associated divisor is precisely $D_{m,V}$. In particular, $D_{m,V}$ is effective and $L_{m,V}$ is isomorphic to $\cO_{C_V}$.

\begin{remark}\label{rmk: computation_mV,mi,nj}
    The map $M_X \to \Gamma(C,\oM_C^{\gr})$ is completely determined by the cone $\sigma$ and is given as follows. Under the identification \eqref{eqn: sections barM}, given $m \in M_X$, the associated collections $\{m_V\}_V \subseteq Q^{\gr}, \{m_i\}_i , \{n_j\}_j\subseteq \Z$ are given by
    \begin{equation}\label{eqn: section m on C}
        m_V([h \colon \mathsf{C} \to N_X]) = m(h(V)), \quad m_i = m(u_{p_i})=0,  \quad n_j=m(u_{z_j})
    \end{equation}
    for every $[h \colon \mathsf{C} \to N_X] \in \sigma= \mathrm{Hom}(Q,\N)$. In particular, both $L_{m,V}$ and $D_{m,V}$ in Equations \eqref{eqn: line bdl LmV} and \eqref{eqn: divisor DmV} only depend on $\sigma$ and the pointed domain curve.
\end{remark}

For an oriented edge $\vec{e}: V \to V'$, denote by
$$
H_{C,q_e} := T_{C_V, q_{V}} \otimes 
T_{C_V',q_{V'}}$$
where $q_{V}$ and $q_{V'}$ are the preimages of the node $q_e$ in $C_V$ and $C_{V'}$ respectively. Notice that there is a canonical isomorphism \footnote{This is standard. If $\underline{S'}$ is a versal family around $\underline{S}$ in $\mathfrak{M}_{1,n+m}$, then $H_{C,q_e}$ is canonically identified with the normal bundle of the branch $B$ of the component of the normal crossing divisor $\partial \mathfrak{M}_{1,n+m}$ corresponding to the node $e$. The ghost sheaf of $\mathfrak{M}_{1,n+m}$ at $\underline{S}$ is $\prod_{\mathrm{edges}} \mathbb{N}$, and the section $1_e$ yields an isomorphism $\mathcal{N}_{B/\mathfrak{M}_{1,n+m}} \simeq \cO_B(B) \simeq \mathcal{O}_{\mathfrak{M}_{1,n+m}}(1_e)$ whose restriction to $\underline{S}$ is precisely $\Lambda_e$ in Equation \eqref{eqn: Lamndae}.} of vector spaces
\begin{equation}\label{eqn: Lamndae}
\Lambda_e: H_{C,q_e} \xrightarrow{\sim} \cO_{S}(\rho_e)
\end{equation}

Then, the line bundle $L_m$ and the section $\sigma_m$ are obtained by gluing $L_{m,V}|_{q_V}$ to $L_{m,V'}|_{q_{V'}}$, and $\sigma_{m,V}(q_V)$ to $\sigma_{m,V'}(q_{V'})$, along the isomorphism \begin{equation}\label{eqn: glueing_LmV}
    L_{m,V}|_{q_V} \xrightarrow{\sim} L_{m,V'}|_{q_{V'}}
\end{equation}
corresponding to $\Lambda_e^{\otimes m_{\vec{e}}}$ under the sequence of isomorphisms
\begin{equation}\label{eqn: glueing LmV}
   \mathrm{Hom}(L_{m,V}|_{q_V}, L_{m,V'}|_{q_{V'}}) \simeq \mathrm{Hom}(\cO_S(m_V) \otimes T_{C_V,q_V}^{\otimes m_{\vec{e}}}, \cO_S(m_{V'}) \otimes T_{C_{V'},q_{V'}}^{\otimes -m_{\vec{e}} }) \simeq \mathrm{Hom}(H_{C,q_e}, \cO_S(\rho_e))^{\otimes m_{\vec{e}}}. 
\end{equation}

Let

\begin{equation}\label{eqn: iso psi}
\psi_{\vec{e},m}: H^0(C_V',L_{V',m}) \xrightarrow{\sim} H^0(C_{V},L_{V,m})
\end{equation}
be the unique isomorphism making the following diagram 
\[
\begin{tikzcd}
H^0(C_{V'}, L_{m,V'}) 
\arrow[r, "\psi_{\vec{e},m}"] 
\arrow[d, "\simeq", "ev_{q_{V'},m}"']
& H^0(C_V, L_{m,V}) 
\arrow[d, "\simeq", "ev_{q_V,m} \otimes \mathrm{Id}"'] \\
L_{m,V'}\big|_{q_e} 
\arrow[d, equals]
& L_{m,V}\big|_{q_e} 
\arrow[d, equals] \\
 \cO_S(m_{V'}) \otimes T^{\otimes -m_{\vec{e}}}_{C_{V'},q_{V'}}
\arrow[r, "\sim"] 
& \cO_S(m_V) \otimes  T^{\otimes m_{\vec{e}}}_{C_V,q_V}
\end{tikzcd}
\]
commute. 

Let $x_{q_V}$ and $x_{q_{V'}}$ be local coordinates on $C_V$ and $C_{V'}$ around $q_V$ and $q_{V'}$, respectively, and write
\begin{equation}\label{eqn: coordinates for sigmas}
\mathrm{ev}_{q_V,m}(\sigma_{m,V})
= \nu_{m,V} \otimes (dx_{q_V})^{\otimes - m_{\vec{e}}}
\quad \text{and} \quad
\mathrm{ev}_{q_{V'},m}(\sigma_{m,V'})
= \nu_{m,V'} \otimes (dx_{q_{V'}})^{\otimes m_{\vec{e}}},
\end{equation}
where $0 \neq \nu_{m,V} \in \cO_S(m_{V})$ and $ 0 \neq \nu_{m,V'} \in \cO_S(m_{V'})$. Set
\[
D_e := \left.\frac{\partial}{\partial x_{q_V}}\right|_{q_V}
\otimes 
\left.\frac{\partial}{\partial x_{q_{V'}}}\right|_{q_{V'}}
\in H_{C,q_e}.
\]
Then \(0 \neq D_e^{\otimes m_{\vec e}} \in H_{C,q_e}^{\otimes m_{\vec e}}\), and by definition one has
\[
\psi_{\vec e,m}(\sigma_{m,V'})
= \sigma_{m,V'} 
\]
if and only if
\begin{equation}\label{eqn: constraint for number of maps}
\Lambda_e^{\otimes m_{\vec{e}}}(D_e)=\nu_{m,V} \otimes \nu_{m,V'}^{-1}.
\end{equation}

\begin{theorem}[Lifting condition]\label{thm: realizability genus 1} 
Let $\underline{f} \colon (\underline{C}, \mathsf{p}) \to X$ be a map from a genus $g$ nodal curve to $X$ satisfying properties (1) and (2) of Observation~\ref{disc:properties-over-sigma}. Let $Q = \mathrm{Hom}_{\mathsf{Mon}}(\sigma \cap N_\sigma, \N)$, and let $\cM_C$ be the logarithmic structure on $\underline{C}$ as determined by Definition~\ref{def: stable log map}(1) and Remark~\ref{rmk: local structure at node}. Let $D_{m,V}$ and $L_{m,V}$ be the divisors and the line bundles, respectively, defined by Equations~\eqref{eqn: line bdl LmV} and~\eqref{eqn: divisor DmV}. Then, the following are equivalent:
\begin{enumerate}[label=(\arabic*)]
\item there exists an enhancement of $\underline{f} \colon (\underline{C},\mathsf{p}) \to X$ to a minimal fs logarithmic map whose combinatorial type is that of $\sigma$;
\item for every $m \in M_X$, there exist trivializing sections $\sigma_{m,V} \in \Gamma\bigl(C_V, L_{m,V}\bigr)$ on each component $C_V$ of $\underline{C}$ such that, for every oriented edge $\vec{e}$, Equation~\eqref{eqn: constraint for number of maps} is satisfied for some element $0 \neq D_e \in H_{C,q_e}$ (independent of $m$);
\item for every $m \in M_X$, there exist meromorphic sections $\{ [\nu_{m,V}] \}_V$, defined up to a scalar, which admit lifts $\nu_{m,V}$ such that:
\begin{enumerate}
\item $\mathrm{div}(\nu_{m,V}) = D_{m,V}$ for all $m \in M_X$ and $V \in V(\mathsf{C})$;
\item for every oriented edge $\vec{e}$, there exist coordinates $x_{q_V}$ and $x_{q_{V'}}$ at $q_V \in C_V$ and $q_{V'} \in C_{V'}$ (independent of $m$) such that the lowest-order coefficients of $\nu_{m,V}$ in the coordinate $x_{q_V}$ and of $\nu_{m,V'}$ in the coordinate $x_{q_{V'}}$ agree.
\end{enumerate}
\end{enumerate}
\end{theorem}

\begin{remark}
    The divisors $D_{m,V}$ have degree $0$. Therefore, whenever the classes $[\nu_{m,V}]$ exist, they are uniquely determined.
\end{remark}

\begin{remark} \label{rem: basis}
In Theorem~\ref{thm: realizability genus 1}(2)--(3), the condition ``for every \(m \in M_X\)'' may be weakened to ``for every \(m_i\) in a basis of \(M_X\)''. Indeed, if \(m = \sum_i a_i m_i\), then
\[
\nu_{m,V} := \prod_i \nu_{m_i,V}^{a_i}
\]
satisfies conditions~\((a)\) and~\((b)\) in~\((3)\). The same argument applies to \(\sigma_{m,V}\).
\end{remark}

The following definitions will be useful in the next sections:

\begin{definition} \label{def: lifting condition a}
    Fix $m \in M_X$ and a vertex $V$ of $\mathsf{C}$. Say that the meromorphic function $\nu_{m,V}$ on $C_V$ satisfies lifting condition (a) if $\mathrm{div}(\nu_{m,V}) = D_{m,V}$. 
\end{definition}

\begin{definition} \label{def: lifting condition b}
    Fix a oriented edge $\vec{e}:V \to V'$, and coordinates $x_{q_V}$ and $x_{q_{V'}}$ at $q_V \in C_V$ and $q_{V'} \in C_{V'}$. Fix $m \in M_X$. We say that a pair of meromorphic functions $\nu_{m,V}, \nu_{m, V'}$ on $C_V$ and $C_{V'}$ respectively satisfies lifting condition (b) with respect to $x_{q_V}, x_{q_{V'}}$ if the lowest-order coefficients of $\nu_{m,V}$ in the coordinate $x_{q_V}$ and of $\nu_{m,V'}$ in the coordinate $x_{q_{V'}}$ agree. 
\end{definition}

Before starting the proof of Theorem \ref{thm: realizability genus 1},  we need to recall some standard facts about the sheaves $\mathbb{G}_{\log}\otimes N_X$ and $\mathbb{G}_{\mathrm{trop}}\otimes N_X$. 

Let $\mathbb{G}_{\log} \otimes N_X$ (resp. $\mathbb{G}_{\mathrm{trop}} \otimes N_X$) be the sheaf on the Big Strict Étale Site of $\mathsf{LS}$ defined as the sheafification of the presheaf $(U,\mathcal{M}_U) \mapsto \mathrm{Hom}(M_X,\Gamma(U,\mathcal{M}_U^{\mathrm{gp}}))$ (resp. $(U,\mathcal{M}_U) \mapsto \mathrm{Hom}(M_X,\Gamma(U,\overline{\mathcal{M}}_U^{\mathrm{gp}}))$). Then we have an exact sequence of sheaves
$$0 \to \mathbb{G}_m\otimes N_X \to \mathbb{G}_{\mathrm{log}}\otimes N_X \to \mathbb{G}_{\mathrm{trop}}\otimes N_X \to 0$$
and a cartesian diagram in the category of stacks over $\mathsf{LS}$
\begin{equation}\label{eqn: cartesian square Glog}
\begin{tikzcd}
X \arrow[r] \arrow[d] & \mathcal{A}_{X} \arrow[d] \\
\mathbb{G}_{\mathrm{log}} \otimes N_X \arrow[r] & \mathbb{G}_{\mathrm{trop}} \otimes N_X
\end{tikzcd}
\end{equation}

\begin{proof}[Proof of Theorem \ref{thm: realizability genus 1}]
    The discussion preceding the statement of Theorem \ref{thm: realizability genus 1} demonstrates that $(1) \implies (2)$. The implication $(2) \implies (3)$ is trivial. To prove $(3) \implies (2)$, take $\sigma_{m,V} \in H^0(C_V,L_{m,V})$ to be the unique section of $L_{m,V}$ corresponding to $\nu_{m,V'} \otimes (dx_{q_{V'}})^{\otimes m_{\vec{e}}}$ under the identification $H^0(C_V,L_{m,V}) \xrightarrow{\sim} L_{m,V}$ given by the evaluation map (see Equation \eqref{eqn: coordinates for sigmas}). It remains to prove $(2) \implies (1)$. As discussed in \S\ref{sec: trop maps}, the cone $\sigma$ yields a family of tropical curves in $\Sigma(X)$:
    \[
        \begin{tikzcd}
        & \Sigma(C) \arrow[r] \arrow[d] & \Sigma(X) \\
        & \sigma 
    \end{tikzcd}
    \]
     Such family is precisely the data of an element of 
    $$
    \mathrm{Hom}_{\mathsf{GenCones}}(\Sigma(C), \Sigma(X))=\mathrm{Hom}_{\mathsf{fs}}(C, \mathcal{A}_X)
    $$
    (see \cite[Proposition 2.10]{ACGSI}) and thus we get a map $C \to \mathcal{A}_C \to \mathcal{A}_X$. Under the identification \eqref{eqn: sections barM}, the corresponding map $C \to \mathbb{G}_{\mathrm{trop}} \otimes N_X$ is given by the morphism $M_X \to \Gamma(C,\oM_C^{\gr})$ associating to $m \in M_X$ the collection $\{m_V\}_V \subseteq Q^{\gr}, \{m_i\}_{p_i}, \{n_j\}_{z_j} \subseteq \Z$ in Equation \eqref{eqn: section m on C}. By the Cartesian square \eqref{eqn: cartesian square Glog}, it suffices to show that the data of $\sigma_{m,V}$ yield a lift of this map to $\mathbb{G}_{\mathrm{log}} \otimes N_X$; that is, a factorization the map $M_X \to \Gamma(C, \oM_C^{\gr})$ through $\Gamma(C, \mathcal{M}_C^{\gr})$. Equivalently, given the exact sequence$$\Gamma(C, \mathcal{M}_C^{\gr}) \to \Gamma(C, \oM_C^{\gr}) \to H^1(C, \mathcal{O}^{\times}),$$ it is enough to show that for every $m \in M_X$, the data of $\sigma_{m,V}$ yields a trivialization of the line bundle associated to $\{m_V\}_V \subseteq Q^{\gr}$ and $\{m_i\}_{p_i}, \{n_j\}_{z_j} \subseteq \mathbb{Z}$. This line bundle is obtained by gluing the $L_{m,V}$ along the isomorphisms $L_{m,V}|_{q_V} \xrightarrow{\sim} L_{m,V'}|_{q_{V'}}$ corresponding to $\Lambda_e^{\otimes m_{\vec{e}}}$ under the identification \eqref{eqn: glueing LmV}. By assumption, the $\sigma_{m,V}$ also glue (via the same identifications) to a section of this line bundle, and this section necessarily trivializes it, given that each $\sigma_{m,V}$ trivializes $L_{m,V}$. This concludes the proof.
\end{proof}

\begin{remark} \label{rmk: stronger_version_realizability}
In fact, the proof of Theorem~\ref{thm: realizability genus 1} shows that the data in \textup{(1)} and \textup{(2)} of the statement are equivalent, namely:
\begin{enumerate}[label=(\arabic*)]
    \item an enhancement of $\underline{f} \colon (\underline{C},\mathsf{p}) \to X$ to a minimal fs logarithmic map whose combinatorial type is that of $\sigma$;
    \item for every $m \in M_X$, a collection of trivializing sections
    \[
    \sigma_{m,V} \in \Gamma\bigl(C_V, L_{m,V}\bigr)
    \]
    on each irreducible component $C_V$ of $\underline{C}$ such that, for every oriented edge $\vec{e}$, Equation~\eqref{eqn: constraint for number of maps} is satisfied for some element $D_e \neq 0$ of $H_{C,q_e}$, independently of $m$.
\end{enumerate}
\end{remark}

Using Theorem~\ref{thm: realizability genus 1}, one can construct an example of a  maximal cone $\sigma$ of well-spaced tropical curves with $w(\sigma)=0$ as follows.

\begin{example}\label{example: weight zero example}
    Let $X = \mathbb{P}^1$, and consider the degree $1$, genus $1$ tropical map to $\Sigma(X)$ illustrated in Figure \ref{fig:nonrealizable}. There is a stable map $\underline{f}:(\underline{C}, \mathsf{p}) \rightarrow \mathbb{P}^1$ satisfying properties (1) and (2) of Observation \ref{obs: further properties from sigma}, but this cannot be enhanced to a logarithmic map. Indeed, taking $m=1 \in M_{\P^1}=N_{\P^1}^* \simeq \N$ we see that there is no meromorphic function $\nu_{1,V_1}$ on $C_{V_1}$ which is degree $1$ and taking the same value at two distinct points. 
    
    The tropical map in Figure \ref{fig:nonrealizable} is well-spaced and the corresponding cone $\sigma \subseteq W_{\Gamma}(\Sigma(\P^1))$ has weight $w(\sigma)=0$.
    
    \begin{figure}[h]
        \centering
        \begin{tikzpicture}
            \draw (-4, 0) -- (2,0); 
            \draw[dashed] (0,0)  to[in=50,out=130,loop, style={min distance=12mm}] (0,0);
            \draw[->] (-2, -0.5) -- (-2,-1); 
            \draw (-4,-1.5) -- (2, -1.5);
            \node[forestgreen] at (0.5,0.2) {$1$};
            \node[forestgreen] at (-0.5,0.2) {$1$};
            \draw[forestgreen, ->] (0,0) -- (-1,0);
            \draw[forestgreen, ->] (0,0) -- (1,0);
            \node at (0, -0.3) {$V_1$};
            \fill (0,0) circle[radius=2pt];
            \fill (-2,0) circle[radius=2pt];
            \fill (-2,-1.5) circle[radius=2pt];
        \end{tikzpicture}
        \caption{A genus one tropical curve which does not have a logarithmic curve which tropicalizes to it.}
        \label{fig:nonrealizable}
    \end{figure}
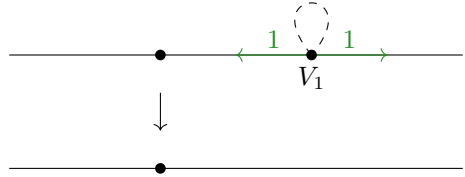

\end{example}

\section{Pullbacks in Chow under logarithmic maps}

In this section, we prove Proposition \ref{prop: technical tool}.

\begin{lemma}
    Let $X$ be a logarithmically smooth Deligne--Mumford stack, and let 
    $D_1,\ldots,D_k$ be $\mathbb{Q}$-Cartier divisors on $X$ corresponding to rays in $\Sigma(X)$ that form a simplicial cone $\sigma$.
    Then
    \[
    [D_1]\cdot \ldots \cdot [D_k]
    = \frac{1}{\mathrm{mult}(\sigma)} [V(\sigma)],
    \]
    where $\mathrm{mult}(\sigma)$ denotes the index of the sublattice generated by the primitive generators of the rays of $\sigma$ inside $ N_\sigma$.
\end{lemma}

\begin{proof}
    When $X$ is a complete simplicial toric variety, this is \cite[Lemma~12.5.2]{CLS}. Note, however, that in that lemma the assumption that $X$ is complete is not necessary. In general, by \cite[Corollary 3.3.4]{Ogus} around each generic point of the intersection $D_1 \cap \cdots \cap D_k$, étale locally $X$ admits a strict smooth morphism
    $
    X \to U_\sigma
    $
    to the simplicial affine toric variety
    $
    U_\sigma = \mathrm{Spec}\big(\mathbb{C}[\sigma^\vee]\big)
    $, where we recall that $\sigma^\vee = \{ m \in M_\sigma \mid \langle m, v \rangle \ge 0 \text{ for all } v \in \sigma \}$. This concludes.
\end{proof}

\begin{proof}[Proof of Proposition \ref{prop: technical tool}]

For every ray $\rho$ of $\Sigma(Y)$ (resp. $\Sigma(X)$), denote by $D_\rho$ 
the corresponding divisor in $\mathcal{A}_Y$ (resp. $\mathcal{A}_X$).
 Since $Y$ is smooth, $\tau$ is spanned by precisely $k=\dim \tau$ rays $\rho_1,\ldots,\rho_k$. Since we are assuming the existence of the commutative diagram \eqref{eqn: artin-fan-diagram}, the map $X \to Y \to \mathcal{A}_Y$ factors as
$
X \longrightarrow \mathcal{A}_X \longrightarrow \mathcal{A}_Y.
$
The first map is smooth by assumption and the second is always logarithmically étale. Thus the morphism $X \to \mathcal{A}_Y$ is log flat and integral. In particular it is flat by \cite[Proposition~2.3]{Mirror-GS}, and the pullback to $X$ of every $\mathbb{Q}$-Cartier divisor $D_\rho$ is again $\mathbb{Q}$-Cartier.

For every ray $\rho$ in $\Sigma(Y)$, we write
\[
\mathcal{A}(f)^* D_{\rho}= \sum_i x_{\rho,i} D_{\rho,i}
\]
for certain integers $x_{\rho,i} \in \mathbb{Z}$. Then we can compute
\begin{align*}
f^* [V(\tau)] 
&= \prod_{i=1}^k f^*D_{\rho_i} \\ 
&= \varphi^* \bigg( \prod_{i=1}^k  \mathcal{A}(f)^*D_{\rho_i} \bigg) \\
&= \varphi^* \bigg( \prod_{i=1}^k \Big(\sum_j x_{\rho_i,j}D_{\rho_i,j}\Big)\bigg)\\
&= \sum_{\sigma} \frac{\Big(\prod_{\rho \in \sigma} x_\rho\Big)}{\mathrm{mult}(\sigma)} [V(\sigma)],
\end{align*}
where the sum runs over the cones $\sigma$ of $\Sigma(X)$ mapping onto to $\tau$. Indeed, if a collection of divisors $D_{\rho,i}$ does not correspond to a cone, their product is zero. For a nonzero product of $k$ divisors $D_{\rho,i}$ above, the corresponding rays must form a $k$-dimensional cone of $\Sigma(X)$ and must map to the $k$ rays of $\tau$. Since each cone $\sigma$ appearing in the sum maps surjectively onto $\tau$ and has the same dimension as $\tau$, it is automatically simplicial.

The claim now follows from the chain of identities
\[
\Big(\prod_{\rho \in \sigma} x_\rho\Big)
= [  N_\tau: \langle \Sigma(f)(u_\rho) \ | \ \rho \in \sigma \rangle]
= [ N_\sigma: \langle u_\rho \ | \ \rho \in \sigma \rangle ] \cdot [ N_\tau: \Sigma(f)(N_\sigma)]
=  \mathrm{mult}(\sigma) [ N_\tau: \Sigma(f)(N_\sigma)]
\]
valid for all $\sigma$ appearing in the sum.  Here $\langle - \rangle$ denotes the $\Z$-span and $u_\rho$ is the primitive generator of the ray $\rho$. The first equality uses the fact that $\tau$ is a smooth cone (that is, it is simplicial and $\mathrm{mult}(\tau)=1$), the second uses that 
$\Sigma(f)$ is injective as a map $N_\sigma \to N_\tau$, and the last equality 
follows from the definition of $\mathrm{mult}(\sigma)$.
\end{proof}

\section{Proof of the correspondence theorem, Part I}\label{sec: correspondence theorems}

In this section, we prove Part I of the correspondence theorem (Theorem~\ref{thm: tropical correspondence}). In order to apply Proposition~\ref{prop: technical tool}, we must establish the commutativity of the diagram~\eqref{eqn: artin-fan-diagram} associated with the map~\eqref{eqn: map}. This is not a purely formal verification, as there are known examples where such a commutative diagram fails to exist (see, e.g., \cite[Section~4.6.2]{ACMUW}). The proof of commutativity is provided in the next subsection.

\subsection{ Tropical and algebraic evaluation maps}\label{sec: key commutative diagram}

Notice that there is a natural map
\[
\mathcal{A}_{\WG(X)} \to \mathcal{A}_{\Ev}= \mathcal{A}_X^n \times \mathcal{A}_{\oM_{1,1}}
\]
to consider. For this, one would like to invoke \cite[Proposition~2.10]{ACGSI} stating that
\begin{equation}\label{eqn: ACSI-functoriality-Artin-fans}
\mathrm{Hom}_{\mathsf{fs}}(\mathcal{A}_{\WG(X)}, \mathcal{A}_{\Ev})
=
\mathrm{Hom}_{\mathsf{GenCones}}(\Sigma(\WG(X)), \Sigma(\Ev)).
\end{equation}
Then, composing the tropical evaluation map together with the tropical $j$-invariant defines a map
\[
W_{\Gamma}(\Sigma) \to \Sigma(\Ev).
\]
with the map $\Sigma(\WG(X)) \to W_{\Gamma}(\Sigma)$ in Equation \eqref{eqn: tropicalization map}, we obtain a map $\mathcal{A}(\mathsf{ev} \times \mathsf{st}): \mathcal{A}_{\WG(X)} \to \mathcal{A}_{\Ev}$. This provides a natural candidate for making the diagram
\begin{equation}\label{eqn: our artin-fan-diagram}
\begin{tikzcd}
\WG(X) 
\arrow[r, "\mathsf{ev} \times \mathsf{st}"] 
\arrow[d, "\varphi"']               
& \Ev
\arrow[d, "\psi"]     \\
\mathcal{A}_{\WG(X)} 
\arrow[r, "\mathcal{A}(\mathsf{ev} \times \mathsf{st})"] 
& \mathcal{A}_{\Ev}
\end{tikzcd}
\end{equation}
commute.

However, we cannot directly apply \cite[Proposition~2.10]{ACGSI}, as that result requires $\mathcal{A}_X$ to be an fs log scheme. Nevertheless, as the following proposition shows, Equation \eqref{eqn: ACSI-functoriality-Artin-fans} remains valid in our setting.

\begin{proposition}\label{prop:improved-2.10-ACGSI}
Let $Y$ be a Zariski fs log scheme, logarithmically smooth over
$\mathrm{Spec} k$. Let $T$ be a finite-type fs logarithmic algebraic stack.
Then there is a canonical bijection
\[
\mathrm{Hom}_{\mathsf{fs}}(T,\mathcal{A}_Y)
\longrightarrow
\mathrm{Hom}_{\mathsf{Cones}}\bigl(\Sigma(T),\Sigma(Y)\bigr),
\]
which is functorial in $T$.
\end{proposition}

\begin{proof}

We first observe that logarithmic morphisms to $\mathcal{A}_Y$ satisfy descent as a sheaf of sets for the strict smooth topology.

In general, defining a morphism to a stack via descent requires a strict smooth morphism $f \colon T' \to A_Y$, a 2-isomorphism $\alpha \colon p_1^* f \Rightarrow p_2^* f$ over $T'' = T' \times_T T'$, and the satisfaction of a cocycle condition on $T''' = T' \times_T T' \times_T T'$. However, because the morphism $A_Y \to \mathsf{Log}$ is representable, morphisms  $T \to A_Y$ over $\mathsf{Log}$ have no non-trivial 2-automorphisms. Consequently, such a 2-isomorphism $\alpha$ is unique if it exists, and the cocycle condition holds automatically. It follows that the functor of logarithmic morphisms to $Y$ is a sheaf of sets with respect to the strict smooth topology.

Thus, if $S'\to S$ is a strict smooth surjective morphism and
$
S''=S'\times_S S',
$
then
\[
\mathrm{Hom}_{\mathsf{fs}}(S,\mathcal{A}_Y)
=
\mathrm{Eq} \left(
\mathrm{Hom}_{\mathsf{fs}}(S',\mathcal{A}_Y)
\rightrightarrows
\mathrm{Hom}_{\mathsf{fs}}(S'',\mathcal{A}_Y)
\right).
\]

We now prove the statement in two steps.

First suppose that $T$ is an fs logarithmic algebraic space. Choose a
strict \'{e}tale surjective morphism
$
T'\longrightarrow T
$
from an scheme $T'$, and set
$
T''=T'\times_T T'.
$
Since $T$ is an algebraic space, its diagonal is representable by
schemes, and hence $T''$ is a scheme. We equip $T'$ and $T''$ with
the logarithmic structures pulled back from $T$. In particular, both
projections
$
T''\rightrightarrows T'
$
are strict.

By strict \'{e}tale descent,
\[
\mathrm{Hom}_{\mathsf{fs}}(T,\mathcal{A}_Y)
=
\mathrm{Eq}\left(
\mathrm{Hom}_{\mathsf{fs}}(T',\mathcal{A}_Y)
\rightrightarrows
\mathrm{Hom}_{\mathsf{fs}}(T'',\mathcal{A}_Y)
\right).
\]
By \cite[Proposition 2.10]{ACGSI} applied to the
fs log schemes $T'$ and $T''$, there are canonical functorial
bijections
\[
\mathrm{Hom}_{\mathsf{fs}}(T',\mathcal{A}_Y)
\simeq
\mathrm{Hom}_{\mathsf{Cones}}\bigl(\Sigma(T'),\Sigma(Y)\bigr)
\qquad \text{and} \qquad \mathrm{Hom}_{\mathsf{fs}}(T'',\mathcal{A}_Y)
\simeq
\mathrm{Hom}_{\mathsf{Cones}}\bigl(\Sigma(T''),\Sigma(Y)\bigr).
\]
It follows that
\[
\mathrm{Hom}_{\mathsf{fs}}(T,\mathcal{A}_Y)
\simeq
\mathrm{Eq}\left(
\mathrm{Hom}_{\mathsf{Cones}}\bigl(\Sigma(T'),\Sigma(Y)\bigr)
\rightrightarrows
\mathrm{Hom}_{\mathsf{Cones}}\bigl(\Sigma(T''),\Sigma(Y)\bigr)
\right).
\]
By definition (see for example \cite[Section 2.1.4]{ACGSI}),
\[
\Sigma(T)
=
\mathrm{colim}\left(
\Sigma(T')\rightrightarrows\Sigma(T'')
\right).
\]
Therefore, by the universal property of the colimit,
\[
\mathrm{Hom}_{\mathsf{fs}}(T,\mathcal{A}_Y)
\simeq
\mathrm{Hom}_{\mathsf{Cones}}\bigl(\Sigma(T),\Sigma(Y)\bigr).
\]
This proves the result when $T$ is an algebraic space.

Now let $T$ be an fs logarithmic algebraic stack. Choose a
smooth surjective morphism
$
T'\longrightarrow T
$
from a scheme $T'$, and set
$
T''=T'\times_T T'.
$
The logarithmic structures on $T'$ and $T''$ are the pullbacks of the
logarithmic structure on $T$. Since the diagonal of an algebraic stack
is representable by algebraic spaces, $T''$ is an algebraic space. We now apply the argument above, replacing the scheme-level result of \cite[Proposition 2.10]{ACGSI} with its extension to algebraic spaces.

\end{proof}

\begin{theorem}\label{thm: commutative diagram for mapping spaces}
    The diagram in Equation~\eqref{eqn: our artin-fan-diagram} commutes.
\end{theorem}

\begin{proof}
    First note that we can check the statement at strict closed points. Indeed, for any strict map $S \to \WG(X)$, the induced map $S \to  \cW_{\Gamma}(X) \to \Ev \to \mathcal{A}_{\Ev}$ yields an element
\begin{equation}\label{eqn: homo-to-Artin-fans}
\mathrm{Hom}_{\mathsf{fs}}(S,\mathcal{A}_{\Ev})
=
\mathrm{Hom}_{\mathsf{GenCones}}(\Sigma(S), \Sigma(\Ev)),
\end{equation}
If we know the statement for points, then $\Sigma(S) \to \Sigma(\Ev)$ restricts on every cone of $\Sigma(S)$ to the evaluation map times the $j$-invariant map and we are done.
    
    Now, let $\mathrm{Spec}(Q_\W \to\C) \to \WG(X)$ be a strict closed point. Denote by $\sigma$ the cone of $\Sigma(\WG(X))$ dual to $Q_\W$. We have a commutative diagram 

    \begin{equation*}\label{eqn:full-diagram}
    \begin{tikzcd}
    \mathrm{Spec}(Q_\W \to \C)
    \arrow[r]
    \arrow[drr]
    & \WG(X)
    \arrow[r]
    & \mathcal{A}_{\WG(X)}
    \arrow[r]
    & \mathcal{A}_X^n \times \mathcal{A}_{\oM_{1,1}}
    \\
    & 
    & \mathcal{A}_{\sigma}
    \arrow[u, hook]
    \arrow[ur]
    \end{tikzcd}
    \end{equation*}

    where the map $\mathrm{Spec}(Q_\W \to \C) \to  \mathcal{A}_{\sigma}$ is induced by the identity on $\sigma$ and the map $\mathcal{A}_{\sigma} \to \mathcal{A}_X^n \times \mathcal{A}_{\oM_{1,1}}$ by definition by the tropical evaluation and $j$ invariant map.

    On the other hand, the composition 
    $$
    \mathrm{Spec}(Q_\W \to \C) \to \WG(X) \to X^n \times \oM_{1,1} \to \mathcal{A}_{X}^n \times \mathcal{A}_{\oM_{1,1}}
    $$
    corresponds to the map $\sigma \to \Sigma(X)^n \times \Sigma(\oM_{1,1})$ defined as follows. First, the map $\mathrm{Spec}(Q_\W \to \C) \to \WG(X)$ yields a minimal logarithmic map
    \[
    \begin{tikzcd}
    & C \arrow[r, "f"] \arrow[d, "\pi"'] & X \\
    &(\mathrm{Spec}(\C),Q_\W) \arrow[u, bend left=60, "p_i"] \arrow[u, bend right=60, "z_j"'] 
    \end{tikzcd}
    \]
    and thus we obtain a morphism
    \[
    p_i^{-1} \underline{f}^{-1} \overline{\mathcal{M}}_{X,\underline{f}(p_i)} \longrightarrow p_i^{-1} \overline{\mathcal{M}}_{C,p_i} = Q_\W \oplus \mathbb{N},
    \]
    where the map to $\mathbb{N}$ is the zero map. In particular, this gives a morphism
    \[
    p_i^{-1} \underline{f}^{-1} \overline{\mathcal{M}}_{X,\underline{f}(p_i)} \longrightarrow Q_\W,
    \]
    which in terms of the identification \eqref{eqn: basic monoids} is just the inclusion in the factor $\oM_{X,\underline{f}(\eta)}$ with $p_i \in \overline{\{ \eta\}}$. Its dual $\sigma \to p_i^{-1} f^{-1} \overline{\mathcal{M}}_{X,f(p_i)}^\vee \subseteq \Sigma(X)$ is precisely the evaluation at the vertex $V_\eta$ where $p_i$ is attached by \eqref{eqn: dual-of-Q}. For the map to $\Sigma(\cM_{1,1}) = \mathbb{R}_{\ge 0}$, the same argument applies. The map
\[
\overline{\mathcal{M}}_{\oM_{1,1}, \mathsf{st}([f])} \longrightarrow Q_\W
\]
is either the zero map when $\mathsf{st}([f])$ is not nodal, or the map $\mathbb{N} \to Q_\W$ which, in terms of~\eqref{eqn: basic monoids}, is the diagonal inclusion in the factors $\mathbb{N}$ corresponding to edges in the cycle of the dual graph of $C$ and $0$ on all other factors. Its dual
\[
\sigma \longrightarrow \N
\]
is precisely given by taking the sum of the lengths of the edges forming the cycle of the dual graph of $C$.

The conclusion of the theorem now follows from~\eqref{eqn: homo-to-Artin-fans}.
\end{proof}

\begin{remark}
    The argument above is general and applies to all the logarithmic mapping spaces.
\end{remark}

\subsection{Modification of the evaluation map}

Next, we replace the evaluation space $\Ev$ with a smooth toric variety as follows. 
The standard theory of modular forms implies that the moduli stack $\oM_{1,1}$ is isomorphic to the stacky projective line 
$\P(4,6)$. Thus, we have a morphism
\[
\oM_{1,1} \simeq \P(4,6) \to \P(2,3) \to \P^1,
\]
where the first map is the $\mu_2$-rigidification and the second map is the coarse moduli space morphism. Moreover, we may choose these maps so that the scheme-theoretic preimage of $0$ is precisely the $B\mu_2$ corresponding to the unique nodal curve  in $\oM_{1,1}$. Consider the morphism
\[
\Ev = X^n \times \oM_{1,1} \longrightarrow X^n \times \P^1 .
\]

Since both the point class and the unit class on $\oM_{1,1}$ are pulled back from $\P^1$, we may replace the space $\Ev$ by $X^n \times \P^1$. Moreover, for any class  $\alpha \otimes \beta \in \mathsf{CH}^*(X^n) \otimes \mathsf{CH}^*(\oM_{1,1})$ that is the restriction of some class $\alpha \otimes \beta' \in \mathsf{CH}^*(X^n) \otimes \mathsf{CH}^*(\P^1)$, the restriction of the Minkowski weight $c_{\alpha \otimes \beta'}$ to $\Sigma(\Ev) \subseteq \Sigma(X^n \times \P^1)$ coincides with $c_{\alpha \otimes \beta}$  as defined in~\eqref{eqn: def of Minkowski weight c}. 

By abuse of notation, we will again denote the space $X^n \times \P^1$ by $\Ev$, the class 
$\alpha \otimes \beta'$ by $\alpha \otimes \beta$, and the Minkowski weight 
$c_{\alpha \otimes \beta'}$ by $c_{\alpha \otimes \beta}$.

\begin{lemma}
    There exists a commutative diagram
\begin{center}
\begin{tikzcd}
\widetilde{\mathcal{W}}_{\Gamma}(X) 
  \arrow[r, "\widetilde{\mathsf{ev} \times \mathsf{st}}"] 
  \arrow[d, "s"']               
&  \widetilde{\Ev} 
  \arrow[d, "t"]     \\
\mathcal{W}_{\Gamma}(X) 
  \arrow[r, "\mathsf{ev} \times \mathsf{st}"] 
& \Ev             
\end{tikzcd}    
\end{center}
such that:
\begin{itemize}
\item $s$ and $t$ are logarithmic blow-ups;
\item $\widetilde{\Ev}$ is a smooth toric variety;
\item the morphism $\widetilde{\tau}$ satisfies the assumptions of Proposition~\ref{prop: technical tool}.
\end{itemize}

\end{lemma}
\begin{proof}
    The images $\Sigma(\mathsf{ev} \times \mathsf{st})(\sigma)$ of the cones 
$\sigma$ of $\Sigma(\mathcal{W}_{\Gamma}(X))$ induce a subdivision of 
$\Sigma(\Ev)$. Up to a further subdivision of $\Sigma(\Ev)$, we may assume that the induced logarithmic blow-up $\widetilde{\Ev}$ is a smooth toric variety. 

The subdivision $\Sigma(\widetilde{\Ev})$ of $\Sigma(\Ev)$ induces, by pullback, a subdivision of $\Sigma(\mathcal{W}_{\Gamma}(X))$. Let $\widetilde{\mathcal{A}}_{\cW_\Gamma(X)} \to \mathcal{A}_{\cW(X)}$ be the corresponding morphism of Artin fans, and let $\widetilde{\mathcal{W}}_{\Gamma}(X)= \cW_\Gamma(X) \times_{\mathcal{A}_{\cW(X)}} \widetilde{\mathcal{A}}_{\cW_\Gamma(X)}$ be the resulting logarithmic blow-up.

There is a factorization
\[
\widetilde{\mathcal{W}}_{\Gamma}(X) \longrightarrow \widetilde{\mathcal{A}}_{\cW_\Gamma(X)} \longrightarrow \mathcal{A}_{\widetilde{\mathcal{W}}_{\Gamma}(X)} \to \mathsf{Log}
\]
into strict morphisms, where $\mathcal{A}_{\widetilde{\mathcal{W}}_{\Gamma}(X)} \to \mathsf{Log}$ is étale and representable. The morphism $\widetilde{\mathcal{A}}_{\mathcal{W}_{\Gamma}(X)} \to \mathcal{A}_{\widetilde{\mathcal{W}}_{\Gamma}(X)}$ is not necessarily representable, but the strata of $\widetilde{\mathcal{A}}_{\mathcal{W}_{\Gamma}(X)}$ are trivial gerbes over the strata of $\mathcal{A}_{\widetilde{\mathcal{W}}_{\Gamma}(X)}$ (see \cite[Example 3.3.1]{ACMW} and \cite[Section 4.5]{MPS}).

The space $\widetilde{\mathcal{W}}_{\Gamma}(X)$ comes equipped with a logarithmic map
\[
\widetilde{\mathsf{ev} \times \mathsf{st}} \colon \widetilde{\mathcal{W}}_{\Gamma}(X) \longrightarrow \widetilde{\Ev},
\]
which is induced by the universal property of fs Cartesian products. Since the cone complex of $\widetilde{\Ev}$ is monodromy-free, the map $\widetilde{\mathcal{A}}_{\mathcal{W}_{\Gamma}(X)} \to \mathcal{A}_{\widetilde{\Ev}}$ factors through $\mathcal{A}_{\widetilde{\mathcal{W}}_{\Gamma}(X)}$, thereby giving the required commutative diagram.
\end{proof}

\begin{remark}
    Note that we are not claiming the existence of a commutative diagram 
    \begin{center}
\begin{tikzcd}
\widetilde{\mathcal{W}}_{\Gamma}(X) 
  \arrow[r] 
  \arrow[d, "s"']               
&  \mathcal{A}_{\widetilde{\mathcal{W}}_{\Gamma}(X)} 
  \arrow[d, dotted]     \\
\mathcal{W}_{\Gamma}(X) 
  \arrow[r] 
& \mathcal{A}_{\mathcal{W}_{\Gamma}(X)}           
    \end{tikzcd}    
    \end{center}
    Indeed, the cone complex $\Sigma(\mathcal{W}_{\Gamma}(X))$ may not be monodromy-free (see \cite[Example 3.3.1]{ACMW}).  

    On the other hand, the formation of the cone complex is functorial, yielding a natural morphism of cone complexes $\Sigma(\widetilde{\mathcal{W}}_{\Gamma}(X)) \to \Sigma(\mathcal{W}_\Gamma(X))$ (see \cite[Remark 2.5]{ACGSI}). Indeed, by construction, the cone complex $\Sigma(\widetilde{\mathcal{W}}_{\Gamma}(X))$ is obtained from the subdivision of $\Sigma(\mathcal{W}_{\Gamma}(X))$ by discarding certain cone automorphisms.
\end{remark}

\subsection{Proof of  Theorem \ref{thm: tropical correspondence}}

We can now proceed with the proof of the correspondence theorem for well-spaced genus $1$ maps.

\begin{proof}[Proof of  Theorem \ref{thm: tropical correspondence}]

    Since logarithmic blow-ups are birational, it is enough to show that 
\[
\big(\widetilde{\mathsf{ev} \times \mathsf{st}}\big)^* t^*(\alpha \otimes \beta) \cap 
[\widetilde{\mathcal{W}}_{\Gamma}(X)]
\]
equals the right-hand side of \eqref{eqn: correspondence}.

We begin by computing $t^*(\alpha \otimes \beta)$. The advantage of having $\Ev$ 
and $\widetilde{\Ev}$ be toric varieties is that we can apply the theory of 
Minkowski weights (see ~\cite{FuSt}) to these spaces.  Then, with the notation of \cite[Corollary 3.6]{FuSt}, we have
$$
t^*(\alpha \otimes \beta) \cap [\widetilde{\Ev}]= \sum_{\widetilde{\tau}, \tau} m_{\widetilde{\tau},\tau}^0 c_{\alpha \otimes \beta}(\tau)[V(\widetilde{\tau})]
$$
where the sum is over all pairs $\widetilde{\tau} \in \Sigma(\widetilde{\Ev})$ of codimension $\dim \Ev - \dim \WG(X)$ and $ \tau \in \Sigma(\Ev) $ of codimension $\dim  \WG(X)$, and
\[
m_{\widetilde{\tau},\tau}^0=
\begin{cases}
\left[  N_{\Ev}: \Sigma(t)( N_{\widetilde{\tau}})+ N_\tau \right] 
& \text{if } \Sigma(t)(\widetilde{\tau}) \text{ meets } \tau +W, \\[6pt]
0 & \text{otherwise}.
\end{cases}
\]
Here $W \in (\R^r)^n \times \R_{\geq 0}$ is a generic displacement vector as in the statement of Theorem \ref{thm: tropical correspondence}. The fact that $\Sigma(t)(\widetilde{\tau})$ meets $\tau+W$ for the generic $W$ forces the above index to be finite, and in particular
\begin{equation}\label{eqn: injective map of lattices}
N_{\widetilde{\tau}} \to N_{\Ev} \to N_{\Ev}/N_\tau
\end{equation}
is injective. 

Next, by Proposition \ref{prop: technical tool}, 
\begin{equation}\label{eqn: correspondence-after-subdivision}
\widetilde{\mathsf{ev} \times \mathsf{st}}^*[V(\widetilde{\tau})]= \sum_{\widetilde{\sigma}} [  N_{\widetilde{\tau}}: \Sigma(\widetilde{\mathsf{ev} \times \mathsf{st}})( N_{\widetilde{\sigma}})] [V(\widetilde\sigma)]
\end{equation}
where the sum is over the cones $\widetilde{\sigma}$ of $\Sigma( \widetilde{\mathcal{W}}_{\Gamma}(X))$ such that $\Sigma(\widetilde{\mathsf{ev} \times \mathsf{st}})(\widetilde{\sigma})= \widetilde{\tau} $.

Observe that by the injectivity of~\eqref{eqn: injective map of lattices}, we have the following relation between lattice indices:
\[
[  N_{\widetilde{\tau}}: \Sigma(\widetilde{\mathsf{ev} \times \mathsf{st}})( N_{\widetilde{\sigma}})]\,
[  N_{\mathrm{Ev}} :\Sigma(t)( N_{\widetilde{\tau}})+ N_\tau ]
=
[ N_{\mathrm{Ev}}/N_\tau: \Sigma(\widetilde{\mathsf{ev} \times \mathsf{st}})( N_{\widetilde{\sigma}})].
\]
Furthermore, for dimensional reasons, if $\sigma$ is a cone in $\Sigma(\mathcal{W}_\Gamma(X))$ such that $\Sigma(\mathsf{ev} \times \mathsf{st})(\sigma)$ intersects $\tau + W$, there exists a unique cone $\widetilde{\sigma}$ in the subdivision $\Sigma(\widetilde{\mathcal{W}}_{\Gamma}(X))$ contained in $\sigma$ that satisfies
\[
\widetilde{\sigma} \cap (\tau + W) \neq \emptyset.
\]
We claim that:
\begin{enumerate}[label=\textbf{Claim (\roman*)}, leftmargin=*]
    \item For every cone $\widetilde{\sigma}$ appearing in the sum \eqref{eqn: correspondence-after-subdivision}, the natural map $\widetilde{\sigma} \to \Sigma(\widetilde{\mathcal{W}}_{\Gamma}(X))$ is an inclusion;
    \item If a cone $\widetilde{\sigma} \in \Sigma(\widetilde{\mathcal{W}}_\Gamma(X))$ maps to a cone $\sigma \in \Sigma(\mathcal{W}_\Gamma(X))$, then the pushforward of $[V(\widetilde{\sigma})]$ is $[V(\sigma)]$; that is, the stabilizer of the point $f_{\widetilde{\sigma}}$ in $ \widetilde{\mathcal{W}}_\Gamma(X)$ corresponding to $\widetilde{\sigma}$ is equal to the stabilizer of the point $f_\sigma$ of in  $\mathcal{W}_{\Gamma}(X)$ corresponding to $\sigma$ \footnote{Note that this may fail if the map $\sigma \to \Sigma(X)$ is not an inclusion. For example, consider $X = [\mathbb{A}^2/\mu_2]$, where $\mu_2$ swaps the coordinates. The subdivision of $\Sigma(X)= \mathbb{R}_{\geq 0}^2/\mu_2$ obtained by adding the ray $\mathbb{R}_{\ge 0}(1,1)$ corresponds to the stack $X'=[\operatorname{Bl}_0(\mathbb{A}^2)/\mu_2]$, with exceptional divisor $[\mathbb{P}^1/\mu_2]$ on which $\mu_2$ acts by $[a:b] \mapsto [b:a]$. The fixed points of this action on $\mathbb{P}^1$ are $[1:1]$ and $[1:-1]$, whereas the points maximally degenerate points $[0:1]$ and $[1:0]$ are swapped by the action. Consequently, these swapped points yield a stratum in $X'$ with trivial stabilizer, whereas the origin $0 \in \mathbb{A}^2$ has stabilizer $\mu_2$.}.
\end{enumerate}

Indeed, Claim~(i) follows from the observation that $\widetilde{\sigma}$ surjects onto a maximal cone of $\widetilde{\Ev}$, which corresponds to a smooth toric variety. 

Claim~(ii), on the other hand, requires some explanation. Let $G_\sigma$ and $G_{\widetilde{\sigma}}$ be the stabilizers of $f_\sigma$ and $f_{\widetilde{\sigma}}$ respectively. Furthermore, set $T_{\sigma} = \mathrm{Hom}(\overline{\mathcal{M}}_{\mathcal{W}_{\Gamma}(X), f_{\sigma}}^{\mathrm{gp}}, \mathbb{G}_m)$ and define $T_{\widetilde{\sigma}}$ analogously. We then have a Cartesian diagram:

\[
\begin{tikzcd}[row sep={40,between origins}, column sep={70,between origins}]
      &\widetilde{\cW}_{\Gamma}(X)  \ar{rr}\ar{dd} & & \widetilde{\mathcal{A}}_{\cW_{\Gamma}(X)}\vphantom{\times_{S_1}} \ar{dd} \\
    BG_{\widetilde{\sigma}} \ar[crossing over]{rr} \ar{dd} \ar{ur} & & BT_{\widetilde{\sigma}} \ar{ur} \\
      & \cW_\Gamma(X) \ar{rr} & & \mathcal{A}_{\cW_\Gamma(X)} \vphantom{\times_{S_1}} \\
    BG_\sigma \ar{rr} \ar{ur} &&  BT_\sigma \ar[from=uu,crossing over] \ar{ur}
\end{tikzcd}
\]

By Claim~(i), the horizontal maps on the left-hand face are faithful inclusions, and the morphism $BT_{\widetilde{\sigma}} \to BT_\sigma$ is an isomorphism. Consequently, $BG_{\widetilde{\sigma}} \to BG_\sigma$ is also an isomorphism, which completes the proof of~(ii).

Finally, since $[N_{\widetilde{\sigma}} : N_\sigma] = 1$, we can rewrite Equation \eqref{eqn: correspondence-after-subdivision} as
\[
(\mathsf{ev} \times \mathsf{st})^*( \alpha \otimes \beta) = \sum_{\sigma \in \Sigma(\mathcal{W}_\Gamma(X))} \sum_{\tau} c_{\alpha \otimes \beta}(\tau) \bigl[  N_{\mathrm{Ev}}: \Sigma( \mathsf{ev} \times \mathsf{st})(N_\sigma) + N_\tau  \bigr] [V(\sigma)].
\]
The outer sum can be replaced by a sum over maximal cones of $W_\Gamma(\Sigma(X))$, provided each cone is counted with the weight $w(\sigma)$ as defined in Definition~\ref{def: weights}. Note that for a cone $\sigma \in \Sigma(\mathcal{W}_\Gamma(X))$, the cycle $[V(\sigma)]$ corresponds to a stacky point with automorphism group $\mathrm{Aut}(f_{\sigma})$. This concludes the proof.
\end{proof}

\section{Proof of the correspondence theorem, Part II}\label{sec: weights}

Let $\sigma$ be a maximal cone in $W_{\Gamma}(\Sigma)$. We aim to count the number of maximal cones of $\Sigma(\WG(X))$ lying above $\sigma$. Each such cone corresponds to a map in $\WG(X)$ lying at the intersection of a maximal collection of logarithmic divisors in $\WG(X)$.

Let
\begin{equation}\label{eqn: def n(sigma)}
\sigma_{1,1}, \ldots, \sigma_{1,t_1}, \;\ldots,\; \sigma_{n(\sigma),1}, \ldots, \sigma_{n(\sigma),t_{n(\sigma)}}
\end{equation}
be the maximal cones of $\Sigma(\WG(X))$ lying over $\sigma$, ordered so that for each $i$, the cones $\{\sigma_{i,j}\}_{j=1}^{t_i}$ correspond to points $(C_{i,j}/S_{i,j},\mathsf{p}^{i,j}, f_{i,j}, \rho_{i,j})$ whose underlying stable map of schemes
$
(\underline{C}_i/\underline{S}_i, \underline{\mathsf{p}}^{\,i}, \underline{f}_i)
$
and radial alignment $(\underline{C}_i/\underline{S}_i, \rho)$ is the same (with $\underline{S}_i = \mathrm{Spec}(\mathbb{C})$ for all $i$). 
The computation of the total number of cones $\sigma_{i,j}$ will be carried out in two steps:
\begin{itemize}
    \item First, we compute the number $n(\sigma)$, which we refer to as \emph{the number of fine maps determined by $\sigma$} \footnote{The quantity $n(\sigma)$ was introduced in the Introduction as the number of non-isomorphic fine basic logarithmic maps. This coincides with the one in Equation \eqref{eqn: def n(sigma)} by Lemma \ref{lemma: uniqueness of fine maps}.}.
    \item Then, for each $i = 1, \ldots, k$, we compute the number $t_i$, which we refer to as \emph{the number of logarithmic enhancements of $\underline{f}_i$}.
\end{itemize}

\begin{observation}\label{obs: further properties from sigma}
    Other than (1) and (2) in Observation \ref{disc:properties-over-sigma}, the data of $\sigma$ also yields the following two objects: 
    \begin{enumerate}[start=3, label=(\arabic*)]
        \item Two fine monoids $Q^{\fine}(\sigma)$ and $Q_{\mathsf{W}}^{\fine}(\sigma)$ defined as follows. The monoid  $Q^{\fine}(\sigma)$ is given by
        \begin{equation}\label{eqn: monoid Qfinesigma}
        Q^{\fine}(\sigma) := \frac{\prod_V \mathrm{Hom}(\vartheta_V \cap N_{\vartheta_V}, \mathbb{N}) \times \prod_e \mathbb{N}}{R}
        \end{equation}
        where $R \subseteq \left(\prod_V \mathrm{Hom}(\vartheta_V \cap N_{\vartheta_V}, \mathbb{N}) \times \prod_e \mathbb{N}\right)^{\gr}$ is the subgroup generated by the relations
        \begin{equation}\label{eqn: relations from tropical curve}
        a_{\vec{e}}(m) = \left((\ldots, \iota_{\vartheta_V}(m), \ldots, -\iota_{\vartheta_{V'}}(m), \ldots), (\ldots, u_{\vec{e}}(m), \ldots)\right),
        \end{equation}
        for $m \in M_X$ and $\vec{e}: V \to V'$ an oriented edge in $\mathsf{C}$. Here, $\vartheta_V$ denotes the cone of $\Sigma(X)$ containing the image of the vertex $V$ for tropical curves of type $\sigma$, and
     \[
        \iota_{\vartheta_V} : M_X \longrightarrow M_{\vartheta_V} := N_{\vartheta_V}^{*}
    \]
    is the dual of the inclusion $N_{\vartheta_V} \hookrightarrow N_X$. Moreover,
    \[
        u_{\vec e} : M_X \longrightarrow \mathbb Z
        \]
    is the homomorphism given by $m \mapsto m(u_{\vec e})$. For the definition of $Q_\W^{\fine}(\sigma)$, let $\tau $ be a cone in $\Sigma(\mathfrak{M}_{1,n})$ be the cone to which $\sigma$ map under the natural map $W_{\Gamma}(\Sigma(X)) \to \Sigma(\mathfrak{M}_{1,n})$. The cone $\tau$ naturally lives in $\prod_e \N$ and we set  
        \begin{equation}\label{eqn: monoid QWfinesigma}
        Q_\W^{\fine}(\sigma):= Q^{\fine}(\sigma) \times_{\prod_e \N} \tau^\vee
        \end{equation}
        See also Equation \eqref{eqn: basic monoid Wfine}.
        \item Endow each $\underline{S_{i,j}} = \operatorname{Spec}(\mathbb{C})$ with the fine log structure given by $Q_{\mathsf{W}}^{\mathsf{fine}}(\sigma)$. Then, we obtain a log-smooth curve $C \to S_{i,j}$ by endowing $\underline{C_{i,j}}$ with the sheaf of fine monoids $\mathcal{M}_C^{\fine}$ on $C$, as determined by Definition~\ref{def: stable log map}(1) and Remark~\ref{rmk: local structure at node}.
    \end{enumerate}
\end{observation}

\subsection{The number of logarithmic enhancements of $\underline{f}_i$ is the saturation index}

Fix $(\underline{C}_i/\underline{S}_i, \underline{\mathsf{p}}^{\,i}, \underline{f}_i)$ for some index $i$. Endow $\underline{S_{i}}$ and $\underline{C_i}$ with the log structure described Observation \ref{obs: further properties from sigma}$(4)$.

By \cite[Theorem~1.1]{Wise}, there is a factorization 
\begin{equation}\label{eqn: Wise factorization}
\underline{\mathrm{Hom}}_{\mathsf{LS}^{\mathsf{fine}}}(C_i, X)
\longrightarrow \mathcal{T}_i
\longrightarrow \mathrm{Hom}_{\mathsf{Sch}}(\underline{C}_i, \underline{X}),
\end{equation}
where the first map is a monomorphism and the second is étale. The underline on the first $\mathrm{Hom}$ indicates that we restrict to basic logarithmic maps. The space $\mathcal{T}_i$ is the space of types and it is defined as follows. For a scheme $\underline{S}'$, write $\underline{C}'_i = \underline{C}_i \times \underline{S}'$ and let $g \colon \underline{C}'_i \to \underline{C}_i$ be the projection. Then
\[
\mathcal{T}_i(\underline{S}') 
= \mathrm{Hom}\bigl(g^{-1} f_i^{-1} \oM_X^{\mathrm{\gr}},\; g^{-1} \oM_{C_i}^{\mathrm{\gr}}\bigr).
\]

As the proof of next lemma shows, the data of $\sigma$ and $\underline{f}_i$ determine the type. This generalizes \cite[Proposition~3.6.3]{R17} to genus~$1$ and to the fine basic monoid (rather than the unsaturated $Q^{\mathrm{un}}$).

\begin{lemma}\label{lemma: uniqueness of fine maps} 
Fix $i = 1, \ldots, n(\sigma)$, and, for $j=1,\ldots, t_i$, denote by $(C_{i,j}^{\mathsf{fine}}/S_{i,j}^{\mathsf{fine}}, \mathsf{p}^{i,j}, f_{i,j}^{\mathsf{fine}})$ the image of $(C_{i,j}/S_{i,j}, \mathsf{p}^{i,j},\\ f_{i,j}, \rho_{i,j})$ in $\mathsf{ACGS}_{\Gamma}(X)^{\mathsf{fine}}$. View this as an element of $\underline{\mathrm{Hom}}_{\mathsf{LS}^{\mathsf{fine}}}(C_i, X)$. Then its image in $\mathcal{T}_i$ is independent of $j \in \{1, \ldots, t_i\}$.
\end{lemma}

\begin{proof}
     First observe that the basic monoid $Q_{i,j}^{\mathsf{fine}}$ of $(C_{i,j}^{\mathsf{fine}}/S_{i,j}^{\mathsf{fine}} , \mathsf{p}^{i,j}, f_{i,j}^{\mathsf{fine}})$ defined in Equation \eqref{eqn: variants of Q} coincides with $Q_\W^\fine(\sigma)$ in Equation \eqref{eqn: monoid QWfinesigma}. In particular, it is independent of $i$ and $j$.

    Next, the logarithmic structure $\mathcal{M}_{C_i}^{\mathsf{fine}}$ on $C$ coincides with the one obtained from the natural logarithmic map
    $
    S_i \longrightarrow \mathfrak{M}_{1,n}.
    $ 
    Finally, the morphism
    $
    \underline{f_{i,j}}^{-1}\mathcal{M}_X= \underline{f_{i}}^{-1}\mathcal{M}_X \longrightarrow \mathcal{M}_{C_{i,j}}
    $
    factors through $\mathcal{M}_{C_i}^{\mathsf{fine}}$ by the definition of $Q_{i,j}^{\mathsf{fine}}$. This concludes.
    
\end{proof}

Since the first map in Equation \eqref{eqn: Wise factorization} is a monomorphism, we obtain that the image $(C_{i,j}^{\mathsf{fine}}/S_{i,j}^{\mathsf{fine}}, \mathsf{p}^{i,j}, f_{i,j}^{\mathsf{fine}}, \rho_{i,j}) \in \WG(X)^{\fine}$ of the $(C_{i,j}/S_{i,j}, \mathsf{p}^{i,j}, f_{i,j}, \rho_{i,j})$ is independent of $j$. Call it 
$$
W_i=(C_{i}^{\mathsf{fine}}/S_{i}^{\mathsf{fine}}, \mathsf{p}^{i}, f_{i}^{\mathsf{fine}}, \rho_{i})= \mathrm{Spec}(Q_{\W}^{\fine}(\sigma) \to \C)=S_i.
$$
Note that the monoid of $S_i$ is independent of $i$ and coincides with $Q_{\mathsf{W}}^{\mathsf{fine}}(\sigma)$, as noted in the proof of Lemma~\ref{lemma: uniqueness of fine maps}. By Lemma \ref{lemma: uniqueness of fine maps}, the number of logarithmic enhancements of $\underline{f_i}$ is equal to the cardinality of the fiber of the  map in \eqref{eqn: map W} over $W_i$. Call 
\[
\begin{aligned}
W_i^{\sat}
&= W_i 
  \times_{\WG(X)^{\fine}} 
   \WG(X)
\end{aligned}
\]
the saturation of $W_i$. Here the fiber product is taken in $\mathsf{Sch}$. 

\begin{lemma}[Generalization of {\cite[Lemma 3.1]{Gross23}}]\label{lemma: generalization of Mark}
Let \( W := \operatorname{Spec}(Q \to \kappa) \) be a log point with \( Q \) a fine monoid. Then, the number of connected components of \( W^{\mathrm{sat}} \) is \( \lvert (Q^{\mathrm{gp}})_{\mathrm{tors}} \rvert \).
    
\end{lemma}

\begin{proof}
When \( Q \) is sharp and fine, this is \cite[Lemma~3.1]{Gross23}. 
To extend this to a fine (not necessarily sharp) monoid, replace \( \mathsf{m} \) in the proof there with 
$
\{ z^q \mid q \in Q \smallsetminus Q^{\times} \}.
$

Otherwise, assuming \cite[Lemma~3.1]{Gross23} we can give a direct proof as follows. Denote by $\overline{Q}= Q/ Q^{\times}$. We have a cartesian diagram of schemes

\[\begin{tikzcd}[row sep={40,between origins}, column sep={70,between origins}]
      &\overline{W}^{\sat}  \ar{rr}\ar{dd}\ar{dl} & & W^{\sat}\vphantom{\times_{S_1}} \ar{dd}\ar{dl} \\
    \overline{W}:=\mathrm{Spec}(\overline{Q} \to \C) \ar[crossing over]{rr} \ar{dd} & & W=\mathrm{Spec}(Q \to \C) \\
      & \mathrm{Spec}(\C[\overline{Q}^{\sat}]) \ar{rr} \ar{dl} & & \mathrm{Spec}(\C[Q^{\sat}]) \vphantom{\times_{S_1}} \ar{dl} \\
    \mathrm{Spec}(\C[\overline{Q}]) \ar{rr} &&  \mathrm{Spec}(\C[Q])\ar[from=uu,crossing over]
\end{tikzcd}\]
where $\overline{W}^{\sat}$ denotes the saturation of $\overline{W}$. The bottom square is cartesian by \cite[Proposition 1.3.5 (5)]{Ogus} and the square on top is cartesian by \cite[Lemma~3.1]{Gross23}. The conclusion now follows from the fact that $\overline{W} \to W$ is an isomorphism of schemes.
\end{proof}

\begin{theorem}\label{thm: number of logarithmic enhancements}
    The number of logarithmic enhancements of each $\underline{f_{i}}$ is equal to $|((Q_\W^{\mathsf{fine}}(\sigma))^{\gr})_{\tors}|$.
\end{theorem}

\begin{proof}
    Apply Lemma \ref{lemma: generalization of Mark} to $W_i$.
\end{proof}

\begin{definition} \label{def: sat}
    We call $\sat(\sigma):= |(Q^{\fine}_W(\sigma)^{\gr})_{\tors}|$
    the saturation index of the cone $\sigma$. 
\end{definition}

\subsection{Automorphisms}\label{sec: automorphisms}

In this section, we relate the automorphism groups of basic fs logarithmic maps to those of the associated basic fine logarithmic maps and their corresponding tropical maps.

Let $\sigma$ be a maximal cone in $\Sigma(\mathcal{W}_\Gamma(X))$, $[h:(\mathsf{C},\mathsf{p}) \to \Sigma(X)]$ a point in $\sigma$ and let $(C,\mathsf{p}, f_\sigma,  \rho)$ (resp. $(C^\fine,\mathsf{p}, f_\sigma^\fine ,  \rho)$) be the corresponding point in $\mathcal{W}_\Gamma(X)$ (resp. the image of $(C,\mathsf{p}, f_\sigma,  \rho)$ in $\cW_\Gamma(X)^\fine$).

\begin{lemma}\label{lemma: automorphisms fix vertices}
Every automorphism of the induced map $[h:(\mathsf{C},\mathsf{p}) \to (N_X)_\R]$ to $(N_X)_\R$ fixes all the vertices of $\tropC$.
\end{lemma}

\begin{proof}
Let $A$ be such an automorphism. Define the (in general disconnected) subgraph
\[
\mathsf{T} := \overline{\tropC \setminus \tropC_0} \subseteq \tropC,
\]
which is a union of trees, each carrying at least one of the markings $z_j$.

Since the markings are fixed by $A$, it follows that $A$ preserves each connected component of $\mathsf{T}$. In particular, $A$ fixes the intersection points $\mathsf{T} \cap \tropC_0$. Hence all vertices of $\tropC_0$ that lie in the closure of $\mathsf{T}$ are fixed by $A$.

Now every vertex in $\tropC_0$ is at least $3$-valent, and at most two of its adjacent edges lie in the cycle $\tropC_0$. Therefore, each vertex of $\tropC_0$ is incident to at least one edge contained in $\mathsf{T}$, and hence lies in $\mathsf{T} \cap \tropC_0$. It follows that all vertices of $\tropC_0$ are fixed by $A$.

Now let $V$ be a vertex in a connected component of $\mathsf{T}$. There exists a unique path
\[
V = V_k, V_{k-1}, \dots, V_0
\]
from $V$ to the nearest vertex $V_0$ in the cycle $\tropC_0$.

Suppose that $A(V) = V'$. Since $A$ preserves $\tropC_0$, it also preserves the distance to $\tropC_0$, hence $V'$ lies at the same distance $k$ from the cycle. If $V \neq V'$, then there exists a marking $z_j$ such that the unique path from $z_j$ to $\tropC_0$ passes through $V$ but not through $V'$. This yields a contradiction because $A$ fixes the marking $z_j$ and thus the path from $z_j$ to $\tropC_0$.

\end{proof}

\begin{corollary} \label{cor: automorphisms}
The tropical map $[h:(\mathsf{C},\mathsf{p}) \to \Sigma(X)]$ has no nontrivial automorphisms unless the cycle $\tropC_0$ is either a self-loop  or consists of two vertices joined by two parallel edges with the same direction vector, possibly together with additional $2$-valent vertices along those edges (see Figure \ref{fig: two vertices in cycle}). In these cases, the automorphism group is isomorphic to $\mu_2$. 
\end{corollary}

 \begin{figure}[h]
    \centering
    \begin{tikzpicture}
        \draw[burntsiena] (0.6, -2) -- (0.6, 1); 
        \draw[burntsiena] (0.6, -2) -- (1.3, 1);
    
        \fill (2,0) circle[radius=2pt];
        \fill (1.5,0) circle[radius=2pt];
        \draw (0, 0.05) -- (1.5, 0.05);
        \draw (0, -0.05) -- (1.5, -0.05);
        \draw (1.5, 0) -- (2,0);
        \draw (0, 0) -- (-1, 1); 
        \draw (0, 0) -- (-1, -1); 
        \draw (2, 0) -- (3, 1); 
        \draw (2, 0) -- (3, -1);
        
        \node at (-0.4,0) {$V$};
        \node at (1.5,-0.3) {$V'$};

        \fill (0.6,0.05) circle[radius=1pt];
        \fill (0.6,-0.05) circle[radius=1pt];
        \fill (1.08,0.05) circle[radius=1pt];
        \fill (1.06,-0.05) circle[radius=1pt];

        \fill (0,0) circle[radius=2pt];
    \end{tikzpicture}
    \caption{A parametrized tropical curve where the cycle consists of two vertices (depicted by large bullets) connected by two edges with the same direction vector, with additional $2$-valent vertices along the edges (depicted by small bullets). The codimension $1$ cones of $\Sigma(X)$ which contain the constrained vertices are shown in orange.}
    \label{fig: two vertices in cycle}
\end{figure}
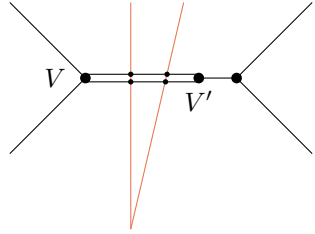

\begin{proof}
This follows essentially from Lemma~\ref{lemma: automorphisms fix vertices}. Any automorphism of the induced map to $(N_X)_\R$ fixes every vertex. Consequently, it can only permute half-edges incident to a common vertex. If an automorphism exchanges two half-edges, then the corresponding edges must connect the same pair of vertices, or else form a self-loop at a single vertex. In either case there is a unique nontrivial involution, and hence the automorphism group is isomorphic to $\mu_2$.

Since $\sigma$ is maximal, if the cycle $\tropC_0$ is a self-loop, then its image under $h$ lies in the interior of a maximal cone. It follows that the corresponding $\mu_2$-involution lifts to an involution of the map to $\Sigma(X)$.

Similarly, suppose the cycle consists of two vertices connected by two parallel edges, possibly together with additional $2$-valent vertices along those edges. Then the $\mu_2$-action again lifts to the map to $\Sigma(X)$. Indeed, whenever the cycle meets a codimension-$1$ stratum of $\Sigma(X)$, it does so simultaneously along both parallel edges, so exchanging the two branches preserves the map to $\Sigma(X)$.

\end{proof}

Let $G_\sigma \trianglelefteq \mathrm{Aut}(f_\sigma)$ (resp. $G^\fine_\sigma \trianglelefteq \mathrm{Aut}(f_\sigma^\fine)$) denote the normal subgroup consisting of those automorphisms of $f_\sigma$ (resp. $f_\sigma^\fine$) that preserve each component $C_V$ corresponding to a vertex $V$ lying on the cycle of the curve, as well as every node corresponding to an edge of the cycle $\tropC_0$.

\begin{lemma} \label{lem: automorphisms dV}
    There are isomorphisms
    \[
    G_\sigma^\fine \simeq G_\sigma \simeq \prod_{V \text{ constrained}} \mu_{d_V},
    \]
    where each factor $\mu_{d_V}$ acts on the component $C_V$ and commutes with the map $\underline{f}|_{C_{V}}$.
\end{lemma}

\begin{proof}
    First, recall that if the maps parametrized by $\sigma$ send a vertex $V$ to a codimension one cone $\tau$ of $\Sigma(X)$, then the restriction $\underline{f}|_{C_{V}}$ is a degree $d_V$ cover of the toric $\mathbb{P}^1$ corresponding to $\tau$. This cover is fully ramified over the toric points of $\mathbb{P}^1$.
    The group $\mu_{d_V}$ is precisely the group of automorphisms of $C_V$ commuting with $\underline{f}|_{C_{V}}$.
    Its action on $C_V$ clearly preserves the divisors $D_{m,V}$ for all $m \in M_X$. Moreover, because the line bundles $L_{m,V}$ in Equation \eqref{eqn: line bdl LmV} and the sections $\sigma_{m,V}$ are pulled back from $X$, they are also preserved by $\mu_{d_V}$. By Remark~\ref{rmk: stronger_version_realizability}, the action of $\mu_d$ on $\underline{f}$ lifts to an automorphism of the logarithmic map. Note that because $\sigma$ is a maximal cone, none of the $p$-markings can be on the constrained vertex $V$. 

    Let $A \in G_\sigma$ (or $G_\sigma^\fine$).
    By Corollary \ref{cor: automorphisms}, the automorphism $A$ preserves each irreducible component of $C$. Thus, after composing $A$ with an appropriate element of $\prod \mu_{d_V}$, we may assume that $A$ acts as the identity on each $C_V$ for all constrained vertices $V$.

    Define the subcurve
\[
T := \overline{C \setminus \left( \bigcup_{V \in V(\mathsf{C}_0)} C_V \right)} \subseteq C,
\]
which consists of a collection of trees of marked rational curves. We further mark the intersection points
\[
T \cap \left( \bigcup_{V \in V(\mathsf{C}_0)} C_V \right).
\]
The automorphism $A$ fixes each of these intersection points, as well as all marked points $p_i$ and $z_j$ on $C$. Since each connected component of $T$ is a stable genus $0$ curve (after adding the markings coming from the intersection points and an extra marking on each component $C_V$ corresponding to constrained vertices $V$), it admits no nontrivial automorphisms. Therefore, $A$ must be the identity on $T$. Since $A$ also fixes all the nodes contained in the cycle it is the identity on every component of the cycle which has at least $3$ special points. This concludes.
\end{proof}

We conclude this section stating the following fundamental result, which will be proven in \S\ref{sec: computations}.

\begin{theorem}\label{thm: automorphisms}
        We have an isomorphism of groups
        $$
        \mathrm{Aut}((C,\mathsf{p}, f_\sigma,  \rho)) \simeq \mathrm{Aut}(C^\fine,\mathsf{p}, f^\fine_\sigma,  \rho).
        $$
    \end{theorem}

Combining Theorem \ref{thm: automorphisms} and Lemma \ref{lem: automorphisms dV}, we also obtain an injective homomorphism of groups
\begin{equation}\label{eqn:aut_injection}
\frac{\mathrm{Aut}(C,\mathsf{p},f_\sigma,\rho)}{G_\sigma}
\simeq
\frac{\mathrm{Aut}(C^\fine,\mathsf{p},f^\fine_\sigma,\rho)}{G^\fine_\sigma}
\hookrightarrow
\mathrm{Aut}(\sigma).
\end{equation}
The group on the right-hand side is either trivial or isomorphic to $\mu_2$. In the latter case, the above inclusion need not be an isomorphism. 

\begin{example}
    Consider a tropical curve in $\Sigma(\mathbb{P}^1)$ where the cycle is a self-loop at a four-valent vertex $V$. This is deficiency 1, Case (1) in Section~\ref{sec: dimension 1}. Parametrize $C_V \simeq \mathbb{P}^1$ so that the special points are  are $0, \infty$ for the self-node, and $1, q$ for the other two special points. There are no non-trivial automorphisms of $C_V$ preserving the set $\{0,\infty \}$ while fixing $1$ and $q$, unless $q = -1$.  We will see in Theorem \ref{thm: calculation of n(sigma)} that one can have $q \neq -1$. 
\end{example}

\subsection{Computation of $n(\sigma), \nn(\sigma)$ and more automorphisms}

Let $\sigma$ be a cone of maximum dimension in $W_{\Gamma}(\Sigma(X))$. In this section, we compute the number $n(\sigma)$ of basic fine logarithmic maps determined by $\sigma$ (Theorem \ref{thm: calculation of n(sigma)}) along with its weighted analogue $\nn(\sigma)$ (Theorem \ref{thm: calculation of weights}). Recall from the introduction that  
$$
\nn(\sigma)=\sum_{i=1}^{n(\sigma)} \frac{1}{|\mathrm{Aut}(f_i^\fine)/ \prod \mu_V|}.
$$ 
where the automorphism group is taken viewing 
$
f_i^\fine= (C_{i}^{\mathsf{fine}}/S_{i}^{\mathsf{fine}}, \mathsf{p}^{i}, f_{i}^{\mathsf{fine}}, \rho_{i})
$
as an element of $\cW_\Gamma(X)^\fine$.

Finally, we demonstrate that the automorphism group of each fs map $(C_{i,j}/S_{i,j}, \mathsf{p}^{i,j}, f_{i,j}, \rho_{i,j})$ is isomorphic to that of its associated fine map $(C_{i}^{\mathsf{fine}}/S_{i}^{\mathsf{fine}}, \mathsf{p}^{i}, f_{i}^{\mathsf{fine}}, \rho_{i})$ (Theorem \ref{thm: automorphisms}). This establishes the equality in Equation~\eqref{eqn: reduction_to_aut_of_fine_map} stated in the introduction.

\subsubsection{Reducing to the neighborhood of the cycle} 

Let $\underline{f}: (\underline{C},\mathsf{p}) \rightarrow X$ by a $N$-marked genus $1$ stable map satisfying conditions $(1)$ and $(2)$ of Observation \ref{disc:properties-over-sigma}. Let $(C_0, \mathsf{p}')$ be the cycle inside $\underline{C}$. Here $\mathsf{p}'$ consists of all the markings $p_i$ and $z_j$ in $C_0$ together with the points $C_0 \cap \overline{C \smallsetminus C_0}$ ordered arbitrarily.

The restriction to the neighborhood of the cycle defines a tropical type $\sigma_0$ of a parametrized tropical curve in $(N_X)_\R$.

\begin{definition}\label{def: lifting property C0}
    We say that $(C_0, \mathsf{p}')$ satisfies the lifting property with respect to $\sigma$ if:
    \begin{enumerate}
        \item[$\bullet$] for every $m \in M_X$ and $V \in V(\tropC_0)$ there exists meromorphic functions $\nu_{m,V}$ satisfying the lifting condition~$(a)$ (see Definition \ref{def: lifting condition a});
        \item[$\bullet$] for every oriented edge $\vec{e}:V \to V'$ in $\tropC_0$ there exist coordinates $x_{q_V}$ and $x_{q_{V'}}$ at $q_V\in C_V$ and $q_{V'}\in C_{V'}$ such that the pair $\nu_{m,V}$ and $\nu_{m,V'}$ satisfies lifting condition (b) with respect to $x_{q_V}, x_{q_{V'}}$ for every $m \in M_X$ (see Definition \ref{def: lifting condition b}).
    \end{enumerate} 
\end{definition}

As in Remark \ref{rem: basis}, we can check if the lifting property for $(C_0, \mathsf{p}')$ on a basis of $M_X$. 

The cone $\sigma \subseteq W_\Gamma(\Sigma(X))$ maps to a unique smallest cone $\sigma'$ of $M_{\Gamma}(\Sigma(X))$ (see \S\ref{sec: well-spaced tropical maps}).

\begin{lemma} \label{lem: liftingtrees}
    The map $\underline{f}: (\underline{C},p) \rightarrow X$ admits a fs logarithmic lift with combinatorial type $\sigma'$ if and only if $(\underline{C}_0\, \mathsf{p}')$ satisfies the lifting property with respect to $\sigma$. 
\end{lemma}

Note that this is not claiming yet that such a logarithmic lift can be taken well-spaced.

\begin{proof}
The ``only if'' direction follows from Theorem \ref{thm: realizability genus 1}, so we prove the ``if'' direction. Suppose that \((\underline{C}_0, \mathsf{p}')\) satisfies the lifting property. We construct the remaining meromorphic functions \(\nu_{m,V}\) for \(V \in V(\tropC) \setminus V(\tropC_0)\) and \(m \in M_X\) inductively and use Theorem \ref{thm: realizability genus 1}.

For every node \(q \in C_V \cap C_{V'}\) not contained in \(C_0\), we choose local coordinates \(x_{q_V}\) and \(x_{q_{V'}}\) around \(q_V = q \in C_V\) and \(q_{V'} = q \in C_{V'}\). We also fix an ordering of the vertices outside the circuit such that if \(j > j'\), then the path from \(V_{j'}\) to the circuit does not pass through \(V_j\).

Assume that \(\nu_{m,V_j}\) has already been constructed for all \(j < k\). Let \(\vec{e} : V_k \to V_j\) be an oriented edge with \(j < k\). We then choose a meromorphic function \(\nu_{m,V_k}\) satisfying condition (a), and rescale it so that condition (b) holds along \(\vec{e}\) with respect to the chosen coordinates.

Note that the coordinates \(x_{q_V}\) remain fixed throughout and are independent of \(m\).

\end{proof}

\begin{proposition}\label{prop: reduce to nbhd}
Suppose that $(\widetilde{C_0}, \widetilde{\mathsf{p}'})$ is a marked nodal curve whose dual graph is equal to the graph underlying $\sigma_0$, and which satisfies the lifting condition with respect to $\sigma$. 

Then there exists a unique basic fine well-spaced map  $(C^{\mathsf{fine}}/S^{\mathsf{fine}}, \mathsf{p}, f^{\mathsf{fine}}, \rho)$ with combinatorial type $\sigma$ and whose associated cycle $(C_0, \mathsf{p}')$ is isomorphic, as a pointed curve, to $(\widetilde{C_0}, \widetilde{\mathsf{p}'})$.
\end{proposition}

Strictly speaking, we have not defined the combinatorial type of a basic fine well-spaced map
$$
(C^{\mathsf{fine}}, \mathsf{p}, f^{\mathsf{fine}}, \rho)=\mathrm{Spec}(Q^\fine \to k)
\to \mathcal{W}_{\Gamma}(X)^{\mathsf{fine}}.
$$
However, the cone \((Q^{\mathsf{fine}})^\vee=\mathrm{Hom}_{\mathsf{Mon}}(Q^\fine,\N)\) naturally identifies with a cone in \(W_{\Gamma}(\Sigma(X))\). By the combinatorial type of
$
(C^{\mathsf{fine}}/S^{\mathsf{fine}}, \mathsf{p}, f^{\mathsf{fine}}, \rho),
$
here we mean the cone \( (Q^{\mathsf{fine}})^\vee \subset W_{\Gamma}(\Sigma(X))\).

\begin{proof}[Proof of existence in Proposition \ref{prop: reduce to nbhd}]
     Let $\underline{f}: (\underline{C},\mathsf{p}) \rightarrow X$ by a $N$-marked genus $1$ stable map satisfying conditions $(1)$ and $(2)$ of Observation \ref{disc:properties-over-sigma} and with cycle $(\widetilde{C}_0,\widetilde{\mathsf{p}}')$. By Lemma \ref{lem: liftingtrees}, the map $\underline{f}$ admits a fs logarithmic lift $f$ with combinatorial type $\sigma'$. Because $\sigma \subseteq \sigma'$ parametrizes well-spaced tropical curve, by \cite[Theorem 4.4.8]{RSWII}, we obtain a well-spaced map $(C,\mathsf{p},f,\rho)$ with tropicalization in $\sigma$. Checking their proof, we observe that the cycle of $(C,\mathsf{p},f,\rho)$ is equal to the cycle of the map $f$, that is it is equal to $(\widetilde{C}_0,\widetilde{\mathsf{p}'})$. The associated basic fine map $(C^\fine ,\mathsf{p},f^\fine,\rho)$ satisfies the required properties.
\end{proof}

The proof of uniqueness in Proposition \ref{prop: reduce to nbhd} requires a few more results. Let $(C^{\mathsf{fine}}, \mathsf{p}, f^{\mathsf{fine}}, \rho)$ be a basic fine well-spaced map as in Proposition~\ref{prop: reduce to nbhd}.

\begin{lemma} \label{lem: one-branch-vanishing derivative}
Let $H \subset (N_X)_\mathbb{R}$ be an codimension-$1$ hyperplane containing $h(\tropC_0)$ and such that at the minimal distance where $\tropC$ leaves $H$, 
\begin{enumerate}
    \item there is precisely one vertex $V$ whose neighborhood is not contained in $H$; and 
    \item at this vertex $V$, there are precisely three adjacent half-edges not contained in $H$. 
\end{enumerate}

Let $m \in M_X$ be such that 
$
H = \{u \in (N_X)_\mathbb{R} \mid \langle u, m \rangle = 0\}.
$
Let $q_1$ denote the node in $C_V$ corresponding to the edge on the path from $V$ to $\tropC_0$, and let $q_2, q_3, q_4$ be the points on $C_V$ corresponding to the three half-edges not contained in $H$. Then the class of meromorphic function $[\nu_{m,V}]$ on $C_V$ from Theorem \ref{thm: realizability genus 1} satisfies  $d\nu_{m,V}(q_1) = 0$.
\end{lemma}

\begin{proof}
    
    Let $\widetilde{C} \to C$ and $\widetilde{C} \to \overline{C}$ be the destabilization and contraction morphisms associated to $\rho$ and the image of $m$ in $\Gamma(C,\oM_C^{\gr})$ constructed in \S\ref{sec: aligned curves}. Let $\sigma_m$ denote the image of $m$ in $\Gamma(\widetilde{C},\mathcal{M}_{\widetilde{C}}^{\gr})$. We may view $\sigma_m$ as a section of the line bundle $L_m$ dual to the line bundle associated to the image of $m$ in $\Gamma(\widetilde{C},\overline{\mathcal{M}}_{\widetilde{C}}^{\gr})$. Explicitly, we may write
\[
\sigma_m=\{\sigma_{m,V}\}_{V},
\]
where each $\sigma_{m,V}$ is a section of the line bundle $L_{m,V}$ defined in Equation~\eqref{eqn: line bdl LmV} and $V$ runs over the components $\widetilde{C}_V$ of $\widetilde{C}$. The sections $\sigma_{m,V}$ satisfy the compatibility conditions discussed in \S\ref{sec: realizability}.

By assumption, $\sigma$ comes from some $\overline{\sigma} \in \Gamma(\overline{C},\cM_{\overline{C}})$. Again $\overline{\sigma}$ is a section of a line bundle $\overline{L}_m$ and satisfies similar properties.

Because the edge corresponding to $q_1$ is contained in $H$, the degree at $q_1$ of the divisor $D_{m,V}$ in Equation~\eqref{eqn: divisor DmV} is $0$ and, with the notation as in Remark \ref{rmk: computation_mV,mi,nj}, the element $m_V$ associated to the image of $m$ in $\Gamma(C,\oM_{C}^{\gr})$ satisfies $m_V=0$. In particular, $L_{m,V}$ is canonically identified with $\cO$ around $q_1$ and $\sigma_{m,V}$ is just a regular function around $q_1$. 

Similarly, because  $(C^{\mathsf{fine}}, \mathsf{p}, f^{\mathsf{fine}}, \rho)$ is well spaced, the image of $m$ in $\Gamma(\overline{C}, \cM_{\overline{C}}^{\mathrm{gr}})$ yield a line bundle $\overline{L}_m$ on $\overline{C}$ with a section $\overline{\sigma}_m$. The restriction  $\overline{L}_{m,V}= \overline{L}_m|_{C_V}$ is trivial (and $\overline{\sigma}_{m,V}=\overline{\sigma}_{m}|_{C_V}$ is constant) on components $\overline{C}_V$ of $\overline{C}$ which comes from the destabilization $\widetilde{C}$ but do not correspond to components of $C$. Let $\nu$ be the genus $1$ singularity in $\overline{C}$. Because $V$ is the only vertex of $\tropC$ whose neighborhood is not contained in $H$, the component $\overline{C}_V \subseteq \overline{C}$ is the only component of $\overline{C}$ containing $\nu$ on which $\overline{L}_m$ is non-trivial. Moreover, $\overline{L}_m|_{\overline{C}_V} = L_{m,V}$ is canonically identified with $\cO$ around $q_1$. Thus we have an identification $\overline{L}_m = \cO$ around $\nu$ that restricts to the identification of $L_{m,V}=\overline{L}_m|_{\overline{C}_V}$ with $\cO$ around $q_1$. Viewing $\overline{\sigma}_m|_{\overline{C}_V}$ as a section of $\cO$, we have $\overline{\sigma}_m|_{\overline{C}_V}=\sigma_{m,V}$ around $q_1$. The upshot is that it is enough to check that $d(\overline{\sigma}_m|_{\overline{C}_V})(q_1)=0$ when viewing $\overline{\sigma}_m|_{\overline{C}_V}$ as a regular function around $q_1$. 

For this we use coordinates. Recall the following local model for genus $1$ Gorenstein singularities. Étale locally around $\nu$, the curve $\overline{C}$ is isomorphic to either a cusp (one branch), or a tacnode (two branches), or the germ at the origin of the union of $m \geq 3$ general lines through the origin in $\mathbb{A}^{m-1}$ ($m$ branches). The seminormalization map is given by
\begin{equation}\label{eqn: seminormalization}
(\overline{C}',\nu') \to (\overline{C},\nu)
\end{equation}
where $(\overline{C}',\nu')$ is étale locally isomorphic to the germ at the origin of the union of $m$ coordinate axes in $\mathbb{A}^m$. The ring of regular functions on $(\overline{C}',\nu')$ corresponds to the collection of $m$-tuples of function germs $(f_1,\ldots,f_m)$ at the origin of $\mathbb{A}^1$ whose values at $0$ coincide. We can choose the seminormalization map so that the ring of regular functions on the original curve germ $(\overline{C},\nu)$ is the subring corresponding to those tuples that further satisfy the linear condition
\begin{equation}\label{eqn: regular_fcts_barC}
\sum_{\ell=1}^{m}  f_\ell'(0) = 0.
\end{equation}

Taking the $f_i$ to be the sections $\overline{\sigma}_{m,V}$, where $V$ ranges over the irreducible components of $\overline{C}$ containing $\nu$, the above equation becomes exactly $d(\overline{\sigma}_m|_{\overline{C}_V})(q_1)=0$.
\end{proof}

\begin{lemma} \label{lem: one-branch}
Let $H \subset (N_X)_\mathbb{R}$, $V \in V(\tropC)$, $m \in M_X$ and the points $q_1,q_2,q_3,q_4 \in C_V$ be as in Lemma \ref{lem: one-branch-vanishing derivative}. Then the isomorphism class of $(C_V, q_1, \ldots, q_4)$ is determined by the existence of a meromorphic function $\nu_{m,V}$ satisfying lifting condition (a) and with $d\nu_{m,V}(q_1) = 0$.
\end{lemma}

\begin{proof}
    Let $\zeta_1, \ldots,  \zeta_4$ be the half-edges adjacent to $V$ corresponding to points $q_1, \ldots, q_4 \in C_V$. Choose coordinates $t$ on $C_V$ such that these points are $p, 0, 1, \infty$ respectively. Write
    $$\nu_{m, V} (t)= At^{m(u_{\zeta_2})}(t-1)^{m(u_{\zeta_3})}$$
    for some $A \in \mathbb{C}^*$. We want $d\nu_{m, V}(p) = 0$, which has a unique solution
    $$p = \frac{m(u_{\zeta_2})}{m(u_{\zeta_2}) + m(u_{\zeta_3})}.$$
    Note that this is distinct from $0, 1, \infty$, since $m(u_{\zeta_2}), m(u_{\zeta_3}) \neq 0$. 
\end{proof}

Recall the notation introduced in Definition~\ref{def: filtration}. Part of the data of $\sigma$ is a radial alignment 
$
\rho : V(\tropC) \to \{0,\ldots,k\}.
$
To $\sigma$ we have associated a filtration
$$
H_0 \subseteq H_1 \subsetneq \cdots \subsetneq H_e \subseteq (N_X)_\mathbb{R}
$$
together with integers $a_i \in \{0,\ldots,k\}$ such that a neighbourhood of the subcurve of $\tropC$ whose vertices satisfy $\rho(V) \leq a_i$ is contained in $H_i$ but not in $H_{i-1}$. The vertices in $\rho^{-1}(a_i)$ are denoted by $V_{i,j}$ for $j = 1,\ldots, h_i$. Each $V_{i,j}$ has $f_{i,j}$ half-edges $\zeta_{i,j,k}$ pointing away from $\tropC_0$, i.e.\ half-edges belonging to edges connecting $V_{i,j}$ to vertices $V$ with $\rho(V) > a_i$. The direction vector of $\zeta_{i,j,k}$ is $u_{i,j,k} \in N_X$.

\begin{lemma} \label{lem: good-plane}
    For any $i=2,\ldots, e$ and $j=1,\ldots, h_{i}$ such that $f_{i,j} \geq 3$ and for any choice of $1 \leq k, k', k'' \leq f_{i,j}$ distinct, there exists a codimension $1$ subspace $H$ of $(N_X)_{\mathbb{R}}$ such that the neighborhood of the cycle $\nbhdtropC$ is contained inside $H$, and at the minimum distance from the cycle to where the curve leaves $H$, it leaves along exactly the half-edges $\zeta_{i,j,k}, \zeta_{i,j,k'}, \zeta_{i,j,k''}$. 
\end{lemma}

For given vertex $V$ \emph{outside} of the cycle $\tropC_0$ and with valency bigger than $3$, the above lemma is providing an hyperplane as in Lemma \ref{lem: one-branch}, thus pinning down the position of all the nodes and non-free markings on $C_V$. Note that the free marking $p_i$ are necessarily attached to three valent verteces because by assumption $\sigma$ is a maximal cone. This determines the isomorphism class of the domain curve $(\underline{C},\mathsf{p})$ of $(C^{\mathsf{fine}}, \mathsf{p}, f^{\mathsf{fine}}, \rho)$.

\begin{proof}
Suppose $f_{i,j} \geq 3$. Without loss of generality, we may assume that $j= h_i$, $k=h_i-3,k'=h_i-1$ and $k'=h_i$. Consider the set 
\begin{equation} \label{eqn: basis vectors}
\mathcal{U} = \{u_{i,1,1}, \ldots, u_{i,1,f_{i,1}-1}\} \cup  \ldots \cup\{ u_{i,h_i-1,1}, \ldots, u_{i,h_i-1,f_{i,h_i-1}-1}\} \cup \{u_{i, h_i,1}, \ldots, u_{i,h_i, f_{i,h_i}-2}\}
\end{equation}
of all but one direction vectors of the half-edges in direction opposite to $\tropC_0$ adjacent to $V_{i,j}$ for $j = 1, \ldots, h_i -1$, and all but two direction vectors for $V_{i,h_i}$. If $H_{i-1} + \Span\mathcal{U} \subseteq H_{i}$ has positive codimension, then the curve cannot be well-spaced. Indeed, let $H'$ be any codimension $1$ subspace of $H_{i}$ containing $H_{i-1} + \Span\mathcal{U}$, but not containing $u_{i,h_i, f_{i,h_i}-1}$ and $ u_{i,h_i, f_{i,h_i}}$. Let $H''$ be any complement of $H_{i}$ in $(N_X)_{\mathbb{R}}$. Then $H'+ H''$ is a codimension $1$ subspace of $(N_X)_{\mathbb{R}}$ for which the well-spacedness condition fails (at the minimum distance from the cycle to where the curve leaves the plane, it leaves along exactly two half-edges). Thus, $H_{i-1} + \Span\mathcal{U} = H_{i}$ and the vectors in $\mathcal{U}$ span $H_i/H_{i-1}$. 
From \cite[Proposition 3.2.18(c)]{Tor14}, for each $i = 2, \ldots, e$  have 
    \begin{equation} \label{eqn: tor c}
        \sum_{j=1}^{h_i} (\val{V_{i,j}}-1) - h_i - 1 = \dim H_{i} - \dim H_{i-1}. 
    \end{equation}

Since the left-hand side is equal to the cardinality of $\mathcal{U}$. 
So under the projection $H_{i} \rightarrow H_{i}/H_{i-1}$, the elements of $\mathcal{U}$ form a basis of $H_{i}/H_{i-1}$. In particular, the set 
$$\mathcal{U}' = \{u_{i,1,1}, \ldots, u_{i,1,f_{i,1}-1}\} \cup  \ldots \cup\{ u_{i,h_i-1,1}, \ldots, u_{i,h_i-1,f_{i,h_i-1}-1}\} \cup \{u_{i, h_i,1}, \ldots, u_{i,h_i, f_{i,h_i}-3}\},$$
where now we remove three vectors for $V_{i,h_i}$, has the property that $H_{i-1} + \Span \mathcal{U}' + H''$ is a codimension $1$ subspace in $(N_X)_\R$ with the required properties.

\end{proof}

\begin{lemma}\label{lemma: all_but_4_vect}
    Fix $i \in \{2, \ldots, e\}$, and suppose that $h_i \geq 2$. Let $j,j' \in \{1, \ldots, h_i\}$ be distinct indices, and for each of them choose two vectors 
\[
u_{i,j,a_j}, \, u_{i,j,b_j}
\qquad\text{and}\qquad
u_{i,j',a_{j'}}, \, u_{i,j',b_{j'}}.
\]
Then there exists a codimension-one subspace
$
H \subset (N_X)_{\mathbb{R}}
$
such that $H$ contains $H_{i-1}$ and all vectors $u_{i,j,k}$ for $j=1,\ldots , h_i$ and $k=1,\ldots, f_{i,j}$ except for these four vectors.
\end{lemma}

\begin{proof}
    Without loss of generality,  $a_j = f_{i,j} -1, b_j = f_{i,j}$, $j' = h_i, a_{j'} = f_{i,h_i}-1, b_{j'} = f_{i,h_i}$. First we find a codimension $1$ subspace $H' \subset H_i$ containing $H_{i-1}$ and all $u_{i,j,k}$ except for $u_{i,h_i, f_{i,h_i}-1}, u_{i,h_i, f_{i,h_i}}, u_{i,j,f_{i,j}-1}$, and $u_{i,j,f_{i,j}}$. The vectors in the set $\mathcal{U}$ in Equation \ref{eqn: basis vectors} form a basis of $H_i/H_{i-1}$, and so $\mathrm{Span}_\R (\mathcal{U} \smallsetminus \{u_{i,j,f_{i,j}-1}\}) +H_{i-1}$ spans a codimension $1$ subspace $H'$ of $H_i$. 
    
    We claim that $H'$ does not contain $u_{i,h_i, f_{i,h_i}-1}, u_{i,h_i, f_{i,h_i}}, u_{i,j,f_{i,j}-1}$, or $u_{i,j,f_{i,j}}$. It does not contain $u_{i,j,f_{i,j}-1}$, so by balancing it does not contain $u_{i,j,f_{i,j}}$. By symmetry, the set
    $$(\mathcal{U} \smallsetminus \{u_{i,j,f_{i,j}-1}\} ) \cup \{u_{i,h_i, f_{i,h_i}-1}\} $$
    is also a basis of $H_i/H_{i-1}$, and so $u_{i,h_i, f_{i,h_i}-1} \notin H'$. Again by balancing, we have $u_{i,h_i,f_{i,h_i}} \notin H'$. To conclude, let $H''$ be a complement of $H_i$ in $(N_X)_{\mathbb{R}}$ and set $H = H' + H''$. 
\end{proof}

\begin{proof}[Proof of uniqueness in Proposition \ref{prop: reduce to nbhd}]

    To check the uniqueness of $(C^\fine ,\mathsf{p},f^\fine,\rho)$ it is enough to check uniqueness of the fine map $f^\fine : (C^\fine, \mathsf{p}) \to X$ and of the radial alignment $\rho$. 
    
   We start with the map. Lemmas \ref{lem: one-branch} and \ref{lem: good-plane} show that there is at most one isomorphism class for the pointed curve $(C,\mathsf{p})$. Also, the map $\underline{f}$ is, up to isomorphism, determined by Observation \ref{obs: further properties from sigma}(2). Its fine log enhancement is also unique by Lemma \ref{lemma: uniqueness of fine maps}. 

    Next we prove uniqueness of the radial alignment $\rho$. Recall that 
\begin{equation}\label{eqn:locus_of_rho}
\rho = (\rho_{i,j})_{\substack{i=2,\ldots,e \\ j=2,\ldots,h_i}}
\in \prod_{i=2}^e \mathbb{G}_m^{\,h_i-1}
=
\prod_{i=1}^{\ell(\rho)} \mathbb{G}_m^{\,|\rho^{-1}(i)|-1}.
\end{equation}
where each factor $\mathbb{G}_m^{h_i-1}$ is naturally identified with the collections of isomorphisms
\begin{equation}
 T_{\nu_{V_{i,1}}}C_{V_{i,1}} \simeq T_{\nu_{V_{i,j}}}C_{V_{i,j}}  \ \text{for} \ j=2,\ldots, h_i.
\end{equation}
Here $\nu_V$ is the node associated to the unique edge $e$ of $\mathsf{C}$ connecting $V$ to a vertex $V'$ with $\rho(V') < \rho(V)$. The identity in Equation \eqref{eqn:locus_of_rho} uses Lemma \ref{lemma: Ha'=Ha+1'}.

Fix $i=2,\ldots, e$ and $j=2,\ldots,h_i$. We want to prove that each $\rho_{i,j} \in \mathbb{G}_m$ is uniquely determined.

By Lemma \ref{lemma: all_but_4_vect}, there exists a codimension $1$ hyperplane $H \subseteq (N_X)_\R$ containing $H_{i-1}$ and all $u_{i,j,k}$ except for $u_{i,1, f_{i,1}-1}, u_{i,1, f_{i,1}}, u_{i,j,f_{i,j}-1}$, and $u_{i,j,f_{i,j}}$.

Let $m \in (N_X/(H \cap N_X))^*$ be a generator. Let $\widetilde{C} \to C$ and $\widetilde{C} \to \overline{C}$ be the destabilization and contraction morphisms associated to $\rho$ and the image of $m$ in $\Gamma(C,\oM_C^{\gr})$ constructed in \S\ref{sec: aligned curves}. Let $\sigma_m$ denote the image of $m$ in $\Gamma(\widetilde{C},\mathcal{M}_{\widetilde{C}}^{\mathrm{gr}})$. We may view $\sigma_m$ as a section of the line bundle $L_m$ dual to the line bundle associated to the image of $m$ in $\Gamma(\widetilde{C},\overline{\mathcal{M}}_{\widetilde{C}}^{\mathrm{gr}})$. Explicitly, we may write
\[
\sigma_m=\{\sigma_{m,V}\}_{V},
\]
where each $\sigma_{m,V}$ is a section of the line bundle $L_{m,V}$ defined in Equation~\eqref{eqn: line bdl LmV} and $V$ runs over the components $\widetilde{C}_V$ of $\widetilde{C}$. For every node $q \in \widetilde{C}_V \cap \widetilde{C}_{V'}$ of $\widetilde{C}$, there is an identification
$
L_{m,V} \simeq L_{m,V'}
$
as in Equation~\eqref{eqn: glueing LmV}, and the line bundle $L_m$ is obtained by gluing the $L_{m,V}$ via these identifications. Moreover, under this identification, we have
\[
\sigma_{m,V}(q_V)=\sigma_{m,V'}(q_{V'}).
\]

By assumption, $\sigma$ comes from some $\overline{\sigma} \in \Gamma(\overline{C},\cM_{\overline{C}})$. Again $\overline{\sigma}$ is a section of a line bundle $\overline{L}_m$ and satisfies similar properties. Note that $\overline{L}_{m,V}$ is trivial (and $\overline{\sigma}_{m,V}$ is constant) on components $\overline{C}_V$ of $\overline{C}$ which comes from the stabilization $\widetilde{C}$ but do not correspond to components of $C$. 

The vertex $V_{i,j}$ of $\tropC$ corresponds to a component $\overline{C}_{V_{i,j}}$ of $\overline{C}$ and the point $\nu:=\nu_{V_{i,j}}$ is the genus $1$ singular point of $\overline{C}$. As in the proof of Lemma \ref{lem: one-branch-vanishing derivative}, the line bundle $\overline{L}_m|_{\overline{C}_V}=L_{m,V}$ is canonically identified with $\cO$ around $\nu_{V_{i,j}}$ and the differential
\[
d\overline{\sigma}_{m,V_{i,j}}(\nu_{V_{i,j}})\colon T_{\nu_{V_{i,j}}}\overline{C}_{V_{i,j}} \to \C 
\]
is necessarily an isomorphism.

To see this, let $t$ be a coordinate on $\overline{C}_{V_{i,j}}$ around $\nu_{V_{i,j}}$ such that $\nu_{V_{i,j}}$ corresponds to the point $1$, while the nodes of $\overline{C}$ in $C_{V_{i,j}}$ corresponding to the half-edges $\xi_{i,j,f_{i,j}-1}$ and $\xi_{i,j,f_{i,j}}$ correspond to $0$ and $\infty$, respectively. Then
\[
\overline{\sigma}_{m,V_{i,j}}(t)=A t^{m(u_{f_{i,j}-1})},
\qquad \text{with } A \neq 0.
\]
whose derivative at $t=1$ is clearly nonzero.

For the same reason the differential $d\overline{\sigma}_{m,V_{i,1}}(\nu_{V_{i,1}})\colon T_{\nu_{V_{i,1}}}\overline{C}_{V_{i,1}} \to \C$
is an isomorphism. We claim that the isomorphism $\rho_{i,j}: T_{\nu_{V_{i,1}}}C_{V_{i,1}} \simeq T_{\nu_{V_{i,j}}}C_{V_{i,j}}$ must be equal to the composition 
\begin{equation}\label{eqn: rho_ij}
T_{\nu_{V_{i,1}}}C_{V_{i,1}} \xrightarrow{(d\sigma_{m,V_{i,1}}(\nu_{V_{i,1}}))^{-1} } \C\xrightarrow{d\sigma_{m,V_{i,j}}(\nu_{V_{i,j }})} T_{\nu_{V_{i,j}}}C_{V_{i,j}}.
\end{equation}
Uniqueness in Proposition \ref{prop: reduce to nbhd} would then follow.

To see this, we use coordinates. Étale locally around $\nu$, the curve $\overline{C}$ is isomorphic to either a tacnode, modeled by $V(y^2-yx^2)$ for $m=2$ branches, or the germ at the origin of the union of $m \geq 3$ general lines through the origin in $\mathbb{A}^{m-1}$. Note that because $h_i > 1$, the singularity has multiple branches and therefore cannot be a cusp. Using seminormalization map in Equation \eqref{eqn: seminormalization} together with the description of the sheaf of regular functions on $\overline{C}$ at $\nu$ in Equation \eqref{eqn: regular_fcts_barC}, and taking the $f_i$ to be the sections $\overline{\sigma}_{m,V}$, where $V$ ranges over the irreducible components of $\overline{C}$ containing $\nu$, the above equation becomes
\[
\overline{\sigma}_{m,V_{i,1}}'(\nu_{V_{i,1}})
+
\overline{\sigma}_{m,V_{i,j}}'(\nu_{V_{i,j}})
=0.
\]
In other words, under the coordinates provided by the seminormalization map~\eqref{eqn: seminormalization}, both the gluing isomorphism $\rho_{i,j}$ and the map~\eqref{eqn: rho_ij}
\[
\mathbb{C}
\simeq
T_{\nu_{V_{i,1}}}C_{V_{i,1}}
\longrightarrow
T_{\nu_{V_{i,j}}}C_{V_{i,j}}
\simeq
\mathbb{C}
\]
are given by multiplication by $-1$.

\end{proof}

\subsubsection{Liftings for the cycle} \label{sec: liftings for the cycle}

To prove Theorem \ref{thm: calculation of weights}, we need to find all cycles $(C_0, \mathsf{p}')$ satisfying the lifting condition with respect to $\sigma$. Then by Proposition \ref{prop: reduce to nbhd}, for each such cycle $(C_0, \mathsf{p}')$ there exists a unique basic fine map with combinatorial type $\sigma$ and cycle $(C_0,\mathsf{p}')$ and we count it with weight $\prod \mu_V/\Aut(f_j^\fine)$. 

Let $n_1, \ldots, n_{r_0}$ be a basis of $H_0\cap N_X$. Extend it to a basis $n_1, \ldots, n_{r_1}$ of $H_1 \cap N_X$, then extend again to a basis $n_1, \ldots, n_r$ of $N_X$. Let $m_i \in M_X$ be the dual basis. We find appropriate meromorphic functions $\nu_{m_i, V}$ on $C_V \subseteq C_0$ for $V \in V(\tropC_0)$ in groups: $i = 1, \ldots, r_0$, $i = r_0+1, \ldots, r_0+r_1$, and $i = r_0+r_1 +1, \ldots, r$. Only the second group will give conditions on $(C_0, \mathsf{p}')$ for the lifting condition with respect to $\sigma$ to be satisfied. 

\begin{lemma} \label{lem: lift in cycle}
    Fix any $(C_0, \mathsf{p}')$ marked nodal curve whose dual graph is equal to the graph underlying $\sigma_0$. There exists meromorphic functions $\nu_{m_i, V}$ on $C_V$ for $V \in V(\tropC_0)$ and for $i = 1, \ldots, {r_0}$, and coordinates $x_{q_V}$ on $C_V$ for every node $q$ corresponding to edge in $\tropC_0$ to $V$, such that \begin{itemize}
        \item for $i = 1, \ldots, {r_0}$ the $\nu_{m_i,V}$ satisfies lifting condition (a);
        \item for every oriented edge $\vec{e}\colon V \to V'$ in $\tropC_0$, the pair $\nu_{m_i, V}, \nu_{m_i, V'}$ satisfies lifting condition (b) with respect to $x_{q_V}$ and $x_{q_{V'}}$.
    \end{itemize}
    
\end{lemma}

\begin{proof}
    Let $V_1, \ldots, V_{E_0}$ be the vertices of $\tropC_0$. For convenience, we work modulo $E_0$ for the indexing of the edges and vertices. Let $\vec{e}_j$ be the oriented edge $V_j \rightarrow V_{j+1}$ for $j = 1, \ldots, E_0$. Let $u_j = u_{\vec{e}_j}$. Choose privileged indices $j_1, \ldots, j_{r_0} \in \{1, \ldots, E_0\}$  such that $u_{j_1}, \ldots, u_{j_{r_0}}$ span $H_0$ over $\mathbb{R}$. 

    For $i = 1, \ldots, r_0$ and for every $V_j$ in the cycle, fix any lifts $\tilde{\nu}_{m_i, V_j}$ satisfying the lifting condition (a). For each edge $e_j$, fix coordinates $x_{(q_j)_{V_j}}$ and $x_{(q_j)_{V_{j+1}}}$ on $C_{V_j}$ and $C_{V_{j+1}}$ respectively around the node $q_j$ corresponding to edge $e_j$. Let $A_{i,j} \in \mathbb{C}$ and $\mu_{1}, \ldots, \mu_{r_0} \in \mathbb{C}^*$. For each privileged index $k=1,\ldots, r_0$, consider change of coordinates $x_{(q_{j_k})_{V_{j_k}}}\mapsto \mu_{k} x_{(q_{j_k})_{V_{j_k}}}$, and define $\nu_{m_i, V_j} \colon = A_{i,j}\tilde{\nu}_{m_i, V_j}$. 
    We want to find $A_{i,j}$ and $ \mu_k$ such that for every oriented edge $\vec{e} \colon V \rightarrow V'$ in the circuit, the pair $\nu_{m_i, V}, \nu_{m_i, V'}$ satisfies the lifting condition (b). 
 
    Let $c_{i,j}$ be the first non-zero coefficient of the expansion of $\tilde{\nu}_{m_i, V_j}$ around the node $q_j$ with respect to coordinates $x_{(q_j)_{V_j}}$ and let $c'_{i,j}$ be the first non-zero coefficient of the expansion of $\tilde{\nu}_{m_i, V_{j+1}}$ around the node $q_j$ with respect to coordinates $x_{(q_j)_{V_{j+1}}}$. 
    For each $i = 1, \ldots, r_0$, and for each oriented edge $\vec{e}_j$, the lifting condition (b) on the pair $\nu_{m_i, V_j}$, $\nu_{m_i, V_{j+1}}$ gives:
    $$ \begin{cases} 
        c'_{i,j} A_{i,j+1} = c_{i,j} A_{i,j} & \text{if } j \notin \{j_1,\ldots, j_{r_0} \} \\ 
        c'_{i,j_k} A_{i,j_k+1} = c_{i,j_k} \mu_k^{m_i(u_{j_k})} A_{i,j_k} & \text{if } j=j_k \in \{j_1,\ldots, j_{r_0} \} 
    \end{cases}$$
     This difference accounts for our chosen change of variables $x_{(q_{j_k})_{V_{j_i}}} \mapsto \mu_{k} x_{(q_{j_k})_{V_{j_k}}}$ at the privileged indices $j_k$. Note that the order of $\nu_{m_i,V_{j_k}}$ at the node corresponding to $e_{j_k}$ is $m_i(u_{j_k})$. 

    For each $i= 1, \ldots, r_0$, this gives a relation 
    $$A_{i,1} = \bigg(\prod_j \frac{c_{i,j}}{c'_{i,j}}\bigg) \mu_1^{m_i(u_{j_1})} \ldots \mu_{r_0}^{m_i(u_{j_{r_0}})}A_{i,1}$$
    So, independently of $A_{i,1}$, we have reduced to solving
    $$\mu_1^{m_i(u_{j_1})} \ldots \mu_{r_0}^{m_i(u_{j_{r_0}})} = C_i$$
    for $C_i = \prod_j \frac{c'_{i,j}}{c_{i,j}} \in \C^*$. This is possible since, by the choice of the privileged indices $j_k$, we have 
    $$\det((m_i(u_{j_k}))_{i,k= 1, \ldots, r_0}) \neq 0.$$ Indeed, let $M= (m_i(u_{j_k}))_{i,k= 1, \ldots, r_0}$. One can first solve the real equations
    $$|\mu_1|^{m_i(u_{j_1})} \ldots |\mu_{r_0}|^{m_i(u_{j_{r_0}})} = |C_i|$$
    by taking logarithms to obtain a system of linear equations with matrix $M$. This determines the modulus $|\mu_k|$ for $i=1,\ldots,r_0$. Then replacing $\mu_k$ by $\mu_k/|\mu_k|$, we can assume $|\mu_k| =|C_i| = 1$ for all $k$. Write $\mu_k = e^{2 \pi \sqrt{-1} \times \mu_k'} $, $C_i = e^{2\pi \sqrt{-1} \times C_i'}$ with $\mu_k' , C_i'\in [0,2 \pi)$. Then it suffices to solve
    $$\sum_k m_i(u_{j_k}) \mu_k' = C_i' .$$
    This is again a system of linear equations with matrix $M$.
\end{proof}

\begin{corollary} \label{cor: cyclespans}
    Suppose $\sigma$ is a maximal cone in $W_\Gamma(\Sigma(X))$ corresponding to a tropical type with the property that $H_0 = H_1$. Then $n(\sigma) = 1$ and  $\nn(\sigma) = 1/|\Aut(\sigma)|$. 
\end{corollary} 

By Proposition \ref{prop: nbhd-cycle-span}, the request that $H_0=H_1$ is equivalent to asking that every vertex in the cycle has valence at most $3$. In particular, Corollary \ref{cor: cyclespans} implies Theorem \ref{thm: calculation of weights}(a).

\begin{proof}
    Let $(C_0, \mathsf{p}')$ be the unique (up to isomorphism) marked nodal curve whose dual graph is equal to the graph underlying $\sigma_0$. Let $\nu_{m_i, V_j}$ be the meromorphic functions on $C_{V_j}$ for $i = 1, \ldots, r_0$ and $V_j \in V(\tropC_0)$ found in Lemma~\ref{lem: lift in cycle}. For each $i \in \{r_0+1, \ldots, r\}$, let $\nu_{m_i, V_j}$ be a constant independent of $j$ for all $V_j \in V(\tropC_0)$. Then, because $H_0=H_1$, the cycle $(C_0, \mathsf{p}')$ satisfies the lifting property with respect to $\sigma$ as in Definition \ref{def: lifting property C0}. 
    By Proposition \ref{prop: reduce to nbhd}, there is a unique minimal well-spaced logarithmic fine curve with combinatorial type $\sigma$ and whose associated cycle is $(C_0, \mathsf{p}')$. Because $(C_0, \mathsf{p}')$ is unique, this proves that $n(\sigma)=1$.
    
    Finally, we account for automorphisms and compute $\nn(\sigma)$. Let $(C^\fine,\mathsf{p}, f_\sigma^\fine, \rho)$ be the basic fine map lying over $\sigma$. We claim that the monomorphism in Equation~\eqref{eqn:aut_injection} is an isomorphism in this case. The key point is that the underlying curve $\underline{C}^\fine$ is trivalent.

    More explicitly, by Corollary~\ref{cor: automorphisms}, it suffices to assume that $\mathrm{Aut}(\sigma) \simeq \mu_2$ and that we are in one of the following cases:

\begin{enumerate}
    \item[$\bullet$] The underlying graph of $\sigma$ has a self-loop mapped to the interior of a maximal cone. In this case, the cycle $C_0 \subseteq C^\fine$ consists of a single irreducible component with a self-node and one additional special point. Let $x$ be coordinates on this component so that the nodes corresponding to the self-intersecion are located at $0$ and $\infty$, and the remaining special point is at $1$. The nontrivial automorphism is then given by
    $
    x \mapsto 1/x.
    $

    \item[$\bullet$] The underlying graph of $\sigma$ consists of two vertices connected by two edges with the same direction vector, possibly with additional $2$-valent vertices along the edges. In this case, the cycle $C_0$ consists of two irreducible components $C_{V}$ and $C_{V'}$ joined by two chains of rational curves corresponding to the $2$-valent vertices. There is one additional special point on each of $C_{V}$ and $C_{V'}$. We put coordinates on $C_{V}$ and $C_{V'}$ such that the two points attached to the chains are at $0$ and $\infty$, and the additional special points are at $1$. The nontrivial automorphism acts by simultaneously swapping $0$ and $\infty$ on $C_{V}$ and $C_{V'}$, and by the corresponding inversions on the intermediate rational curves in the chains.

\end{enumerate}

\end{proof}

As in Notation \ref{notation: shape nbhd cycle}, choose an orientation around the loop. Let $V_1,\ldots, V_t$ be the vertices in the cycle with valency bigger than $3$ and ordered following the orientation of the loop. Label the half-edges attached to $V_j$ so that $\xi_{j,1}$ corresponds to the edge directed toward $V_{j-1}$ and $\xi_{j,2}$ corresponds to the edge directed toward $V_{j+1}$. Let $\zeta_{j,k}$ be the remaining half-edges attached to $V_j$. There are $e_j$ of these. Note that $e_j \geq 4$. Let $x_j$ be coordinates on $C_{V_j}$ so that $\xi_{j,1}, \xi_{j,2}, \zeta_{j, e_j}$ are $\infty, 0, 1$ respectively. 

\begin{lemma}\label{lemma: lifting_property_equation}
    The marked scheme $(\underline{C}_0, \mathsf{p}')$ satisfies the lifting property with respect to $\sigma$ if and only if for each $i = r_0 +1, \ldots r_1$, we have
    \begin{equation} \label{eqn: big product}
        \prod_{j = 1}^t \prod_{k=1}^{e_j -1} q_{j,k}^{m_i(u_{j,k})} = 1.
    \end{equation}
    Here $q_{j,k} \in C_{V_j}$ is the node corresponding to $\zeta_{j,k}$ and $u_{j,k}=u_{\zeta_{j,k}}$.
\end{lemma}

\begin{proof}
    
     From Lemma \ref{lem: lift in cycle}, there are meromorphic functions $\nu_{m_i, V}$ for $i = 1, \ldots, r_0$ and $V \in V(\tropC_0)$, as well as coordinates $x_{q_V}$ around every node in the cycle, such that $\nu_{m_i,V}$ satisfies lifting condition (a) and for each edge $e$ between $V$ and $V'$, $\nu_{m_i,V}, \nu_{m_i,V'}$ satisfy lifting condition (2) with respect to $x_{q_V}$ and $x_{q_{V'}}$. 
     For each $i \in \{r_0+r_1 +1, \ldots, r\}$, let $\nu_{m_i, V_j}$ be a constant independent of $j$ for all $V_j \in V(\tropC_0)$. It remains to find appropriate meromorphic functions $\nu_{m_i,V}$ for $i = r_0+1, \ldots, r_1$, and $V \in V(\tropC_0)$. This will give the constraints among the $q_{j,k}$ in Equation \eqref{eqn: big product}.

    For $i = r_0+1, \ldots, r_1$ and $j=1,\ldots,t$ we write
    \begin{equation*}\label{eqn: lifting equation}
        \nu_{m_i, V_j}(x_j) = A_{i,j} \prod_{k=1}^{e_j-1} (x_j-q_{j,k})^{m_i(u_{j,k})} \times (x_j-1)^{m_i(u_{j,e_j})}
    \end{equation*}
    for some $A_{i,j} \in \mathbb{C}^*$. For $V \in V(\tropC_0) \smallsetminus \{V_1,\ldots,V_t\}$, we set $\nu_{m_i, V}$ to be constant equal to $A_{i,V} \in \C$. 
    
    Since $0$ and $\infty$ are not zeros or poles of $\nu_{m_i,V_j}$ (for $i = r_0+1, \ldots, r_1$ and $j=1,\ldots,t$), the lifting condition~(b) (with respect to any coordinates, and thus in particular the coordinates $x_{q_V}$ and $x_{q_{V'}}$) requires that 
    \begin{equation} \label{eqn: lift existence}
        A_{i,j} \prod_{k=1}^{e_j-1} q_{j,k}^{m_i(u_{j,k})} = \nu_{m_i, V_j}(0) = \nu_{m_i, V_{j+1}}(\infty) = A_{i,j+1} \ \  \text{for} \ \ j=1,\ldots,t \ \text{and} \ i=r_0+1,\ldots, r_1
    \end{equation}
    and that $A_{i,V}= \nu_{m_{i}, V_j}(0)= \nu_{m_{i}, V_{j+1}}(\infty)$ if $V$ is between $V_j$ and $V_{j+1}$.
    It follows that 
    \begin{equation*} 
        A_{i,1} \prod_{j = 1}^t \prod_{k=1}^{e_j -1} q_{j,k}^{m_i(u_{j,k})} = A_{i,1}
    \end{equation*}
    which implies Equation \eqref{eqn: big product}. 
    Conversely, given $q_{j,k}$ which satisfies Equation \eqref{eqn: big product}, there exists $A_{i,j}$ such that Equation \eqref{eqn: lift existence} is satisfied. 

\end{proof}

\subsubsection{Explicit computations}\label{sec: computations}

Given any solution of Equation~\eqref{eqn: big product}, we obtain a cycle $(\underline{C}_0, p)$ satisfying the lifting property with respect to $\sigma$. By Proposition~\ref{prop: reduce to nbhd}, this determines a unique basic fine logarithmic map $(C^{\mathsf{fine}}, \mathsf{p}, f^{\mathsf{fine}}, \rho)$ lying over $\sigma$.

However, distinct choices of solutions may still give rise to isomorphic basic fine logarithmic maps, or to maps with nontrivial automorphism groups. In order to compute $n(\sigma)$ and $\nn(\sigma)$, we must understand precisely when this occurs. We start with computing the number of solutions of Equation \eqref{eqn: big product}.

\begin{lemma}\label{lemma: solution_interior_points}
     The number of solutions to Equation \eqref{eqn: big product} is the number of interior points of the region $P(\sigma)$ in Construction \ref{const: polytope}. 
\end{lemma}

\begin{proof}
    Let $u_{j,k}=u_{\zeta_{j,k}}$ and 
    let $M$ be the matrix 
    $$M = \begin{bmatrix}
        m_{r_0+1}(u_{1,1}) & \ldots &  m_{r_0+1}(u_{1,e_1-1}) & \ldots & m_{r_0+1}(u_{t,1}) & \ldots m_{r_0+1}(u_{t,e_t-1}) \\
        m_{r_0+2}(u_{1,1}) & \ldots &  m_{r_0+2}(u_{1,e_1-1}) & \ldots & m_{r_0+2}(u_{t,1}) & \ldots m_{r_0+2}(u_{t,e_t-1}) \\
         & \ddots &  & \ddots &  & \ddots  \\
        m_{r_1}(u_{1,1}) & \ldots &  m_{r_1}(u_{1,e_1-1}) & \ldots & m_{r_1}(u_{t,1}) & \ldots m_{r_1}(u_{t,e_t-1}) \\
    \end{bmatrix}$$
    By Proposition \ref{prop: nbhd-cycle-span}, this is a $(r_1 - r_0) \times (r_1 - r_0)$ matrix. Note this is the same matrix $M$ as in Construction \ref{const: polytope} after the identification of $H_1/H_0$ with $\R^{r_1-r_0}$ via $n_{r_0+1},\ldots,n_{r_0+r_1}$. We claim that $|q_{j,k}| = 1$. Using row operation we can make $M$ lower triangular. This is equivalent to a change of basis $m_i$ ($i=r_0+1,\ldots,r_0+r_1$) of $H_1/H_0$. Then Equation \eqref{eqn: big product} for $i=1$ reads $q_{1,1}^{m_{r_0+1}(u_{1,1})} = 1$, so $|q_{1,1}| = 1$. Considering one row at a time, the claim follows.

    Write $q_{j,k} = e^{2\pi \sqrt{-1} q_{j,k}'}$, with $q_{j,k}' \in \mathbb{R}/\mathbb{Z}$. Let 
    $$
    \underline{q} = (q_{1,1}', \ldots, q_{1,e_1-1}', \ldots, q_{t,1}'\ldots, q_{t,e_t-1}') \in (\R/\Z)^{r_1-r_0}
    $$  Then the relation in Equation \eqref{eqn: big product} is equivalent to $M \underline{q}=0 \in (\R/\mathbb{Z})^{r_1 - r_0} $. Recall $\bar{u}_{j,k}$ is the image of $u_{j,k}$ under the projection $H_1 \rightarrow H_1/H_0 \simeq \R^{r_1-r_0}$ as in Construction~\ref{const: polytope}.  The map $M: (\mathbb{R}/\mathbb{Z})^{r_1 - r_0} \rightarrow (\mathbb{R}/\mathbb{Z})^{r_1 - r_0}$  factors as 
    \begin{center}
        % https://tikzcd.yichuanshen.de/#N4Igdg9gJgpgziAXAbVABwnAlgFyxMJZABgBpiBdUkANwEMAbAVxiRAAoAdTgWzpwAWAIyHAASgF8A9Nz6CRwAFoSAlAD0ATiAml0mXPkIoAjOSq1GLNrP7DRktQGYpXXrYXKmAfUcACANS+3AxQEDhwAUFu8qKeXsAa-gBMqtq6IBjYeARESWbU9MysiBw2MeLSZXZKqk7a5jBQAObwRKAAZhoQPEhkIDgQSKYWRdacAMYETWkdXT2IwwNIeSNWJdw4MAAeOMBYUL4EvhJRctUOGr71EkA
    \begin{tikzcd}
    (\mathbb{R}/\mathbb{Z})^{r_1 - r_0} \arrow[r, "\simeq"] & \mathbb{R}^{r_1 - r_0}/(\mathbb{Z}\bar{u}_{1,1} + \ldots + \mathbb{Z}\bar{u}_{t,e_t-1}) \arrow[r] & (\mathbb{R}/\mathbb{Z})^{r_1 - r_0}
    \end{tikzcd}
    \end{center}
    where the first map is induced by $M$ on $\mathbb{R}^{r_1-r_0}$, and the second arrow is induced by the identity on $\mathbb{R}^{r_1 - r_0}$. Thus, the number of solutions satisfying \(q_{j,k} \neq 1\) for all \(j,k\) is equal to the number of interior lattice points of the parallelotope whose vertices are
    \[
    \sum_{s \in S} \overline{u}_s,
    \]
    where \(S\) ranges over all subsets of
    $
    \{(j,k)\mid 1\le j\le t,\; 1\le k\le e_j-1\}.
    $
    We must then exclude those solutions for which \(q_{j,k}=q_{j,k'}\) for some \(k\neq k'\). The result follows.

\end{proof}

\begin{theorem}\label{thm: calculation of n(sigma)}
Let $\sigma$ be a maximal-dimensional cone of $W_\Gamma(\Sigma(X))$. Then the number of fine maps $n(\sigma)$ is given by the following formulas.

\begin{enumerate}[label=(\arabic*)]
    \item If $H_0 = H_1$, then
    \[
        n(\sigma)=1.
    \]

    \item Suppose the cycle is a self-loop.
    \begin{enumerate}
        \item If $t=1$, $e_1=2$, and $m_{r_0+1}(u_{1,1})$ is even (see Figure \ref{fig: cases where aut exists}), then
        \[
            n(\sigma)
            =\frac{i(P(\sigma))-1}{2}+1
            =\frac{m_{r_0+1}(u_{1,1})-2}{2}+1.
        \]

        \item Otherwise,
        \[
            n(\sigma)=\frac{i(P(\sigma))}{2}.
        \]
    \end{enumerate}

    \item Suppose the cycle consists of two vertices joined by two parallel edges with the same direction vector (possibly together with additional $2$-valent vertices along these edges).
    \begin{enumerate}
        \item If $t=1$, $e_1=2$, and $e_2=1$ (or symmetrically $e_1=1$ and $e_2=2$), and $m_{r_0+1}(u_{1,1})$ is even (see Figure \ref{fig: cases where aut exists}), then
        \[
            n(\sigma)
            =\frac{i(P(\sigma))-1}{2}+1 
            =\frac{m_{r_0+1}(u_{1,1})-2}{2}+1.
        \]

        \item If $t=2$, $e_1=e_2=2$, and both $m_{r_0+1}(u_{1,1}) + m_{r_0+1}(u_{2,1})$ and $m_{r_0+2}(u_{1,1})+ m_{r_0+1}(u_{2,1})$ are even (see Figure \ref{fig: cases where aut exists}), then
        \[
            n(\sigma)
            =\frac{i(P(\sigma))-1}{2}+1.
        \]

        \item Otherwise,
        \[
            n(\sigma)=\frac{i(P(\sigma))}{2}.
        \]
    \end{enumerate}

    \item In all remaining cases,
    \[
        n(\sigma)=i(P(\sigma)).
    \]
\end{enumerate}
\end{theorem}

\begin{figure}[h]
                \centering
                \begin{subfigure}[t]{0.3\textwidth}
                    \captionsetup{labelformat=empty}
            		\centering
            		\begin{tikzpicture}
            			\draw (-2, 0) -- (2,0); 
                        \draw[dashed] (0,0)  to[in=50,out=130,loop, style={min distance=12mm}] (0,0);
                        \node[forestgreen] at (0.6,0.2) {$u_{1,1}$};
                        \draw[forestgreen, ->] (0,0) -- (1,0);
                        \node at (0, -0.3) {$V_1$};
                        \fill (0,0) circle[radius=2pt];
            		\end{tikzpicture}
            		\caption{(2a) Cycle is a self loop and \\$t= 1, e_1 = 2$.}
            	\end{subfigure}	
                \begin{subfigure}[t]{0.3\textwidth}
                    \captionsetup{labelformat=empty}
                    \centering
                    \begin{tikzpicture}
            			\fill (2,0) circle[radius=2pt];
            			\fill (1.5,0) circle[radius=2pt];
            			\draw (0, 0.05) -- (1.5, 0.05);
            			\draw (0, -0.05) -- (1.5, -0.05);
            			\draw (1.5, 0) -- (2,0);
            			\draw (0, 0) -- (-1, 1); 
            			\draw (0, 0) -- (-1, -1); 
            			\draw (2, 0) -- (3, 1); 
            			\draw (2, 0) -- (3, -1);
                        \draw[forestgreen, ->] (0,0) -- (-0.8,0.8);
                        \node at (-0.4,0) {$V_1$};

                        \node[forestgreen] at (-0.2,0.6) {$u_{1,1}$};

                        \fill (0.6,0.05) circle[radius=1pt];
                        \fill (0.6,-0.05) circle[radius=1pt];
                        \fill (1.2,0.05) circle[radius=1pt];
                        \fill (1.2,-0.05) circle[radius=1pt];

                        \fill (0,0) circle[radius=2pt];
                    \end{tikzpicture}
                    \caption{ (3a) Cycle consists of 
                    two vertices of valence $\geq 3$, $t=1$ and $ e_1 = 1$.}
                  \end{subfigure}
                \begin{subfigure}[t]{0.3\textwidth}
                    \captionsetup{labelformat=empty}
                    \centering
                    \begin{tikzpicture}
                        
                        \draw (0, 0.05) -- (2, 0.05); 
                        \draw (0, -0.05) -- (2, -0.05); 
                        \draw (0, 0) -- (-1, 1); 
                        \draw (0, 0) -- (-1, -1); 
                        \draw (2, 0) -- (3, 0.8); 
                        \draw (2, 0) -- (3, -0.8);
                        
                        \draw[forestgreen, ->] (0,0) -- (-0.8,0.8);
                        \node[forestgreen] at (-0.2,0.6) {$u_{1,1}$};

                        \draw[forestgreen, ->] (2,0) -- (2.8,0.648);
                        \node[forestgreen] at (2.2,0.5) {$u_{1,2}$};
                        
                        \node at (-0.4,0) {$V_1$};
                        \node at (2.4,0) {$V_2$};
                        \fill (0.6,0.05) circle[radius=1pt];
                        \fill (0.6,-0.05) circle[radius=1pt];
                        \fill (1.4,0.05) circle[radius=1pt];
                        \fill (1.4,-0.05) circle[radius=1pt];

                        \fill (0,0) circle[radius=2pt];
                        \fill (2,0) circle[radius=2pt];
                    \end{tikzpicture}
                    \caption{(3b) Cycle consists of two vertices of valence $\geq 3$,  $t=2, e_1 = e_2 = 1$.}
                  \end{subfigure}
                    
                  \caption{Cases in Theorem~\ref{thm: calculation of n(sigma)}(2a), (3a) and (3b).}
                \label{fig: cases where aut exists}
            \end{figure}
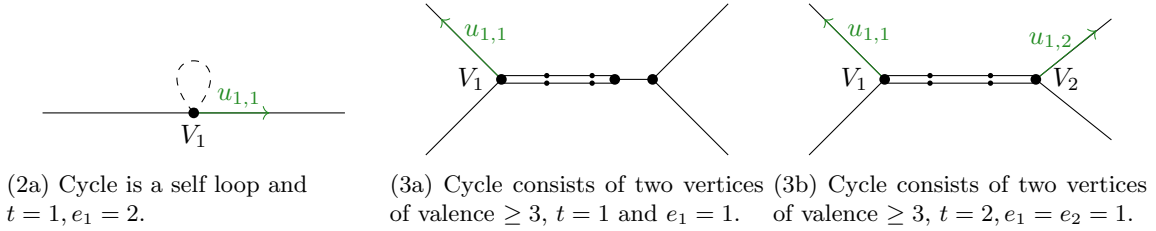

    \begin{proof}
        Case (1) is from Corollary \ref{cor: cyclespans}. 
        Suppose the cycle is a self-loop. 
        
        By Proposition \ref{prop: reduce to nbhd}, Lemma \ref{lemma: lifting_property_equation} and Lemma \ref{lemma: solution_interior_points} the number $n(\sigma)$ is equal to the number interior lattice points of $P(\sigma)$ giving rise to non-isomorphic cycles $(C_0,\mathsf{p}')$. 

        Note that $q_{j,k} \in \mathbb{C}^\times$ defines points in $C_{V_j}$ only after fixing the coordinates $x_j$ introduced immediately before Lemma \ref{lemma: lifting_property_equation}. Let $\underline{q}, \widetilde{\underline{q}} \in (\mathbb{C}^\times)^{r_1-r_0}$ be solutions of Equation \eqref{eqn: big product} that give rise to cycles $(C_0,\mathsf{p}')$ and $(\widetilde{C}_0,\widetilde{\mathsf{p}}')$, respectively, which are isomorphic. Any such isomorphism must map the component $C_{V_j} \subset C_0$  onto the component $\widetilde{C}_{V_j} \subset \widetilde{C}_0$. Each vertex $V_j$ carries a half-edge $\zeta_{j,e_j}$, and the corresponding marked point in $C_{V_j}$ (which is $1$ in the corrdinates $x_j$) must be mapped to the corresponding point in $\widetilde{C}_{V_j}$ (which is $1$ in the coordinates $\widetilde{x}_j$) . In addition, the vertex $V_j$ carries two half-edges $\xi_{j,1}$ and $\xi_{j,2}$ belonging to the cycle, and the isomorphism must map the associated points to one another, except in the case where $\sigma$ has automorphism group $\mu_2$ and is of one of the two types described in Corollary \ref{cor: automorphisms}. In that case, the induced isomorphism $C_V \to \widetilde{C}_V$ fixes the point $1$ and preserves the set $\{0,\infty\}$.

        Because there is a unique automorphism of $\P^1$ fixing $0,1$ and $\infty$, if $\sigma$ has trivial automorphism group, then the cycles associated to $\underline{q}$ and $\widetilde{\underline{q}}$ can be isomorphic only if $\underline{q}=\widetilde{\underline{q}} $. This prove that $n(\sigma)=i(P(\sigma))$ in this case, and thus case $(4)$ of the statement.

        We deal with cases $(2)$ and $(3)$ separately. 

        Suppose the cycle is a self-loop and let $V_1$ be the unique vertex in the cycle. Then $t \leq 1$. If $t=0$, we are in case (1). Suppose $t=1$. The only non-trivial automorphism of $\mathbb{P}^1$ fixing $1$ and preserving the set $\{0,\infty\}$ is given by $x_1 \mapsto 1/x_1$. The only additional fixed point of this involution is $-1$. Hence, unless $e_1=2$ and $-1$ is a solution of $q_{1,1}^{m_{r_0+1}(u_{1,1})}=1$ (equivalently, $m_{r_0+1}(u_{1,1})$ is even), the cycles associated to $\underline{q}$ and $\widetilde{\underline{q}}$ are isomorphic if and only if $q_{j,k}=\widetilde{q}_{j,k}$ for all $j,k$ or $q_{j,k}=\widetilde{q}^{-1}_{j,k}$. When $e_1=2$ and $m_{r_0+1}(u_{1,1})$ is even, we are in the situation described in case (2.a). In this case, the equation for $q_{1,1}$ is
        \[
            q_{1,1}^{m_{r_0+1}(u_{1,1})}=1.
        \]
        This equation has exactly $m_{r_0+1}(u_{1,1})$ solutions in $\mathbb{C}^\times$, of which precisely $m_{r_0+1}(u_{1,1})-2$ are different from $1$ and $-1$. The involution $x \mapsto 1/x$ permutes these remaining solutions, while $-1$ is fixed and provides an additional admissible solution. This proves $(2.a)$.

        Suppose the cycle consists of two vertices $V_1$ and $V_2$ joined by two parallel edges with the same direction vector, possibly together with additional $2$-valent vertices along these edges.

    The unique non-trivial automorphism $\iota$ of the pointed curve obtained from the subcurve $C_0$, equipped with markings at the points corresponding to $\zeta_{j,1}$ for $j=1,2$ (i.e. the points corresponding to $1$ in the coordinates $x_1$ and $x_2$, respectively), is given in coordinates by
    \[
    x_j \longmapsto \frac{1}{x_j}
    \]
    on each component $C_{V_j}$, together with swapping the chains of $\mathbb{P}^1$ between $C_{V_1}$ and $C_{V_2}$.

    Hence, unless we are in one between case (3.a) and (3.b), 
    the cycles associated to $\underline{q}$ and $\widetilde{\underline{q}}$ are isomorphic if and only if either
    \[
    q_{j,k} = \widetilde{q}_{j,k} \quad \text{for all } j,k,
    \qquad \text{or} \qquad
    q_{j,k} = \widetilde{q}_{j,k}^{-1} \quad \text{for all } j,k.
    \]
    In particular, in case (3.c) we obtain
    \[
        n(\sigma)=\frac{i(P(\sigma))}{2}.
    \]

    At this point, case (3.a) is identical to case (2.1). In case (3.b), there are two equations in \eqref{eqn: big product}, namely
    \[
        q_{1,1}^{m_{r_0+1}(u_{1,1})} \, q_{2,1}^{m_{r_0+1}(u_{2,1})}=1,
        \qquad \text{and} \qquad
        q_{1,1}^{m_{r_0+2}(u_{1,1})} \, q_{2,1}^{m_{r_0+2}(u_{2,1})}=1.
    \]
    We require $q_{1,1}, q_{2,1} \neq 1$, and for the involution $\iota$ to fix a solution vector $(q_{1,1},q_{2,1})$, we must have $q_{1,1}=q_{2,1}=-1$. This explains case (2.b).
\end{proof}

We finally prove Theorems \ref{thm: automorphisms} and \ref{thm: calculation of weights}.

\begin{proof}[Proof of Theorem \ref{thm: automorphisms}]
    The proof of Theorem \ref{thm: calculation of n(sigma)} explicitly lists all the basic fine maps $(C^{\mathsf{fine}}, \mathsf{p}^{i}, f_{\sigma}^{\mathsf{fine}}, \rho)$ over a maximal cone $\sigma$ of $\Sigma(\cW_{\Gamma}(X))$ and their automorphisms. These lift to automorphisms of the fs map $(C,\mathsf{p}, f_\sigma, \rho)$ because they preserve the divisors $D_{m,V}$, the line bundles $L_{m,V}$, and their sections $\sigma_{m,V}$ for all $m \in M_X$, and $V \in V(\tropC)$. By Remark~\ref{rmk: stronger_version_realizability}, this implies that they induce automorphisms of the fs map $(C,\mathsf{p}, f_\sigma, \rho)$.
\end{proof}

\begin{proof}[Proof of Theorem \ref{thm: calculation of weights}]
    The case where $H_0 = H_1$ is handled in Corollary \ref{cor: cyclespans}. 

    Fix a maximal cone $\sigma$ of $W_{\Gamma}(\Sigma(X))$. As in \eqref{eqn: def n(sigma)}, let $\sigma_{i,j}$ denote the maximal cones of $\Sigma(\cW_{\Gamma}(X))$ lying above $\sigma$, and let
$
(C_{i,j}/S_{i,j},\mathsf{p}^{i,j},f_{i,j},\rho_{i,j}) \in \cW_{\Gamma}(X)
$
be the corresponding points. Denote by
$
(C^\fine_{i,j}/S^\fine_{i,j},\mathsf{p}^{i,j},f^\fine_{i,j},\rho_{i,j})
$
their images in $\cW_{\Gamma}(X)^\fine$.

The proof of Theorem \ref{thm: calculation of n(sigma)} shows that the object
$
(C_{i,j}/S_{i,j},\mathsf{p}^{i,j},f_{i,j},\rho_{i,j})
$
can have a non-trivial automorphism group only in cases (2.a), (3.a), and (3.b). Moreover, these are precisely the objects corresponding to the exceptional solutions of \eqref{eqn: big product} that contribute $+1$ (rather than $+\frac12$) to the count
\[
n(\sigma)=\frac{i(P(\sigma))-1}{2}+1.
\]
In each of these cases, the automorphism group is isomorphic to $\mu_2$.

    \end{proof}

\begin{corollary} \label{cor: n' sigma depends only on sigma'}
    Suppose $\sigma_1$ and $\sigma_2$ are cones in $W_\Gamma(\Sigma(X))$ mapping to $\sigma'$ in $W_\Gamma((N_X)_{\mathbb{R}})$ under the map in Equation~\eqref{eqn:ws-forgetfan}. Then $\nn(\sigma_1) = \nn(\sigma_2)$. 
\end{corollary}

\begin{proof}
    Follows immediately from Theorem \ref{thm: calculation of weights}. 
\end{proof}

\subsection{Specializing to low dimension}

We apply Theorem \ref{thm: calculation of weights} in the cases where $\dim X = 1$ or $2$, proving Proposition \ref{thm:mult1} and Proposition \ref{thm:mult2} in the introduction.

\subsubsection{Dimension 1} \label{sec: dimension 1}

Suppose $X$ is a toric variety with $\dim X = 1$. We prove Proposition \ref{thm:mult1}. Since  $\sigma$ is a cone of maximum dimension in $W_{\Gamma}(\Sigma(X))$ the deficiency of $\sigma$, $\defic{\sigma}$, is at most $1$. We study the deficiency $0$ and $1$ cases separately.

\begin{enumerate}

\item \textbf{Deficiency 0.}  By Corollary \ref{cor: cyclespans}, we have $\nn(\sigma) = 1/|\Aut(\sigma)|$. 

\item \textbf{Deficiency 1.} Then the overvalence of $\sigma$, $\ov{\sigma}$, is at most $1$. We have three cases:
    \begin{enumerate}
        \item The neighborhood of the cycle spans. 
        
        In this case we must have $\ov{\sigma} = 1$ and the cycle has only one vertex $V$ which has has valency $4$. Let the half-edges adjacent to $V$ outside of the cycle be $\zeta_{1,1}, \zeta_{1,2}$ (see Notation \ref{notation: shape nbhd cycle}). Without loss of generality, we have that the direction vectors $u_{1,1}, - u_{1,2}$ satisfy $d = u_{1,1} = -u_{1,2}$ for some positive integer $d$. Then the region $P(\sigma)$ in Construction \ref{const: polytope} is the interval $[0,d]$ in $(N_X)_{\mathbb{R}}$, which contains $d-1$ interior lattice points. Thus, $\nn(\sigma) = (d-1)/ \Aut(\sigma)$. 
        \item The neighborhood of the cycle does not span, and $\ov{\sigma} = 1$.

        Then, because $\sigma$ parametrizes well-spaces tropical maps, the $4$ valent vertex must be positioned at the minimum distance from $\tropC_0$ where the tropical map becomes non-constant. By Corollary \ref{cor: cyclespans}, we have $\nn(\sigma)=1/\mathrm{Aut}(\sigma)$.

        \item The neighborhood of the cycle does not span, and $\ov{\sigma} = 0$.

        In this case, there must be $2$ vertices at minimum distance from $\tropC_0$ where the maps becomes non-constant. By Corollary \ref{cor: cyclespans}, we have $\nn(\sigma) = 1/\Aut(\sigma)$.
    \end{enumerate}
These cases are illustrated in Figure \ref{fig: dim1, defic1} in the introduction.

\end{enumerate}

\subsubsection{Dimension 2}

Suppose $X$ is a toric variety with $\dim X = 2$. We prove Proposition \ref{thm:mult2}.

\begin{enumerate}
    \item \textbf{Deficiency 0.} In this case the cycle contains at least three vertices and so the parametrized tropical curve has no non-trival automorphisms. By Corollary \ref{cor: cyclespans}, we have $n(\sigma) = \nn(\sigma) = 1$. 

\item \textbf{Deficiency 1.} We have the same three cases as in dimension $1$: 
\begin{enumerate}
    \item The neighborhood of the cycle spans. 

    In this case, $\mathrm{ov}(\sigma) \leq 1$, but because the neighborhood of the cycle spans, it must be $\ov{\sigma}=1$ and the unique $4$ valent vertex has to lie in the cycle. See Figure \ref{fig:dim2def1} for an illustration. Let $n_1$ be a basis for the line $L$ which contains the cycle, and complete to a basis $n_1, n_2$ of $N_X$. Let $u$ be the direction vector of any half-edge leaving the cycle. As in dimension 1, we find $$
    \nn(\sigma) = \frac{|m_2(u)|-1}{\Aut(\sigma)}= \frac{|\det(u, n_1)|-1}{\Aut(\sigma)}.
    $$
    
    \item The neighborhood of the cycle does not span, and $\ov{\sigma} = 1$.

    By Corollary \ref{cor: cyclespans}, we have $\nn(\sigma) = 1/\Aut(\sigma)$.
    
    \item The neighborhood of the cycle does not span, and $\ov{\sigma} = 0$.

    By Corollary \ref{cor: cyclespans}, we have $\nn(\sigma) = 1/\Aut(\sigma)$
\end{enumerate}

\item \textbf{Deficiency 2.} Here we must have $\ov{\sigma}\leq 2$. 
\begin{enumerate}
    \item The neighborhood of the cycle spans.

    Then, since $\ov{\sigma} \leq 2$ and every vertex in the cycle has necessarily valency at least $4$, the only two possibilities are that either:
\begin{enumerate}[label=(\roman*)]
    \item There is a $5$-valent vertex $V$ in the cycle with three non-contracted half-edges adjacent to it.

     Since $\ov{\sigma} \leq 2$, this is the only vertex $V$ with valency greater than $3$. Let $u_1, u_2$ be the direction vectors of any two distinct non-contracted half-edges adjacent to $V$. The region $P(\sigma)$ in Construction~\ref{const: polytope} is given as the parallelogram defined by $u_1, u_2$, with the line from $0$ to $u_1+u_2$ removed. The region $P(\sigma)$ has two triangles with the same number of lattice points in each. One of the triangles has vertices $0$, $u_1$, and $u_1+ u_2$. This is isometric to the dual polygon at the vertex $V$. Thus 
$$\nn(\sigma)  = \frac{i(P(\sigma))}{|\Aut(\sigma)|} = \frac{2 \times i(\text{dual polygon to $V$)}}{|\Aut(\sigma)|}.$$

    \item There are two $4$-valent vertices $V_1$, $V_2$ on the cycle each with two non-contracted half-edges adjacent to it. 

    Let $u_1$ (resp. $u_2$) be any non-contracted adjacent half-edge of $V_1$ (resp. $V_2$).  Theorem \ref{thm: calculation of weights} gives 
$$\nn(\sigma) = \frac{i(\text{parallelogram with vertices } 0,u_1, u_1+u_2, u_2)}{|\Aut(\sigma)|}$$
\end{enumerate}
 Examples for the cases 3.(a).(i) and 3.(a).(ii) are illustrated in Figure~\ref{fig:loopfivevalent}. 
    
    \item The neighborhood of the cycle is contained in a line but not contracted.

     Let $V_1$ be a vertex in the cycle which has valence at least $4$ and whose adjacent half-edges span a line $L$. The valence of $V_1$ must be exactly $4$-valent by Proposition \ref{prop: nbhd-cycle-span}. Since $1 \leq \ov{\sigma} \leq \defic{\sigma}=2$, we have two cases (examples illustrated in Figure~\ref{fig:dim2defic2case2}):
\begin{enumerate}[label=(\roman*)]
    \item If the overvalence of $\sigma$ is $2$, for $\sigma$ to be maximal, the line $L$ cannot contain the full image of the tropical curve. Moreover, because curves in $\sigma$ are well-spaced, the tropical curve must leave the line $L$ at a $4$-valent vertex $V_2$. 
    \item The overvalence of $\sigma$ is $1$, and either:
    \begin{itemize}
        \item the curve leaves the line $L$ at two vertices at the same minimal distance from $\tropC_0$; or
        \item the curve never leaves the line $L$. 
    \end{itemize}
    
\end{enumerate}

In all cases, as in the case $\dim(X)=1$ with deficiency $1$ Case (a), we deduce that 
\[
\nn(\sigma)=\frac{d-1}{|\Aut(\sigma)|}.
\]
Here $d \in \Z_{>0}$ is such that $\pm d \cdot u$ are the direction vectors of the non-contracted edges adjacent to $V_1$ and $u \in N_X$ primitive. 

    \item The neighborhood of the cycle is contracted.

    By Corollary \ref{cor: cyclespans}, we have $\nn(\sigma) = 1/|\Aut(\sigma)|$. Illustrated examples are in Figure \ref{fig:dim2defic2case3}.

\end{enumerate}

\end{enumerate}

\section{Saturation index, lattice index, loop multiplicity, and automorphisms} 
\label{sec: curves-in-Nx}

In this section, we prove Theorem \ref{thm: relation indices, loop aut}, which expresses the saturation index of a maximal cone $\sigma \in W_\Gamma(\Sigma(X))$ in terms of the lattice index $\mathrm{lat}(\sigma)$, the loop multiplicity $\mathrm{loop}(\sigma)$, and the automorphism groups of the tropical curves parametrized by $\sigma$.

\subsection{The inductive strategy}\label{sec: induction}

Let $\sigma$ be a maximal-dimensional cone in $\mathcal{W}_\Gamma(\Sigma(X))$. Under the morphism of cone complexes 
\[
W_\Gamma(\Sigma(X)) \rightarrow W_\Gamma((N_X)_{\mathbb{R}})
\]
defined in Equation \eqref{eqn:ws-forgetfan}, the cone $\sigma$ maps injectively onto a cone $\sigma' \subseteq W_\Gamma((N_X)_{\mathbb{R}})$. This induces an isomorphism of vector spaces $(N_\sigma)_{\mathbb{R}} \to (N_{\sigma'})_{\mathbb{R}}$, which is derived from the map of monoids described below.

Let $[h:(\mathsf{C},\mathsf{p}) \to \Sigma(X)]$ denote points in $\sigma$, and let $[h':(\mathsf{C}',\mathsf{p}') \to (N_X)_{\mathbb{R}}]$ be the corresponding points in $\sigma'$. Recall that the monoid associated to $\sigma$ was defined in Equation \eqref{eqn: monoid QWfinesigma} as:
\[
Q_{\W}^{\mathsf{fine}}(\sigma) = \frac{\prod_V \mathrm{Hom}(\vartheta_V \cap N_{\vartheta_V}, \mathbb{N}) \times \prod_e \mathbb{N}}{R} \times_{(\prod_e \mathbb{N})} \tau^\vee = \frac{\prod_V \mathrm{Hom}(\vartheta_V \cap N_{\vartheta_V}, \mathbb{N}) \times \tau^\vee}{\overline{R}},
\]
where $\overline{R}$ denotes the image of $R$ under the morphism 
\begin{equation*} 
\prod_V \mathrm{Hom}(\vartheta_V \cap N_{\vartheta_V}, \mathbb{Z}) \times \prod_q \mathbb{Z} \to \prod_V \mathrm{Hom}(\vartheta_V \cap N_{\vartheta_V}, \mathbb{Z}) \times (\tau^\vee)^{\text{gp}}.
\end{equation*}
Analogously, we define the monoid associated to $\sigma'$ as:
\begin{equation}\label{eqn: monoid QWfinesigma'}
Q_{\W}^{\mathsf{fine}}(\sigma') = \frac{\prod_V \mathrm{Hom}(N_X, \mathbb{N}) \times \prod_e \mathbb{N}}{R'} \times_{(\prod_e \mathbb{N})} \tau^\vee = \frac{\prod_V \mathrm{Hom}(N_X, \mathbb{N}) \times \tau^\vee}{\overline{R'}}
\end{equation}
where $R' \subseteq \left(\prod_V \mathrm{Hom}(N_X, \mathbb{N}) \times \prod_e \mathbb{N}\right)^{\gr}$ is the subgroup generated by the relations
        $$
        a_{\vec{e}}(m) = \left((\ldots, m, \ldots, -m, \ldots), (\ldots, u_{\vec{e}}(m), \ldots)\right),
        $$
        for $m \in M_X$ and $\vec{e}: V \to V'$ an oriented edge in $\mathsf{C}$.
There is a natural morphism \begin{equation}\label{eqn: map of fine monoids}
    Q_{\W}^{\mathsf{fine}}(\sigma') \to Q_{\W}^{\mathsf{fine}}(\sigma)
\end{equation}
determined by the following assignments:
\begin{itemize}
    \item For each vertex $V$, the copy of $\mathrm{Hom}(N_X,\N)$ in $Q_{\W}^{\mathsf{fine}}(\sigma')$ is mapped via the identity to the corresponding copy of $\mathrm{Hom}(N_X,\N)$ in $Q_{\mathcal{W}}^{\mathsf{fine}}(\sigma)$.
    \item any edge $e'$ in $\mathsf{C}'$ corresponds to a set of edges $\{e_1, \ldots, e_t\}$ in $\mathsf{C}$, and the factor of $\mathbb{Z}$ corresponding to $e'$ is mapped to $\prod_{i=1}^t \mathbb{Z}$ via the diagonal map.
\end{itemize}
The map $N_\sigma \to N_{\sigma'}$ is precisely the dual to the map $Q_{\W}^{\mathsf{fine}}(\sigma') \to Q_{\W}^{\mathsf{fine}}(\sigma)$ in Equation \eqref{eqn: map of fine monoids}. In particular, the lattice index of $\sigma$ is equal to the index of the map
$$
Q_{\W}^{\mathsf{fine}}(\sigma') /Q_\W^{\fine}(\sigma')_{\tors} \to Q_{\W}^{\mathsf{fine}}(\sigma)/Q_\W^{\fine}(\sigma)_{\tors}
$$
between the torsion free parts.

The proof of Theorem \ref{thm: relation indices, loop aut} is obtained in several steps. 

We begin with the tropical curve
$
\tropC^{(0)}=(G^{(0)},p,\ell),
$
where $G^{(0)}$ is obtained from the cycle of $\mathsf{C}$ by removing all constrained vertices and identifying the two edges adjacent to each such vertex. The markings $p_i$ on the cycle of $\mathsf{C}$ are retained in $\tropC^{(0)}$, and the edge-length function $\ell$ is induced from that of $\mathsf{C}$, with the length of an edge created by such an identification equal to the sum of the lengths of the original edges. Note that $\tropC^{(0)}$ need not coincide with $\tropC_0$ when the cycle of $\tropC$ contains constrained vertices. This distinction motivates the use of a superscript rather than a subscript in the notation. Let $h^{(0)}: \tropC^{(0)} \to \Sigma(X)$ be the restriction of $h$ to $\tropC^{(0)}$. We regard $h^{(0)}$ as a \emph{partial parametrized tropical curve} according to the following definition.

\begin{definition}
    A partial parametrized tropical curve in $(N_X)_{\mathbb{R}}$ is a map $h: \tropC \to (N_X)_\mathbb{R}$ where $\tropC$ is a graph with vartices and $h$ is a piecewise linear function with integer slopes. 

    A partial parametrized tropical curve in $\Sigma(X)$ is a partially parametrized tropical curve in $(N_X)_{\mathbb{R}}$ such that every edge of $\tropC$, with the possible exception of edges belonging to the cycle, is mapped into a cone of $\Sigma(X)$.

    The combinatorial type of a partial parametrized tropical curve in $(N_X)_{\mathbb{R}}$ consists of the following data: 
    \begin{itemize}
        \item the underlying finite graph, genus, and markings $(G,g,p)$ of $\tropC$; and
        \item for each half-edge $\zeta$ of $G$, the direction vector $v_{\zeta}$. 
    \end{itemize}

    The combinatorial type of a parametrized tropical curve in $\Sigma(X)$ is the information above, together with the data of the cone $\vartheta_V$ of $\Sigma(X)$ that contains $h(V)$ in its relative interior for each vertex $V$ of $\tropC$. 
\end{definition}

For example, $h^{(0)}$ can be regarded both as a partial parametrized tropical curve in $(N_X)_{\mathbb{R}}$ or in $\Sigma(X)$.
Starting from the type $\sigma^{(0)}$ (resp. $(\sigma')^{(0)}$) of $h^{(0)}$ regarded as a partial tropical curve in $\Sigma(X)$ (resp. $(N_X)_\R$), we consider a sequence of types $\sigma^{(k)}$ (resp. $(\sigma')^{(k)}$) of partial parametrized tropical curves $h^{(k)}: \tropC^{(k)} \to \Sigma(X)$ in $\Sigma(X)$ (resp. $(h')^{(k)}: (\tropC')^{(k)} \to (N_X)_\mathbb{R}$ in $(N_X)_\mathbb{R}$). Recall the definition of the disjoint subsets $\cV_i$ of vertices of $\tropC$ which are required to have the same distance from $\tropC_0$ (Definition \ref{def: sets V}). For every $k$, the tropical curve $\tropC^{(k)}$ (resp. $(\mathsf{C}')^{(k)}$) is a subcurve of $\tropC$, obtained by possibly removing two-valent vertices on the cycle and identifying the two adjacent edges. We require that if two vertices $V_{i,j}$ and $V_{i,k}$ in $\mathcal{V}_i$ both appear as vertices of $\tropC^{(k)}$ (resp. $(\tropC')^{(k)}$), then they have the same distance from the cycle in $\tropC^{(k)}$ (resp. $(\tropC')^{(k)}$).

For every $k$ we construct monoids $Q^\fine_\W(\sigma^{(k)})$ and $Q^\fine_\W((\sigma')^{(k)})$ and a map of monoids 
\begin{equation}\label{eqn: kth map of fine monoids}
\Phi_k: Q^\fine_\W((\sigma')^{(k)}) \to Q^\fine_\W(\sigma^{(k)}).
\end{equation}
The final index $k$ corresponds precisely to the maps $h: \tropC \to \Sigma(X)$ (resp. $h': \tropC' \to (N_X)_\mathbb{R}$) and their associated monoids and map is the map \eqref{eqn: map of fine monoids}. 

Let $\mathrm{lat}(\sigma^{(k)})$ denote the lattice index of the torsion-free part of the map \eqref{eqn: kth map of fine monoids}, and let $\mathrm{sat}(\sigma^{(k)})$ be the cardinality of the torsion subgroup of $Q^\fine_\W(\sigma^{(k)})^{\gr}$. We will show the following  statements:

\begin{enumerate}[label=(\arabic*)]
    \item For $k=0$, one has $\mathrm{sat}(\sigma^{(k)})=\mathrm{loop}(\sigma)$ (see Proposition \ref{lem:loop});
    \item for all $k \geq 1$ the cardinality of the torsion part of $Q_\W(\sigma^{(k-1)})^{\gr}$ divides that of the torsion part of $Q^\fine_\W(\sigma^{(k)})^{\gr}$ and one has 
    $$  \frac{|(Q^\fine_\W(\sigma^{(k)})^{\gr})_{\tors}|}{|(Q^\fine_\W(\sigma^{(k-1)})^{\gr})_{\tors}|} = \mathrm{lat}(\sigma^{(k-1)})  \cdot \frac{\prod_{V \in \tropC^{(k)} \smallsetminus \tropC^{(k-1)} \text{ constrained}} \mu_V}{\mathrm{lat}(\sigma^{(k)})}.
    $$ 
    (see Propositions \ref{prop: cycle induction} and \ref{prop: inductive step}).
\end{enumerate}

Theorem \ref{thm: relation indices, loop aut} will then follow immediately from $(1)$ and $(2)$ above.

\subsection{Preliminary algebraic results}

In our computations, we will make use of the following preliminary index calculations. Even though more general versions of the statements are true, we state them only in the generality that we need. 

\begin{lemma} \label{lem:gcd}
Let $M = \mathbb{N}^t$ and let 
$
A \colon \mathbb{Z}^s \longrightarrow \mathbb{Z}^t
$
be a homomorphism of rank $\rho$. Let $R \subset M$ be the submonoid given by
$
R = \operatorname{im}(A) \cap M.
$
Then
\begin{align*}
\left| \bigl((M/R)^{\mathrm{gp}}\bigr)_{\mathrm{tors}} \right|
& = [\operatorname{im}(A): \operatorname{span}_{\mathbb{R}}(\operatorname{im}(A)) \cap \mathbb{Z}^t] \\
&= \gcd \text{ of the non-zero } \rho \times \rho \text{ minors of } A.
\end{align*}

\end{lemma}

\begin{proof}
    This follows from the exact sequence
    \begin{center}
        % https://tikzcd.yichuanshen.de/#N4Igdg9gJgpgziAXAbVABwnAlgFyxMJZABgBpiBdUkANwEMAbAVxiRAB12BbOnACwBGA4AC0AvgD0EY0uky58hFAEZyVWoxZtOPfkNGScIGXOx4CRAExrq9Zq0QgAFAFkA9ACUAlBOCcA5mhixrIgGGaKRADMNhr2bMTG6jBQ-vBEoABmAE4QXEiqIDgQSJYmIDl5SGRFJYiFdlqOAIIhWbn5iNa1SFFiFGJAA
\begin{tikzcd}
\mathbb{Z}^s \arrow[r, "A"] & \mathbb{Z}^t \arrow[r] & (M/R)^{\gr} \arrow[r] & 0.
\end{tikzcd}
    \end{center}
    and the existence of a Smith normal form for $A$. 
\end{proof}

In particular, $\left| \bigl((M/R)^{\mathrm{gp}}\bigr)_{\mathrm{tors}} \right|$ is invariant under taking transpose. 

\begin{lemma} \label{lem: unconstrained}
    Let $M$ is an abelian group and fix any homomorphism $\alpha: \mathbb{Z}^r \rightarrow M.$
    Consider 
    $$N = (M \times \mathbb{Z}^r)/R$$
    where the relation $R$ is generated by the image of 
    $$\mathbb{Z}^r \rightarrow M \times \mathbb{Z}^r, m \mapsto (\alpha(m),m).$$
    Then $N \simeq M$. 
\end{lemma}

\begin{proof}
    The isomorphism is given by 
    \begin{align*}
        N = M \times \mathbb{Z}^r/R & \rightarrow M \\
        (a, b) &\mapsto a - \alpha(b)
    \end{align*}
    with inverse $a \mapsto (a,0)$.
\end{proof}

\begin{lemma} \label{lem: constrained}
    Let $M = \mathbb{Z}^k \times M_{\mathrm{tors}}$ be an abelian group with torsion subgroup $M_{\mathrm{tors}}$. Fix a positive integer $d \in \mathbb{Z}_{>0}$, a homomorphism $\alpha: \mathbb{Z}^r \to M$, and integers $a_1, \ldots, a_{r-1}$. Consider the quotient group
\begin{equation*}
    N = (M \times \mathbb{Z}^{r-1} \times \mathbb{Z}) / R
\end{equation*}
where the relation $R$ is the subgroup generated by the image of the map
\begin{align*}
    \mathbb{Z}^r & \to M \times \mathbb{Z}^{r-1} \times \mathbb{Z} \\
    (m_1, \ldots, m_r) & \mapsto \left(\alpha(m), (m_1, \ldots, m_{r-1}), \sum_{i=1}^{r-1} a_i m_i + d m_r\right).
\end{align*}
Then,  $|M_{\mathrm{tors}} |$ divides $ |N_{\mathrm{tors}}|$. Let $g \in \Z_{\geq 1}$ be the result of this division. It follows that $g$ divides $d$ and that the induced homomorphism
\begin{equation*}
    \Phi: M/M_{\mathrm{tors}} \to N/N_{\mathrm{tors}}
\end{equation*}
is an isomorphism after tensoring with $\mathbb{R}$ and has lattice index equal to $d/g$.
    
\end{lemma}

\begin{proof}
    We have 
    $$N \simeq \frac{(M \times \mathbb{Z}^{r-1} \times \mathbb{Z})/(\alpha(m_1, \ldots, m_{r-1}, 0), (m_1, \ldots, m_{r-1}), \sum_{i=1}^{r-1} a_i m_i)}{(\alpha(0, \ldots, 0, m_r), (0, \ldots, 0), dm_r)}.$$
    Applying Lemma \ref{lem: unconstrained} to $N' := (M \times \mathbb{Z}^{r-1} \times \mathbb{Z})/(\alpha(m_1, \ldots, m_{r-1}, 0), (m_1, \ldots, m_{r-1}), \sum_{i=1}^{r-1} a_i m_i)$, we obtain $N' \simeq M \times \mathbb{Z}$.

    Write $M_{\tors} \simeq \Z/(c_1) \times \ldots \times \Z/(c_t)$ for some $c_1,\ldots,c_t \in \Z$. Let also $(b_1,\ldots,b_k) \in \Z^k$ be the image of  $\alpha(0,\ldots, 1)$ under the quotient map 
    $M = \mathbb{Z}^k \times M_{\tors} \rightarrow \mathbb{Z}^k$.
    Then, $N$ is naturally a quotient of $\Z^{k+t+1}$ by the image of the matrix:
    \[
    \left[
    \begin{array}{c|ccc}
    b_1     & 0      &        &        \\
    b_2     &        & \ddots &        \\
    \vdots  &        &        &        \\
    b_k     &        &        & 0      \\
    \hline
            & c_1    &        &        \\
    \star   &        & \ddots &        \\
            &        &        & c_t    \\
    \hline
    d       & 0      & \cdots & 0
    \end{array}
    \right]
    \]
    where $\widetilde{b}$ is the gcd of the $b_i$. By Lemma \ref{lem:gcd}, we have that for $g = \operatorname{gcd}(\tilde{b}, d)$, $$|N_{\tors}| = g  c_1 \ldots c_t = g|M_{\tors}|.$$
    Moreover 
    $$N/M_{\tors} \simeq \frac{(\mathbb{Z}^{k-1} \times \mathbb{Z}) \times \mathbb{Z}}{(0, \tilde{b}, d)} \simeq \mathbb{Z}^{k-1} \times \frac{\mathbb{Z}^2}{(\tilde{b},d)} \simeq \mathbb{Z}^{k-1} \times \bigg(\mathbb{Z} \times \frac{\mathbb{Z}}{g}\bigg).$$
    In particular, $N/N_{\tors} \simeq \Z^k$. Let $b' = \tilde{b}/g, d' = d/g$ and let $B,D \in \mathbb{Z}$ such that $b'B + d'D = 1$. The last isomorphism is given by 
    $$(x, (y,z)) \mapsto (x, d'y-b'z, By+Dz).$$
    Thus $\Phi$ can be expressed, with respect to suitable basis, as
    $$(x_1, \ldots, x_{k-1}, y) \mapsto (x_1, \ldots, x_{k-1}, d'y)$$
     and the result follows. 
\end{proof}

\subsection{Proof of Theorem \ref{thm: relation indices, loop aut}}

In this section we present the proof of Theorem \ref{thm: relation indices, loop aut}.

\begin{construction} \label{const: partial-sequence}
  Order the vertices of $\tropC$ according to the following criteria:
\begin{enumerate}
    \item[$\bullet$] The first $E_0$ vertices are precisely the non-constrained vertices contained in the cycle $\tropC_0$.
    \item[$\bullet$] The second $E_1$ vertices are precisely the constrained vertices contained in the cycle $\tropC_0$.
    \item[$\bullet$] For indices $i > j > E_0+E_1$, the path from $V_j$ to $\tropC_0$ does not pass through $V_i$.
\end{enumerate}

We consider a sequence of combinatorial types $\sigma^{(k)}$ of partially parametrized tropical curves in $\Sigma(X)$ (resp. $(\sigma')^{(k)}$ in $(N_X)_\mathbb{R}$) defined as follows. Let $\sigma^{(0)}$ and $(\sigma')^{(0)}$ be the combinatorial type of the partial curve consisting solely of the cycle $\tropC^{(0)}$ (with the two valent vertices removed) and the restriction $h^{(0)} = h|_{\tropC^{(0)}}$. 

For $0 < k \leq E_1$, the type $(\sigma')^{(k)}$ remains unchanged and is equal to $(\sigma')^{(0)}$, while $\tropC^{(k)}$ is obtained from $\tropC^{(k-1)}$ by adding the two-valent vertex $V_k$ on the cycle. The map $h^{(k)}$ is again the restriction of $h$ to $\tropC^{(k)}$.

For $k>E_1$, we build $\tropC^{(k)}$ from $\tropC^{(k-1)}$ by adding the vertex $V_k$ and the unique edge $e_k$ connecting $V_k$ to a vertex $V_{i_k} \in \tropC^{(k-1)}$. If $V_k$ is a constrained vertex, no vertex is added to $(\tropC')^{(k-1)}$; otherwise, $V_k$ is also added to $(\tropC')^{(k-1)}$ together with the unique edge $e_k'$ connecting $V_k$ to a vertex $V_{i_k}' \in (\tropC')^{(k-1)}$. Note that, in general, $V_{i_k} \neq V_{i_k}'$ and that $e_k$ is ``contained'' in $e_k'$.

For every $k > E_1$, we require that all vertices in $\mathcal{V}_i$ (see Definition~\ref{def: cV}) that are contained in $\tropC^{(k)}$ (resp.\ $(\tropC')^{(k)}$) are equidistant from the cycle $\tropC_0$.

In particular, for each $k$, we obtain a cone
\[
\tau_k \subseteq \prod_{e \in E(\tropC^{(k)})} \mathbb{N},
\]
defined by imposing the condition that all vertices in $\mathcal{V}_i \cap V(\tropC^{(k)})$ have the same distance from $\tropC_0$.

Note that vertices in $\mathcal{V}_i$ are not constrained; consequently, the corresponding cone $\tau'_k$ for $(\tropC')^{(k)}$ coincides with $\tau_k$.

The maps $h^{(k)}$ and $(h')^{(k)}$ are defined as the restrictions of $h$ and $h'$ to $\tropC^{(k)}$ and $(\tropC')^{(k)}$, respectively. The combinatorial types $\sigma^{(k)}$ and $(\sigma')^{(k)}$ are the types of $h^{(k)}: \tropC^{(k)} \to \Sigma(X)$ and $(h')^{(k)}: (\tropC')^{(k)} \to (N_X)_\mathbb{R}$, respectively.

Finally, the monoids $Q^\fine_\W(\sigma^{(k)})$ and $Q^\fine_\W((\sigma')^{(k)})$ are defined analogously to those in Equations \eqref{eqn: monoid QWfinesigma} and \eqref{eqn: monoid QWfinesigma'} , with the following modifications: the products are taken over the vertices of $\tropC^{(k)}$ (resp.\ $(\tropC')^{(k)}$), the cone $\tau$ is replaced by $\tau_k$, and the set of relations is generated by the same conditions, now varying over the edges of $\tropC^{(k)}$ (resp.\ $(\tropC')^{(k)}$). For every $k$ there is an natural map

$$
Q^\fine_\W((\sigma')^{(k)}) \to Q^\fine_\W(\sigma^{(k)})
$$
defined as in Equation \eqref{eqn: map of fine monoids}.

\end{construction}

 As explained in \S\ref{sec: induction} the proof is by induction on $k$. The next proposition proves statement $(1)$ at the end of \S\ref{sec: induction}.

\begin{proposition}\label{lem:loop}
Suppose $\gamma$ is the combinatorial type of a partially parametrized tropical curve in $\Sigma(X)$ consisting solely of a cycle and containing no constrained vertices. Then, 
\begin{equation*}
    \mathrm{sat}(\gamma) = \mathrm{loop}(\sigma).
\end{equation*}
\end{proposition}

\begin{proof}
Let $V_1, \ldots, V_{E_0}$ be the vertices and $e_1, \ldots, e_{E_0}$ be the edges of the cycle, oriented as $\vec{e}_i: V_i \to V_{i+1}$ (with the cyclic convention $V_{E_0+1} = V_1$). 

The relations $R$ in $Q^{\mathrm{fine}}_\W(\gamma)$ are given by the image of a map which we represent in matrix form. Since each vertex $V_i$ maps to the interior of a maximal-dimensional cone in $\Sigma(X)$, the dual monoid $\mathrm{Hom}(\vartheta_{V_i} \cap N_{\vartheta_{V_i}}, \mathbb{N})$ reduces to $\mathrm{Hom}(N_X, \mathbb{N}) = N_X^\vee$ for every $i = 1, \ldots, E_0$. Upon identifying 
\begin{equation*}
    M_X = \mathrm{Hom}(N_X, \mathbb{Z}) = N_X^*,
\end{equation*}
we can express these relations as the image of the following block matrix:
\[
\begin{blockarray}{cccccc}
& e_1 & e_2 & \ldots & e_{E_0} \\
\begin{block}{c[ccccc]}
  V_1    & -\Id_{M_X} & 0          & \ldots & \Id_{M_X} \\
  V_2    & \Id_{M_X}  & -\Id_{M_X} & \ldots & 0 \\
  V_3    & 0          & \Id_{M_X}  & \ldots & 0 \\
  \vdots &            &            & \vdots &   \\
  V_{E_0}& 0          & 0          & \ldots & -\Id_{M_X} \\
  e_1    & u_1        & 0          & \ldots & 0 \\
  \vdots &            &            & \vdots &   \\
  e_{E_0}& 0          & 0          & \ldots & u_{E_0} \\
\end{block}
\end{blockarray}
\]
where we set $u_i = u_{\vec{e}_i} \in N_X$ for visual clarity. Here, each block $(V_i, e_j)$ corresponds to an $r \times r$ submatrix, whereas each block $(e_i, e_j)$ corresponds to a $1 \times r$ submatrix, where $r = \operatorname{rank}(M_X)$. 

By making appropriate choices of bases, we can perform row and column operations (corresponding to a change of basis of the target and source, respectively) to reduce this matrix. We begin with column operations: successively add column block $e_1$ to $e_2$, then $e_2$ to $e_3$, and proceed inductively up to the final column block. This yields:
\[
\begin{blockarray}{ccccccc}
& e_1 & e_2 & \ldots & e_{E_0-1} & e_{E_0} \\
\begin{block}{c[cccccc]}
  V_1    & -\Id_r & -\Id_r & \ldots & -\Id_r & 0 \\
  V_2    & \Id_r  & 0      & \ldots & 0      & 0 \\
  V_3    & 0      & \Id_r  & \ldots & 0      & 0 \\
  \vdots &        &        & \vdots &        &   \\
  V_{E_0}& 0      & 0      & \ldots & \Id_r  & 0 \\
  e_1    & u_1    & u_1    & \ldots & u_1    & u_1 \\
  e_2    & 0      & u_2    & \ldots & u_2    & u_2 \\
  \vdots &        &        & \vdots &        &   \\
  e_{E_0}& 0      & 0      & \ldots & 0      & u_{E_0} \\
\end{block}
\end{blockarray}
\]
We next use the identity blocks in rows $V_2$ through $V_{E_0}$ to clear the remaining non-zero blocks in rows $V_1$ and columns $e_1$ through $e_{E_0-1}$. Taking the transpose of the resulting simplified matrix, we isolate the essential non-zero blocks as:
\begin{equation*}
B = \begin{bmatrix}
    0 & \Id_{r(E_0-1)} & 0 & \ldots & 0 \\
    0 & 0              & u_1^T & \ldots & u_{E_0}^T
\end{bmatrix}.
\end{equation*}
For this matrix $B$, the index $[\operatorname{Span}_{\mathbb{R}}(\operatorname{Im}(B)) \cap \mathbb{Z}^t : \operatorname{Im}(B)]$ is precisely to the loop multiplicity of the cycle. An application of Lemma \ref{lem:gcd} to this block form yields the desired result.
\end{proof}

\begin{proposition}\label{prop: cycle induction}
    For all $0<k \leq E_1$, the cardinality of the torsion subgroup of $Q^\fine_\W(\sigma^{(k-1)})^{\gr}$ divides that of the torsion subgroup of $Q^\fine_\W(\sigma^{(k)})^{\gr}$, and one has 
    $$  \frac{|(Q^\fine_\W(\sigma^{(k)})^{\gr})_{\tors}|}{|(Q^\fine_\W(\sigma^{(k-1)})^{\gr})_{\tors}|} = \mathrm{lat}(\sigma^{(k-1)})  \cdot \frac{ d_{V_k}}{\mathrm{lat}(\sigma^{(k)})}
    $$  
\end{proposition}

\begin{proof}
    Let $\vartheta:=\vartheta_{V_k}$ be the codimension $1$ subcone of $\Sigma(X)$ which contains $h^{(k)}(V_k)$. Let  $\vec{e}_1 \colon V_{i_1} \rightarrow V_k$ and $\vec{e}_2 \colon V_k \rightarrow V_{i_2}$ be the two directed edges adjacent to $V_k$. Let $\vartheta_1:= \vartheta_{V_{i_1}}, \vartheta_2=\vartheta_{V_{i_2}} $ be the cones of $\Sigma(X)$ containing $V_{i_1}, V_{i_2}$ (which may be of maximum dimension, or codimension $1$). Let $\iota_{\theta_i} :  M_X  \rightarrow M_{\vartheta_i}$ be dual to the associated inclusion of monoids $N_{\vartheta_i} \rightarrow N_X$. Let $\vec{e}_0 \colon V_{i_1} \rightarrow V_{i_2}$ be the directed edge in $\sigma^{(k-1)}$. We have 
    $$Q_{\W}^{\fine}(\sigma^{(k-1)})^{\gr} = \frac{(\prod_{V \neq V_{i_1}, V_{i_2}} M_{\vartheta_V} \times \prod_{e \neq e_0} \mathbb{Z}) \times M_{\vartheta_1} \times M_{\vartheta_2} \times  \mathbb{Z}_{\vec{e}_0}}{R' + ((0), -\iota_1(m), \iota_2(m), u_{\vec{e}_0}(m))},$$
    where we denoted by $R'$ the set of relations \eqref{eqn: relations from tropical curve} coming from edges in $\sigma^{(k-1)}$ different from $e_0$ and $m$ varies in $M_X$. Meanwhile  
    $$Q_{\W}^{\fine}(\sigma^{(k)})^{\gr} = \frac{(\prod_{V \neq V_{i_1}, V_{i_2}} M_{\vartheta_V} \times \prod_{e \neq e_0} \mathbb{Z}) \times M_{\vartheta_1} \times M_{\vartheta_2}  \times M_{\vartheta} \times \mathbb{Z}_{\vec{e}_1} \times \mathbb{Z}_{\vec{e}_2} }{R' + (a_{\vec{e}_1}(m)) + (a_{\vec{e}_2}(m))}.$$
    where for $m \in M_X$
    \begin{align*}
        &a_{\vec{e}_1}(m) = ((0), -\iota_1(m), 0 , \iota(m),  u_{\vec{e}_1}(m), 0)\\
        &a_{\vec{e}_2}(m) = ((0), 0, \iota_2(m) , -\iota(m), 0, u_{\vec{e}_2}(m))
    \end{align*}
    Note that $u_{\vec{e}_0} = u_{\vec{e}_1} = u_{\vec{e}_2}$. The two relations $a_{\vec{e}_1}(m)$ and $a_{\vec{e}_2}(m)$ can be replaced with 
    \begin{align*}
        &((0), -\iota_1(m), \iota_2(m),  0,  u_{\vec{e}_0}(m), u_{\vec{e}_0}(m))  \quad \text{and} \quad
        ((0), 0, \iota_2(m) , -\iota(m), 0, u_{\vec{e}_0}(m)). 
    \end{align*}
    Consider a change of basis $\mathbb{Z}_{\vec{e}_1} \times \mathbb{Z}_{\vec{e}_2} \rightarrow \mathbb{Z} \times \mathbb{Z}$ given by $(n,m) \mapsto (n, n-m)$. Under this change of basis, the two relations become 
    \begin{align*}
            &((0), -\iota_1(m), \iota_2(m),  0,  u_{\vec{e}_0}(m), 0)  \quad \text{and} \quad
            ((0), 0, \iota_2(m) , -\iota(m), 0, u_{\vec{e}_0}(m)). 
    \end{align*}
    The first relation is one of the defining relations of $Q_{\W}^{\fine}(\sigma^{(k-1)})^{\gr}$. Thus, we have an isomorphism
\[
Q_{\W}^{\fine}(\sigma^{(k)})^{\gr}
\simeq
\frac{
Q_{\W}^{\fine}(\sigma^{(k-1)})^{\gr}
\times M_{\vartheta}
\times \mathbb{Z}
}{
(\alpha(m),-\iota(m),u_{\vec e_0}(m))
}.
\]
Here
\[
\alpha:M_X\longrightarrow Q_\W^\fine(\sigma^{(k-1)})^{\gr}
\]
sends $m\in M_X$ to the image of
\[
((0),0,\iota_2(m),0)
\in
\Bigl(
\prod_{V\neq V_{i_1},V_{i_2}} M_V
\times
\prod_{e\neq e_0}\mathbb Z
\Bigr)
\times M_{\vartheta_1}
\times M_{\vartheta_2}
\times \mathbb Z_{\vec e_0}
\]
under the projection to $Q_\W^\fine(\sigma^{(k-1)})^{\gr}$.

Choose bases of $M_X$ and $M_{\vartheta}$ such that
\[
\iota_\vartheta:M_X\to M_\vartheta,
\qquad
(m_1,\ldots,m_r)\longmapsto (m_1,\ldots,m_{r-1}).
\]
Then, by Lemma~\ref{lem: constrained}, the natural map
\[
Q_\W^\fine(\sigma^{(k-1)})^\gr
\longrightarrow
Q_{\W}^{\fine}(\sigma^{(k)})^{\gr}
\simeq
\frac{
Q_{\W}^{\fine}(\sigma^{(k-1)})^{\gr}
\times M_{\vartheta}
\times \mathbb{Z}
}{
\bigl(
\alpha(m_1,\ldots,m_r),
-(m_1,\ldots,m_{r-1}),
\sum_{i=1}^{r-1} a_i m_i + d_{V_k}m_r
\bigr)
}
\]
becomes an isomorphism after tensoring with $\mathbb R$, and
\[
\frac{\bigl|(Q_\W^\fine(\sigma^{(k)})^{\gr})_{\tors}\bigr|}
     {\bigl|(Q_\W^\fine(\sigma^{(k-1)})^{\gr})_{\tors}\bigr|}
=
\frac{d_{V_k}}
{\left[
Q_\W^\fine(\sigma^{(k)})^{\gr}/(Q_\W^\fine(\sigma^{(k)})^{\gr})_{\tors}
:
\operatorname{im}\!\Bigl(
Q_\W^\fine(\sigma^{(k-1)})^{\gr}/(Q_\W^\fine(\sigma^{(k-1)})^{\gr})_{\tors}
\Bigr)
\right]}.
\]
The conclusion now follows from the identity
$$
\left[
Q_\W^\fine(\sigma^{(k)})^{\gr}/(Q_\W^\fine(\sigma^{(k)})^{\gr})_{\tors}
:
\operatorname{im}\!\Bigl(
Q_\W^\fine(\sigma^{(k-1)})^{\gr}/(Q_\W^\fine(\sigma^{(k-1)})^{\gr})_{\tors}
\Bigr)
\right]= \frac{ \mathrm{lat}(\sigma^{(k)})}{\mathrm{lat}(\sigma^{(k-1)})}
$$
    
\end{proof}

\begin{proposition}\label{prop: inductive step}
    For all $ k>E_1$, the cardinality of the torsion subgroup of $Q^\fine_\W(\sigma^{(k-1)})^{\gr}$ divides that of the torsion subgroup of $Q^\fine_\W(\sigma^{(k)})^{\gr}$, and one has 
    $$  \frac{|(Q^\fine_\W(\sigma^{(k)})^{\gr})_{\tors}|}{|(Q^\fine_\W(\sigma^{(k-1)})^{\gr})_{\tors}|} = \mathrm{lat}(\sigma^{(k-1)})  \cdot \frac{\prod_{V \in \tropC_k \smallsetminus \tropC_{k-1} \text{ constrained}} d_V}{\mathrm{lat}(\sigma_k)}
    $$  
\end{proposition}

\begin{proof}
    Note that none of the vertices in $\cV$ (see Definition \ref{def: cV}) can be constrained. We thus distinguish three cases:

    \medskip
    \noindent \textbf{(Case 1.)} \emph{The vertex $V_k$ is unconstrained and satisfies one of the following conditions: either it does not coincide with any of the vertices $V_{i,j}$ in Definition \ref{def: filtration}, or it is the unique vertex of $\mathcal{V}_i'$ contained in $\tropC^{(k)}$.}

    Then,
    \begin{align*}
        Q^{\fine}_W(\sigma^{(k)})^{\gr} 
        &\simeq \frac{Q^{\fine}_W(\sigma^{(k-1)})^{\gr} \times M_X \times \mathbb{Z}_{e_k}}{((\ldots, 0,-\iota_{i_k}(m),0,\dots ), m, u_{\vec{e_k}}(m))} \\
        &\simeq Q^{\fine}_W(\sigma^{(k-1)})^{\gr} \times \mathbb{Z}_{e_k}
    \end{align*}
    where $\iota_{i_k}$ denotes the dual of the inclusion map $N_{\vartheta_{V_{i_k}}} \subseteq N_X$, the second isomorphism is that in Lemma \ref{lem: unconstrained} and $\vec{e_k}: V_k \to V_{i_k}$.

    Similarly, 
    $$  
        Q^{\fine}_W((\sigma')^{(k)})^{\gr} \simeq Q^{\fine}_W((\sigma')^{(k-1)})^{\gr} \times \mathbb{Z}_{e'_k}
    $$
    The edge $e_k'$ contains the edge $e_k$. Under these identifications, the map $\Phi_k$ in Equation \eqref{eqn: kth map of fine monoids} is given by the block matrix 
    $$\begin{bmatrix}
        \Phi_{k-1} & \star \\
        0 & 1\\
    \end{bmatrix}$$
    In particular, both left-hand side and right-hand side of the equality remain unchanged.

    \medskip
    \noindent \textbf{(Case 2.)} \emph{The vertex $V_k$ is unconstrained, coincides with one of the vertices $V_{i,j}$ in Definition \ref{def: filtration}, and is not the unique vertex belonging to both $\tropC_k$ and $\cV_i$.}

    Let $V_{i,h}$ be another vertex in $\cV_i \cap \tropC_k$. In this case, we have a natural inclusion $\tau_k \subseteq \tau_{k-1} \times \Z$, giving rise to a homomorphism of groups
    $$
    N_{\tau_k} \to N_{\tau_{k-1}} \times \Z
    $$
    Since the distance between $V_{i,j}$ and $\mathsf{C}_0$ is equal to the distance between $V_{i,h}$ and $\tropC_0$ is required to be the same, the composition $ N_{\tau_k} \to N_{\tau_{k-1}} \times \Z \to N_{\tau_{k-1}}$ is an isomorphism of lattices and in fact gives an isomorphism of cones between $\tau_k$ and $\tau_{k-1}$. In particular, we also have an epimomorphism
    $$
    p_k: M_{\tau_{k-1}} \times \Z \to M_{\tau_k}
    $$
    inducing isomorphisms
    $$
    (\tau_{k}^\vee)^{\gr} = M_{\tau_k} \simeq M_{\tau_k} \times \{0\}  \subseteq  M_{\tau_{k-1}} \times \Z \to M_{\tau_k} \simeq (\tau_{k-1}^\vee)^{\gr}
    $$

    Consider following cartesian diagram 

    \begin{equation}\label{eqn: cartesian diagram constrained vertex}
\begin{tikzcd}
\frac{Q^\fine_\W(\sigma^{(k-1)}))^{\gr} \times M_X \times \Z_{e_{k}}}{{((\ldots, 0,-\iota_{i_k}(m),0,\dots ), m, u_{\vec{e_k}}(m))}}  \arrow[r] \arrow[d] & Q^\fine_\W(\sigma^{(k-1)})^{\gr} \times \Z_{e_k} \arrow[d] \\
Q_\W^\fine(\sigma^{(k)})^{\gr} \arrow[r] & \frac{\prod_{V \in V(\tropC^{(k-1)})} M_{\vartheta_V} \times M_{\tau_{k}}}{\widetilde{R_{k-1}}}
\end{tikzcd}
\end{equation}
where the top map is the isomophism in Lemma \ref{lem: unconstrained}, the vertical maps are induced by $p_k$, the group $\widetilde{R_{k-1}}$ is the image of $R_{k-1}$ under the map 
$$
\bigg( \prod_{V \in V(\tropC_{k-1})} M_{\vartheta_V} \bigg) \times M_{\tau_{k-1}} \times \Z \to \bigg( \prod_{V \in V(\tropC_{k-1})} M_{\vartheta_V} \bigg) \times M_{\tau_{k}}
$$
and the bottom horizontal arrow is an isomorphism.

Under the isomorphism $M_{\tau_k} \simeq M_{\tau_{k-1}}$, the bottom left group in Diagram \eqref{eqn: cartesian diagram constrained vertex} is identified with $Q_\W^\fine(\sigma^{(k-1)})$; that is the composition
$$
Q_\W^\fine(\sigma^{(k-1)})^{\gr} \times \{ 0\} \subseteq Q_\W^\fine(\sigma^{(k-1)})^{\gr} \times \Z_{e_k} \to \frac{\prod_{V \in V(\tropC^{(k-1)})} M_{\vartheta_V} \times M_{\tau_{k}}}{\widetilde{R_{k-1}}}
$$
is an isomorphism.

A cartesian diagram similar to that in Equation \eqref{eqn: cartesian diagram constrained vertex} yields an isomorphism $Q_\W^\fine((\sigma')^{(k)}) \simeq Q_\W^\fine((\sigma')^{(k-1)})$. We claim that the resulting diagram
\begin{equation*}
\begin{tikzcd}
Q_\W^\fine((\sigma')^{(k)})^{\gr} \arrow[r, "\sim"] \arrow[d, "\Phi_k"] & Q_\W^\fine((\sigma')^{(k-1)})^{\gr} \arrow[d, "\Phi_{k-1}"] \\
Q_\W^\fine(\sigma^{(k)})^{\gr} \arrow[r, "\sim"'] & Q_\W^\fine(\sigma^{(k-1)})^{\gr}
\end{tikzcd}
\end{equation*}

is commutative, from which the statement would follow. Using the fact that the composition
$$
Q_\W^\fine((\sigma')^{(k-1)})^{\gr} \times \{ 0 \} \times \{0 \} \longrightarrow \frac{Q^\fine_\W((\sigma')^{(k-1)})^{\gr} \times M_X \times \Z_{e_{k}}}{ ((\ldots, 0,-m,0,\dots ), m, u_{\vec{e'_k}}(m)) } \xrightarrow{p_k} Q_\W^\fine((\sigma')^{(k-1)})^{\gr}
$$
is the inverse of $Q_\W^\fine((\sigma')^{(k)})^{\gr} \xrightarrow{\sim} Q_\W^\fine((\sigma')^{(k-1)})^{\gr}$, it is easy to check that 
$$
Q_\W^\fine((\sigma')^{(k-1)})^{\gr} \xrightarrow{\sim} Q_\W^\fine((\sigma')^{(k-1)})^{\gr} \xrightarrow{\Phi_k} Q_\W^\fine(\sigma^{(k)})^{\gr}  \to Q_\W^\fine(\sigma^{(k-1)})^{\gr} 
$$
is precisely the map $\Phi_{k-1}$ and this conclude Case $2$.
    
    \medskip
    \noindent \textbf{(Case 3.)} \emph{The vertex $V_k$ is constrained.}

    Let $\vartheta:= \vartheta_k$ be the codimension $1$ cone in $\Sigma(X)$ in which $V_k$ is constrained. Choose coordinates on $N_X$ such that the inclusion $N_\vartheta \subseteq N_X$ is identified with $\Z^{r-1} \times 0 \subseteq \Z^r$. Orient the edge $e_k$ so that $\vec{e_k}: V_k \to V_{i_k}$, and write $u_{\vec{e_k}}=(a_1,\ldots,a_{r-1},d_{V_k}) \in \Z^r \simeq N_X$.
    Then we have:

    $$
    Q_\W^\fine(\sigma^{(k)})^{\gr} \simeq \frac{Q^{\fine}_W(\sigma^{(k-1)})^{\gr} \times M_\vartheta \times \mathbb{Z}}{((\ldots, 0,-\iota_{i_k}(m),0,\dots ), (m_1 , \ldots, m_{r-1}), \sum_{j=1}^{r-1}a_j m_j+d_{V_k} m_r)}
    $$
    while 
    $$
    Q_\W^\fine((\sigma')^{(k)})^{\gr} = Q_\W^\fine((\sigma')^{(k-1)})^{\gr}
    $$
    By Lemma \ref{lem: constrained}, $|(Q_\W^\fine(\sigma^{(k-1)})^{\gr})_{\tors}|$ divides $|(Q_\W^\fine(\sigma^{(k)})^{\gr})_{\tors}|$. Let $g$ denote the quotient of this division; then the induced map 
    $$
    Q_\W^\fine(\sigma^{(k-1)})^{\gr} / (Q_\W^\fine(\sigma^{(k-1)})^{\gr})_{\tors} \to Q_\W^\fine(\sigma^{(k)})^{\gr} / (Q_\W^\fine(\sigma^{(k)})^{\gr})_{\tors}
    $$
    has index $d_{V:k} /g$. The assertion follows also in this case.
\end{proof}

\begin{proof}[Proof of Theorem \ref{thm: relation indices, loop aut}]
The proof follows immediately by induction on $k$ using Propositions \ref{lem:loop}, \ref{prop: cycle induction} and \ref{prop: inductive step}.
    
\end{proof}

\bibliographystyle{amsalpha}
\bibliography{library}

$\,$\
\noindent

$\,$\
\noindent
\textsc{Department of Pure Mathematics {\it \&} Mathematical Statistics, 
University of Cambridge, Cambridge, UK}

\textit{e-mail address:} \href{mailto:ac2758@cam.ac.uk}{ac2758@cam.ac.uk}

$\,$\
\noindent

$\,$\
\noindent
\textsc{Department of Pure Mathematics {\it \&} Mathematical Statistics, 
University of Cambridge, Cambridge, UK}

\textit{e-mail address:} \href{mailto:sk2050@cam.ac.uk}{sk2050@cam.ac.uk}

\end{document}